\documentclass[twoside,10pt, reqno]{amsart}

\usepackage[foot]{amsaddr}

\makeatletter
\def\subsubsection{\@startsection{subsubsection}{3}%
\z@{.5\linespacing\@plus.7\linespacing}{-.5em}%
{\normalfont\bfseries}}
\makeatother

\usepackage[numbers]{natbib}
\usepackage{amsfonts}
\usepackage{amsmath}
\usepackage{amsthm}
\usepackage{amssymb}
\usepackage{color}
\usepackage{graphicx}
\usepackage{hyperref}
\hypersetup{
    colorlinks=true, 
    linktoc=all,     
    linkcolor=blue,  
}

\usepackage{comment}
\usepackage[top=0.75in, bottom=0.75in, left=0.75in, right=0.75in]{geometry}

\usepackage{algorithm}
\usepackage{algorithmic}

\usepackage[shortlabels]{enumitem}
\usepackage{mathrsfs}  
\usepackage{mathtools}
\usepackage{xcolor}

\usepackage{enumitem}

\usepackage{multirow}

\usepackage{dsfont}

\numberwithin{equation}{section}

\theoremstyle{plain}

\newtheorem{theorem}{Theorem}[section]
\newtheorem{lemma}{Lemma}[section]
\newtheorem{prop}{Proposition}[section]
 
\newtheorem{assumption}{Assumption}[section]

\theoremstyle{definition}

\newtheorem{definition}[theorem]{Definition}

\theoremstyle{remark}
\newtheorem{remark}[theorem]{Remark}

\newtheorem{example}{Example}[section]

\numberwithin{table}{section}

\newcommand{\diag}{\operatorname{diag}}

\newcommand{\vep}{\varepsilon}

\newcommand{\E}{\mathbb{E}}
\newcommand{\R}{\mathbb{R}}

\newcommand{\spn}{\operatorname{span}}

\usepackage{bbm}

\def\a{\alpha}
\def\cA{\mathcal{A}}

\def\cE{\mathcal{E}}
\def\cF{\mathcal{F}}
\def\cG{\mathcal{G}}
\def\cH{\mathcal{H}}
\def\cI{\mathcal{I}}

\def\cL{\mathcal{L}}

\def\cO{\mathcal{O}}
\def\cP{\mathcal{P}}

\def\cS{\mathcal{S}}

\def\d{{\mathrm{d}}}

\def\mx{{\mathbbm{x}}}

\def\ba{{\textbf{\textit{a}}}}

\def\bu{{\textbf{\textit{u}}}}

\def\bz{{\textbf{\textit{z}}}}

\def\bH{{\textbf{H}}}
\def\bG{{\textbf{G}}}

\def\bX{{\textbf{X}}}
\def\bY{{\textbf{Y}}}
\def\bZ{{\textbf{Z}}}

\def\sx{\mathbbm{x}}

\def\sE{{\mathbb{E}}}
\def\sF{{\mathbb{F}}}
\def\sG{{\mathbb{G}}}
\def\sH{{\mathbb{H}}}
\def\sI{{\mathbb{I}}}

\def\sN{{\mathbb{N}}}
\def\sP{\mathbb{P}}

\def\sR{{\mathbb R}}
\def\sS{{\mathbb{S}}}
\def\sX{{\mathbb{X}}}

\newcommand{\tr}{\textnormal{tr}}
\DeclareMathAlphabet{\mymathbb}{U}{bbold}{m}{n} 

\def \Ck2{\begin{pmatrix} C_k & 0 \\ 0 & C_k \end{pmatrix}}

\def\B2{{\begin{pmatrix}
       \Tilde{B} & 0 \\   0 &  \Tilde{B}  \end{pmatrix}}}

\def\A2{{\begin{pmatrix}
       A & 0 \\   0 &  A  \end{pmatrix}}}

\newcommand{\vcat}{\mathsf{vcat}}

\title[$\alpha$-potential game]{An $\alpha$-potential game framework for $N$-player dynamic games}

\author{Xin Guo$^\ast$}
\email{xinguo@berkeley.edu}
\author{Xinyu Li$^\ast$}
\email{xinyu\_li@berkeley.edu}
\author{Yufei Zhang$^\dagger$}
\email{yufei.zhang@imperial.ac.uk}

\address{$^\ast$ {Department of Industrial Engineering and Operations Research, UC Berkeley,  CA, USA}}
\address{$^\dagger${Department of Mathematics, Imperial College London, London, UK}}



\begin{document}

\maketitle 
\begin{abstract}
This paper proposes and studies a general form of dynamic $N$-player non-cooperative games called $\alpha$-potential games, where the change of a player's { objective function} upon her unilateral deviation from her strategy is equal to the change of an $\alpha$-potential function up to an error $\alpha$. Analogous to the static potential game (which corresponds to $\alpha=0$),  the $\alpha$-potential game framework is shown to reduce the challenging task of finding $\alpha$-Nash equilibria for a dynamic game to minimizing the $\alpha$-potential  function. Moreover, an analytical characterization of $\alpha$-potential functions is established, with $\alpha$ represented in terms of the magnitude of the asymmetry of { objective function}s' second-order derivatives. For stochastic differential games in which the state dynamic is a controlled diffusion, $\alpha$ is characterized  in terms of the number of players, the choice of admissible strategies, and the intensity of interactions and the level of heterogeneity among players. Two  classes of stochastic differential games, namely  distributed games and games with mean field interactions, are analyzed  to highlight the  dependence of $\alpha$ on general game characteristics that are beyond the mean-field paradigm, which focuses on the limit of $N$ with homogeneous players. To analyze the $\alpha$-NE,  the associated optimization problem is embedded into a conditional McKean–Vlasov control problem. A verification theorem is established to construct $\alpha$-NE based on solutions to an infinite-dimensional Hamilton-Jacobi-Bellman equation, which is reduced to a system of ordinary differential equations for linear-quadratic games.

\end{abstract}


\section{Introduction}
\subsection{Overview}
Static potential games, introduced by Monderer and Shapley in \cite{monderer1996potential}, are non-cooperative games where any player's change in utility function upon unilaterally deviating from her policy can be evaluated through the change of an auxiliary function called potential function. 
The introduction of the potential function is powerful as it simplifies the otherwise challenging task of finding Nash equilibria in $N$-player non-cooperative games to optimizing a single function. Static potential games and their variants have been a popular framework for studying $N$-player static games, especially with heterogeneous players.

{ 

In the dynamic setting with Markovian state transitions and Markov policies, direct generalization of the static potential game called Markov potential game is proposed in \cite{macua2018learning}.
Unfortunately, most dynamic games are not Markov potential games. In fact, \cite{leonardos2021global}  shows that even a Markov game where the game at each state is a static potential game may not be a Markov potential game. In practice, Markov potential game framework imposes restrictive assumptions for various applied problems,  such as state transitions being of distributed types for multi-agent robotics \cite{kavuncu2021potential, sun2023distributed, sun2024imagined} and  instantaneous reward functions being separable for resource allocation \cite{narasimha2022multi}. 
}




{

Recently, a more general form of dynamic game called Markov $\alpha$-potential game is proposed by \cite{guo2023markov} for $N$-player non-cooperative  Markov games with finite-state, finite-action, and discrete-time state transition.  The introduction of a parameter $\alpha$ and an associated $\alpha$-potential function enables capturing the interactions of players and their heterogeneity.  They establish the existence of $\alpha$-potential function for discrete-time Markov games, and show that maximizing the $\alpha$-potential function yields an $\alpha$-Nash equilibrium (NE). Meanwhile, they identify several important classes of  dynamic $\alpha$-potential games. These present new potential applications, in addition to  various potential games explored earlier in transportation systems \cite{zhou2021power}, power networks \cite{ibars2010distributed}, and multi-agent robotics \cite{kavuncu2021potential, sun2023distributed, sun2024imagined}, along with more recent studies  \cite{leonardos2021global, mao2021decentralized,   song2021can,  zhang2021gradient, ding2022independent, fox2022independent, maheshwari2022independent,   narasimha2022multi, guo2024towards}. 


In this paper, we propose and study general dynamic $\alpha$-potential 
games, including stochastic differential games with continuous state-action space, and with continuous-time state transition. Similar to the $\alpha$-potential game in the discrete-time setting in \cite{guo2023markov},  this general $\alpha$-potential game framework reduces the challenging task of finding approximate NE in a dynamic game to a (simpler) optimization problem of minimizing a single function.

In the framework of $\alpha$-potential games, there are two key mathematical questions: finding and optimizing  the $\alpha$-potential function, and analyzing the magnitude of $\alpha$. In the discrete-time setting with finite state and finite action, these two questions have been answered in 
 \cite{guo2023markov} by formulating a semi-infinite linear programming (SLP) problem such that its optimal solution is the $\alpha$-potential function and its  minimum yields the   $\alpha$. However, this  SLP approach does not  apply to continuous-time  and arbitrary state-action spaces. 
 
Instead,  in this paper
we  adopt the tool of linear derivatives developed in \cite{guo2024towards} to construct the $\alpha$-potential function $\Phi$, and to characterize $\alpha$ in terms of the magnitude of the asymmetry of { objective function}s' second-order derivatives.
For stochastic differential games where the state dynamic is a controlled diffusion, the $\alpha$-potential function is expressed via the sensitivity processes of the controlled diffusion, and $\alpha$ is explicitly characterized in terms of the game structure including the number of players, 
the choice of strategy classes,
and the intensity of interactions and the level of heterogeneity among players. 
To analyze the $\alpha$-NE, our approach is to show that  minimizing \(\Phi\) is  equivalent to solving a conditional McKean-Vlasov control problem:  we first develop   the dynamic programming principle (DPP), and  then establish a verification theorem to construct
a minimizer of the $\alpha$-potential function $\Phi$
based on solutions to an infinite-dimensional   Hamilton-Jacobi-Bellman (HJB) equation.  To the best of our knowledge, this is the first result establishing a DPP for dynamic potential games.  
Prior to our work,  the only known approach  for \(\alpha\)-NE is the policy-gradient algorithm in \cite{guo2023markov} for finite-state discrete-time $\alpha$-potential games.
Our approach is illustrated  through a linear-quadratic network  game, where the $\alpha$-NE and the associated HJB equation are explicitly solved.}

\subsection{Outline of main results}

\subsubsection*{$\alpha$-potential games and approximate Nash equilibria}

Consider a general $N$-player game $\cG$  characterized by $\cG= ({ [N]}, S, (\cA_i)_{i\in { [N]}}, (V_i)_{i\in { [N]}})$,\footnotemark
where ${ [N]}=\{1, \ldots, N\}$ is the set of players, $S$ is the state space of the underlying dynamics, $\cA_i$ is the set of admissible strategies of player $i$, and 
$V_i: \prod_{i\in { [N]}} \cA_i \rightarrow \mathbb{R}$ is the { total cost function} of player $i$, with $V_i(\ba)$ being player $i$'s expected accumulated cost if the state dynamics starts with a fixed initial state $s_0\in S$ and all players take the strategy profile $\ba$. For each $i\in { [N]}$, player $i$ aims to minimize her { objective function} $V_i$ over all admissible strategies in $\cA_i$.

\footnotetext{For notational simplicity, we do not write explicitly the dependence of 
$\cG$
on the fixed initial state $s_0$.}

Here we focus on a class of games called $\alpha$-potential games, where there exists $\alpha\ge 0$ and $\Phi:\cA^{(N)} \to \R$ such that for all $ i \in { [N]}$, $a_i, a_i^{\prime} \in \cA_i$ and $a_{-i} \in \cA_{-i}^{(N)}$,
\begin{equation}
\label{eq:alpha_v_phi_intro}
|V_i\left(\left(a_i^{\prime}, a_{-i}\right)\right)-V_i\left(\left(a_i, a_{-i}\right)\right) - \left( \Phi\left(\left(a_i^{\prime}, a_{-i}\right)\right)-\Phi\left(\left(a_i, a_{-i}\right)\right) \right) | \leq \alpha,   
\end{equation}
with $\cA^{(N)} = \prod_{i\in { [N]}} \cA_i$    the set of strategy profiles for all players, and $\cA_{-i}^{(N)}=\prod_{j \in { [N]} \backslash\{i\}} \cA_j$ the set of strategy profiles of all players except player $i$. 
Such $\Phi$ is called an $\alpha$-potential function for the game $\mathcal G$. In the case of $\alpha=0$, we simply call the game $\cG$ a potential game and $\Phi$ a potential function for $\cG$.

Equation \eqref{eq:alpha_v_phi_intro} relaxes the notion of potential games in \cite{monderer1996potential,macua2018learning} by introducing a   parameter $\alpha$. That is, a game $\cG$ is an $\alpha$-potential game if the change of a player's { objective function} upon her unilateral deviation from her strategy is equal to the change of the $\alpha$-potential function up to an error $\alpha$. This additional parameter $\alpha$ enables capturing important information regarding the interaction between players' state dynamics and strategies, beyond the number of players which has been the primary focus of approximate Nash equilibrium approach such as mean field games.

Similar to potential games,  an $\alpha$-potential game $\cG$ has an important property: any  minimizer of an $\alpha$-potential function of $\cG$ is an $\alpha$-NE of the game $\cG$  (Proposition \ref{prop:optimizer_and_NE}).
Proposition \ref{prop:optimizer_and_NE} suggests    three key components in applying the $\alpha$-potential game framework to analyze general non-cooperative games: 
constructing an 
$\alpha$-potential function, characterizing (upper bounds of) the associated parameter 
$\alpha$, and developing a solution technique for minimizing the 
$\alpha$-potential function over admissible strategy sets.
 

 \subsubsection*{Characterizing general $\alpha$-potential games}
We start by    constructing  the    $\alpha$-potential function and characterizing  the associated parameter $\alpha$ for a given game $\cG$, where all players' strategy classes are convex.
Specifically,  for each $i\in { [N]}$, denote by $\spn(\cA_i)$   the vector space of all linear combinations of strategies in $\cA_i$.
The concept of  linear derivative of $V_i$ with respect to $\cA_i$, introduced in \cite{guo2024towards} for arbitrary convex strategy classes,  enables us to  establish Theorem \ref{thm:potential_hessian}: if the { objective function}s of a game $\cG$ admit  second-order linear derivatives, then under some mild regularity conditions, 
for any fixed $\bz \in \cA^{(N)}$, the function     \begin{equation}\label{eq:Phi_intro}
        \Phi(\ba)\coloneqq \int_0^1 \sum_{j=1}^N \frac{\delta V_j}{\delta a_j}\left(\bz+r(\ba-\bz) ; a_j-z_j\right) \mathrm{d} r 
    \end{equation}
is an $\alpha$-potential function of $\cG$, with 
    \begin{equation}
    \label{eq:alpha_general_intro}
         \alpha \le  { 2} \sup_{i \in { [N]}, a'_i\in \cA_i, \ba ,\ba''   \in \cA^{(N)}}
  \sum_{j=1}^N\left| \frac{\delta^2 V_i}{\delta a_i \delta a_j} \left(\ba ; a_i^{\prime} , a''_j  \right) 
-\frac{\delta^2 V_j}{\delta a_j \delta a_i}\left(\ba ; a''_j , a_i^{\prime} \right)  
 \right|.
    \end{equation}
This characterization generalizes existing results of potential games with finite-dimensional strategy classes  
\cite{monderer1996potential, leonardos2021global, hosseinirad2023general} to
 general dynamic games with arbitrary convex strategy classes.  
In particular, it replaces the Fr\'{e}chet derivatives used in earlier works with linear derivatives,    without requiring a topological structure on $\cA^{(N)}$.  Moreover, it   quantifies the performance of the $\alpha$-potential function \eqref{eq:Phi_intro} in terms of the difference between the second-order linear derivatives of the { objective function}s.

\subsubsection*{Constructing $\alpha$-potential function for stochastic differential game}
   
The main contribution of this paper is to   develop the criteria  \eqref{eq:Phi_intro} 
and \eqref{eq:alpha_general_intro}   for stochastic differential games in which the state dynamic is a controlled diffusion. 
Specifically, let $T\in (0,\infty)$, let $(\Omega, \cF,    \sP)$  be  
a   probability space supporting   an $m$-dimensional Brownian motion $W=(W^k)_{k=1}^m$,
and 
let  $\sF $  be
 the  natural  filtration of   $W$.
Let $\cH^2(\sR^n)$ be the space of $\sR^n$-valued  
 square integrable $\sF$-adapted processes, and 
 for each $i\in { [N]}$, 
$\mathcal A_i$ be a convex subset  of $\cH^2(\sR^n)$
representing  player $i$'s admissible  controls. 
For each $\bu \in \cA^{(N)} $, let $\bX^{\bu}$ be the associated state process satisfying for all $i\in { [N]}$ and $t\in [0,T]$, 
\begin{equation}
\label{eq:state_intro} 
     \d X_{t,i}=  b_i(t,  
     {\bX}_{t }, \bu_{t})    \d t+\sum_{k=1}^m \sigma_{ik}(t,  {\bX}_{t }, \bu_{t}) \d {W}^k_t, 
    \quad X_{0,i}=x_i,
\end{equation}
where $x_i\in \sR$ is a given initial state, 
$  b_i :[0,T]\times  \sR^{Nd }\times \sR^{Nn }\to \sR^d $ 
and 
$\sigma_i=(\sigma_{i1},\ldots, \sigma_{im}) :[0,T]\times  \sR^{Nd }\times \sR^{Nn }\to \sR^{d\times m}$
  are given functions.  The { objective function} $V_i:\cA^{(N)}  \to \sR$ of player $i$ is 
\begin{align}
\label{eq:value_intro}
V_i(\bu)=\sE\left[\int_0^T 
f_i(t, \bX^{\bu}_t, \bu_t)\,\d t + g_i( \bX^{\bu}_T) \right],
\end{align}
where $f_i:[0,T]\times \sR^{Nd} \times \sR^{Nn} \to \sR$ and $g_i:\sR^{Nd} \to \sR$ are given functions. Precise assumptions on $x_i, {b}_i,\sigma_i, f_i$ and $g_i$ are given in Assumption \ref{assum:regularity_general}.

We characterize the linear derivative of $V_i$ and the function $\Phi$   in \eqref{eq:Phi_intro} through the sensitivity processes of the state process with respect to controls (Theorem \ref{thm:potential_diff}). 
In particular, 
assuming  $0\in \cA^{(N)} $, the  function $\Phi $   (with $\bz=0$) can be expressed as
\begin{align}
\label{eq:potential_differential_intro}
\begin{split}
\Phi(\bu)&  =\int_0^1      \sum_{i=1}^N  \mathbb{E}\Bigg[
    \int_0^T
    \begin{pmatrix}
        \bY^{ r \bu ,     u_{i}}_t \\
    u_{t,i} 
    \end{pmatrix}^\top 
    \begin{pmatrix}
        \partial_x f_i \\
    \partial_{u_i} f_i
    \end{pmatrix}
      (t, \bX_t^{r \bu }, r \bu_t   )\,
     \mathrm{d} t
     + (\partial_x g_i)^\top  (\bX_T^{ r \bu} ) \bY^{r\bu, u_{i}}_T\Bigg] \d r,
 \end{split}
\end{align}
 where 
for each $\bu\in\cA^{(N)}$
and $u'_i\in \cH^2(\sR^n)$,
the sensitivity process 
$\bY^{\bu,u'_i}$ 
  is the derivative  (in the $L^2$ sense) of the state $\bX^{\bu}$ when player $i$   varies her control in the direction $u'_i$, and satisfies a controlled linear stochastic differential equation    (as in  \eqref{eq:Y_h_i_general}). See Theorem \ref{thm:potential_diff}
  for the expression of $\Phi$ with general $\cA^{(N)}$. 

{

\subsubsection*{Quantifying  $\alpha$ for  stochastic differential game}

Using the   bound \eqref{eq:alpha_general_intro}, we then quantify the parameter $\alpha$
 for the  game \eqref{eq:state_intro}-\eqref{eq:value_intro}   based on the structure of the game. 
 A key technical step is to characterize and estimate 
 the second-order linear derivative of $V_i$   through   second-order sensitivity processes, representing the derivative of $\bY^{\bu,u'_i}$.
%
Under suitable structural assumptions on the   coefficients of \eqref{eq:state_intro}, we establish precise estimates for these sensitivity processes, 
and obtain  the following upper bound  (stated more precisely in Theorem \ref{thm:value_jacobian_open}):
  for all $i,j\in { [N]}$,
    \begin{equation}
    \label{eq:jacobian_gap_intro}
        \left|\frac{\delta^2 V_i}{\delta u_i \delta u_j}\left(\bu ; u_i^{\prime}, u_j^{\prime \prime}\right)-\frac{\delta^2 V_j}{\delta u_j \delta u_i}\left(\bu ; u_j^{\prime \prime}, u_i^{\prime}\right)\right|
        \leq C C_{f,g,N},
    \end{equation}
where $C\ge 0$ is a constant depending only on the state coefficients and time horizon, and $C_{f,g,N}$ is a constant depending explicitly on the number of players $N$ and the sup-norms of the partial derivatives of  $f_i-f_j$ and $g_i-g_j$. 

 This analysis of   $\alpha$ shows its  general dependence on game characteristics, {\it including possibly asymmetric and heterogeneous forms of  cost functions   and state dynamics}, and is not limited to the scale of $N$ as in the mean-field paradigm. To highlight this distinction, we specialize the above  bound \eqref{eq:jacobian_gap_intro} of $\alpha$ to two  classes of stochastic games: 
\begin{itemize}
  
 \item For distributed games where players only interact through their cost functions and not  the state and control processes, we prove that if a static potential function can be derived from the cost functions, then 
 dynamics games are potential games (with $\alpha=0$), regardless of the number of players  (Example  \ref{ex:dist_game_open_distributed}).  
   \item 
    For games with mean field type interactions,
if each player's dependence on others' states and actions is solely through her empirical measures, 
then  $\alpha$ is of the magnitude $\mathcal O(1/N)$ as $N\to \infty$
(Example \ref{ex:mean_field_open}).
Note that  
    $\alpha$ decays to zero as the number of players increases, 
 even with heterogeneity in state dynamics, in cost functions, and in admissible strategy classes. This is in contrast to the classical mean field games with homogeneous players;  see Remark  \ref{rmk:mean_field_interaction} for a more detailed comparison.

\end{itemize}

\subsubsection*{$\alpha$-NE via a  McKean-Vlasov control problem}

We further develop  a dynamic programming approach to minimize the
function $\Phi$ over $\cA^{(N)}$. The main difficulty is   that the objective \eqref{eq:potential_differential_intro} depends on the aggregated
behavior of the state and sensitivity processes with respect to $r\in [0,1]$, which acts as an additional  noise independent of the
Brownian motion $W$. Meanwhile, the admissible controls in $\cA^{(N)}$ are adapted to a smaller filtration $\sF$ that depends
only on $W$. 
To recover the  dynamic programming principle,
we embed the optimization problem into a    conditional McKean--Vlasov  control problem.
This is achieved by treating  $r$ in  \eqref{eq:potential_differential_intro} as a  uniform random variable  $\mathfrak r$   independent of $W$, 
 and expressing      
the objective $\Phi(\bu)$ 
in terms of $\bu$ and  the conditional law of  $(\bX^{\mathfrak{r}\bu},
\bY^{\mathfrak{r}\bu, u_1},
\ldots, 
\bY^{\mathfrak{r}\bu, u_N},\mathfrak{r})$ given $W$.
This approach allows  us to embed 
the minimization of  $\bu \mapsto\Phi(\bu)$    into a  control problem, where the   state space 
is a subset of the Wasserstein space of probability measures
(Proposition \ref{prop:law_invariance}).
Moreover, 
by   It\^o’s formula along a flow of conditional measures,  
we establish a verification theorem to construct 
a minimizer of $\Phi$
based on solutions to an infinite-dimensional   Hamilton-Jacobi-Bellman (HJB) equation (Theorem \ref{thm: verification}).

}

\subsubsection*{A toy example of dynamic games on   graph}

We illustrate our results  
through a simple  linear-quadratic  game   on a   undirected graph,
whose  the vertices represent players, and edges indicate dependencies between them.
We show that this game is an $\alpha$-potential game,
and characterize $\alpha$ explicitly 
in terms of $N$ the number of players, and $q_{ij}$ the strength and  the degree of heterogeneous interaction between players $i$ and $j$ (see Section \ref{sec:alpha_crowd}).
We further construct
an $\alpha$-NE of the game analytically 
through a system of ordinary differential equations (Theorem \ref{thm:lq_verification}). This is accomplished by solving the associated HJB equation for the $\alpha$-potential function \eqref{eq:potential_differential_intro} by utilizing the game’s linear-quadratic structure.

The asymptotic limit derived from our  analysis allows for general  asymmetric interactions and heterogeneity among players, in contrast to  existing works on the mean field approximation for both differential games (e.g.,  \cite{lasry2006jeux, lasry2007mean,carmona2018probabilistic}) and games on graphs (e.g., \cite{gao2020linear, lacker2022label, 
bayraktar2023propagation}). 
 It shows   that the  $\alpha$-potential game framework enables differentiating game characteristics and player interactions which are hard to quantify under previous more restrictive game frameworks,
such as Markov potential games   \cite{macua2018learning,
leonardos2021global, 
mao2021decentralized, song2021can, ding2022independent}
and near potential games
\cite{candogan2011flows, candogan2013near,  
varga2022limited,
varga2023upper}.

\section{Analytical framework for general $\alpha$-potential games}

\label{sec:abstract}

\subsection{$\alpha$-Potential games and  approximate Nash equilibria}
This section introduces the mathematical framework for $\alpha$-potential games,  starting by some basic notions for the game and associated strategies.

Consider a game $\cG =({ [N]},   S, (\cA_i)_{i\in { [N]}}, (V_i)_{i\in { [N]}})$ defined as follows:   
 ${ [N]}=\{1, \ldots, N\}$, $N \in \mathbb{N}$, is a finite set of players,
 $S$ is a set  representing the state space of the underlying dynamics, 
$\cA_i$ is a subset of a real vector space representing all admissible strategies of player $i$,
and $\cA^{(N)} = \prod_{i\in { [N]}} \cA_i$ is the set of strategy profiles for all players. 
For each $i\in { [N]}$, 
 $V_i: \cA^{(N)} \rightarrow \mathbb{R}$ is the { objective function} of player $i$, 
 where $V_i(\ba)$ is player $i$'s expected cost if the state dynamics starts with a fixed initial state $s_0\in  S$ and all players take the strategy profile $\ba\in \cA^{(N)}$. 
For any $i\in { [N]}$, player $i$ aims to minimize the { objective function} $V_i$ over all admissible strategies in $\cA_i$.
We denote by $\cA_{-i}^{(N)}=\prod_{j \in { [N]} \backslash\{i\}} \cA_j$ the set of strategy profiles of all players except player $i$,
 and by $\ba $ and $a_{-i} $ a generic element of $\cA^{(N)}$ and  $\cA_{-i}^{(N)}$, respectively.

Note that this game framework includes static games, and discrete-time and continuous-time dynamic games. Moreover, depending on the precise definitions of strategy classes, this framework also accommodates stochastic differential games with either open-loop or closed-loop controls.
The focus of this paper is on a class of games $\mathcal G$ called $\alpha$-potential games, defined as follows.

\begin{definition}[$\alpha$-potential game]
\label{def:a_potential_game}
Given a game $\cG =({ [N]},  S, (\cA_i)_{i\in { [N]}}, (V_i)_{i\in { [N]}})$, if there exists $\alpha\ge 0$ and   $\Phi:\cA^{(N)} \to \R$ such that 
for all $ i \in { [N]}$, $a_i, a_i^{\prime} \in \cA_i$ and $a_{-i} \in \cA_{-i}^{(N)}$,
\begin{equation}
\label{eq:alpha_v_phi}
|V_i\left(\left(a_i^{\prime}, a_{-i}\right)\right)-V_i\left(\left(a_i, a_{-i}\right)\right) - \left( \Phi\left(\left(a_i^{\prime}, a_{-i}\right)\right)-\Phi\left(\left(a_i, a_{-i}\right)\right) \right) | \leq \alpha,    
\end{equation}
then we say $\cG$ is an $\alpha$-potential game, and $\Phi$ is an $\alpha$-potential function for $\mathcal G$. In the case where $\alpha=0$, we simply call the game $\cG$ a potential game and $\Phi$ a potential function for $\cG$.
\end{definition}

Intuitively, a game $\cG$ is an $\alpha$-potential game if there exists an $\alpha$-potential function such that whenever one player unilaterally deviates from her strategy, the change of that player's { objective function} is equal to the change of the $\alpha$-potential function up to an error $\alpha$.
This definition generalizes the notion of potential games in \cite{monderer1996potential} by allowing for a positive $\alpha$. Such a relaxation is essential for dynamic games, as many dynamic games that are not potential games are, in fact, $\alpha$-potential games for some $\alpha>0$; see \cite{guo2023markov} and also Sections \ref{sec:open_loop}. 
Indeed, it is clear that if $\hat \alpha \coloneqq  \sup_{i\in { [N]}, \ba\in \cA^{(N)}}|V_i(\ba)|<\infty$, then $\cG $ is a $2\hat \alpha$-potential game   and a $2\hat \alpha$-potential function $\Phi=0$.

For a given game $\cG$, there can be multiple parameters $\alpha$ satisfying the condition \eqref{eq:alpha_v_phi}.
In \cite{guo2023markov}, an $\alpha$-potential game is defined with the optimal $\alpha$ determined by
\begin{equation}
\label{eq:optimal_alpha}
\alpha^* = \inf_{\Phi\in\mathscr F} \sup_{i\in { [N]}, a_i, a_i^{\prime} \in \cA_i, a_{-i} \in \cA_{-i}^{(N)}} 
	|V_i\left(\left(a_i^{\prime}, a_{-i}\right)\right)-V_i\left(\left(a_i, a_{-i}\right)\right) - \left( \Phi\left(\left(a_i^{\prime}, a_{-i}\right)\right)-\Phi\left(\left(a_i, a_{-i}\right)\right) \right) |,
\end{equation}
where $\mathscr F$ contains suitable functions $\Phi:\cA^{(N)}  \to \R$.
For discrete games with finite states and actions, \cite{guo2023markov} shows that selecting $\mathscr F$ as the set of uniformly equi-continuous functions on Markov policies ensures a well-defined $\alpha^*$ and also the existence of an $\alpha^*$-potential function within $\mathscr F$.

However, in continuous-time games with continuous state and action spaces, computing the optimal $\alpha^*$ in \eqref{eq:optimal_alpha} is challenging, and selecting a suitable set $\mathscr F$ for the existence of an $\alpha^*$-potential function remains unclear.
Most critically, as shown in \cite{guo2023markov}, having an appropriate upper bound $\alpha$ of $\alpha^*$ and an associated $\alpha$-potential function $\Phi$ is sufficient for the key analysis.
Therefore, we adopt Definition \ref{def:a_potential_game}, which frees us to focus on characterizing some upper bound of $\alpha^*$ in terms of the number of players, the set of admissible strategies, and the game structure.

For an $\alpha$-potential game, computing an approximate Nash equilibrium reduces to an optimization problem. To see it, we first recall the solution concept of $\varepsilon$-Nash equilibrium.

\begin{definition}\label{def: NE}
 For any $\varepsilon\ge 0$, a strategy profile $\ba = (a_i )_{i\in { [N]}} \in \cA^{(N)}$ is an $\varepsilon$-Nash equilibrium of the game $\mathcal{G}$ if
$ V_i\left(\left(a_i, a_{-i}\right)\right) \leq V_i\left(\left(a_i^{\prime}, a_{-i}\right)\right) + \varepsilon$, for any  $i \in { [N]},  a_i^{\prime} \in \cA_i .$
\end{definition}

Definition \ref{def: NE} provides a unified definition of approximate Nash equilibrium for a general game $\cG$. When $\cG$ is a stochastic differential game and the set $\cA^{(N)}$ of admissible strategy profiles contains the set of open-loop controls or closed-loop controls, Definition \ref{def: NE} is consistent with the concepts of open-loop Nash equilibrium or closed-loop Nash equilibrium described in \cite[Chapter 2]{carmona2018probabilistic}. 

The following proposition shows that an approximate Nash equilibrium of an $\alpha$-potential game can be obtained by optimizing its corresponding $\alpha$-potential function. This is analogous to static potential games with potential functions.
The proof follows directly from Definitions \ref{def:a_potential_game} and \ref{def: NE} and hence is omitted.

\begin{prop}\label{prop:optimizer_and_NE}
Let $\cG$ be an $\alpha$-potential game for some $\alpha$ and  $\Phi$ be an $\alpha$-potential function.
For each $\varepsilon\ge 0$, if there exists $\overline \ba  \in \cA^{(N)} $ such that $\Phi(\overline \ba )\le \inf_{\ba \in \cA^{(N)} }\Phi(\ba)+\varepsilon$, then $\overline \ba  $ is an $( \alpha+\varepsilon) $-Nash equilibrium of $\mathcal{G}$.
\end{prop}

\subsection{Characterization of $\alpha$-potential games via linear derivatives}

Proposition \ref{prop:optimizer_and_NE} highlights the importance of explicitly characterizing an $\alpha$-potential function for a given game and the parameter $\alpha$.

For the special class of potential games (i.e., $\alpha=0$) with  finite-dimensional strategy class \cite{monderer1996potential, leonardos2021global, hosseinirad2023general},  it is well known that a game is a potential game if the { objective function}s are twice continuously (Fr\'{e}chet) differentiable in policy parameters and have symmetric second-order derivatives. 
More precisely, 
        consider a game $\cG =({ [N]}, (\cA_i)_{i\in { [N]}}, (V_i)_{i\in { [N]}})$
        where for all $i\in { [N]}$,
        $\cA_i$  is an interval.
Suppose that for all $i\in { [N]}$, $V_i:\cA^{(N)}\to \sR$ is twice continuously differentiable.
Then by \cite[Theorem 4.5]{monderer1996potential},   
$\cG$ is a potential game  if and only if $ {\partial^2_{a_ia_j} V_i} = {\partial^2_{a_ja_i} V_j} $ for all $i,j\in { [N]}$,  { and a form of potential function is given.}

In this section, we will { provide an analytical framework} to construct the parameter $\alpha$ and the associated $\alpha$-potential functions based on linear derivatives of the { objective function}s with respect to strategies as introduced in 
\cite{guo2024towards}.
Let us start by recalling the linear derivative of a scalar-valued function with respect to unilateral deviations of strategies. For each $i\in { [N]}$, we denote by $\spn(\cA_i)$ the vector space of all linear combinations of strategies in $\cA_i$, i.e., 
{$ 
    \spn(\cA_i) = \left\{\sum_{\ell =1}^m c_\ell a^{(\ell)}_i \mid c_\ell \in \R, a^{(\ell)}_i \in \cA_i, \text{ for any } l = 1,2,\cdots, m, \text{ and }  m\in \sN\right\}.
$ 
}

\begin{definition} 
\label{def:linear_derivative}
Let $\cA^{(N)}=\prod_{i \in { [N]}} \cA_i$ be a convex set and $f: \cA^{(N)} \rightarrow \mathbb{R}$. For each $i \in { [N]}$, we say $f$ has a linear derivative with respect to $\cA_i$, if 
there exists $\frac{\delta f}{\delta a_i}: \cA^{(N)} \times \spn\left(\cA_i\right) \rightarrow \mathbb{R}$, such that for all $\ba=\left(a_i, a_{-i}\right) \in \cA^{(N)}, \frac{\delta f}{\delta a_i}(\ba ; \cdot)$ is linear and
\begin{equation}
\label{eq:linear_derivative_condition}
\lim _{\varepsilon \searrow 0} \frac{f\left(\left(a_i+\varepsilon\left(a_i^{\prime}-a_i\right), a_{-i}\right)\right)-f(\ba)}{\varepsilon}=\frac{\delta f}{\delta a_i}\left(\ba ; a_i^{\prime}-a_i\right), \quad \forall a_i^{\prime} \in \cA_i.
\end{equation}
 
Moreover, for each $i, j \in { [N]}$, we say $f$ has second-order linear derivatives with respect to $\cA_i \times \cA_j$, if (i) for all $k \in\{i, j\}, f$ has a linear derivative $\frac{\delta f}{\delta a_k}$ with respect to $\cA_k$, and (ii) for all $(k, \ell) \in\{(i, j),(j, i)\}$, there exists $\frac{\delta^2 f}{\delta a_k \delta a_{\ell}}: \cA^{(N)} \times \operatorname{span}\left(\cA_k\right) \times \operatorname{span}\left(\cA_{\ell}\right) \rightarrow \mathbb{R}$ such that for all $\ba \in \cA^{(N)}, \frac{\delta^2 f}{\delta a_k \delta a_{\ell}}(\ba, \cdot, \cdot)$ is bilinear and for all $a_k^{\prime} \in \operatorname{span}\left(\cA_k\right), \frac{\delta^2 f}{\delta a_k \delta a_{\ell}}\left(\cdot ; a_k^{\prime}, \cdot\right)$ is a linear derivative of $\frac{\delta f}{\delta a_k}\left(\cdot ; a_k^{\prime}\right)$ with respect to $\cA_{\ell}$. We refer to $\frac{\delta^2 f}{\delta a_i \delta a_j}$ and $\frac{\delta^2 f}{\delta a_j \delta a_i}$ as second-order linear derivatives of $f$ with respect to $\cA_i \times \cA_j$.
\end{definition}

{
\begin{remark}\label{remark-derivatives}

Linear differentiability, as defined in Definition \ref{def:linear_derivative}, is weaker than 
Fr\'echet/G\^{a}teaux differentiability,
as it avoids  introducing
a topology on the strategy classes
 $\mathcal{A}_i$. 

Recall that 
a  function $f: O \subset X \rightarrow \mathbb{R}$ defined on an \emph{open subset} $O$ of a locally convex topological vector space $X$  is G\^{a}teaux differentiable if, for all $u  \in V$,   $Df(u;v)=\lim_{\varepsilon\to 0}\frac{f(u+\varepsilon v)-f(u)}{\varepsilon } $ exists   for all $v \in V$.
If in addition $(X,\|\cdot\|_X)$ is a normed vector space, $X\ni v\mapsto Df(u;v)\in \sR$ is a bounded linear operator,
and 
$\lim_{\|v\|_X\to 0}\frac{|f(u+ v)-f(u)-Df(u;v)|}{\|v\|_X  }=0$, then $f$ is Fr\'echet differentiable. 
Note that both Fr\'echet and G\^ateaux derivatives are defined only in the interior of a set $O$, as their definitions require that $  u + \varepsilon v $  remains within the domain $O$   for all sufficiently small $ \varepsilon $. This necessitates a topological structure on $O$.

Definition \ref{def:linear_derivative} defines  derivatives using convex combinations within the strategy class, without the need of a topology. Therefore, it can be applied to analyze games with any convex strategy class. 
Moreover, 
  if $f$ has  a 
G\^{a}teaux derivative $Df$ with respect to $\cA_i$,
then $f$ also has a linear derivative given by   
$\frac{\delta f}{\delta a_i}(\ba; a'_i-a_i)= Df(\ba; a'_i-a_i)$.
 
\end{remark}

}

 
Note that Definition \ref{def:linear_derivative} generalizes the notion of linear derivative for functions of Markov policies introduced in \cite{guo2024towards} to functions defined on arbitrary convex strategy classes. It enables us to construct an $\alpha$-potential function for a game $\cG$ using the linear derivative of its { objective function}s, with $\alpha$ bounded by the difference between the second-order linear derivatives of the { objective function}s.

\begin{theorem}\label{thm:potential_hessian}
    Let $\cG$ be a game whose set of strategy profiles $\cA^{(N)}$ is convex. Suppose that for all $i, j \in { [N]}$, the { objective function} $V_i$ has second-order linear derivatives with respect to $\cA_i \times \cA_j$ such that for all $\bz=\left(z_j\right)_{j \in { [N]}}\in \cA^{(N)}$, $\ba=\left(a_j\right)_{j \in { [N]}} \in \cA^{(N)}$, $a_i^{\prime}, \tilde{a}_i^{\prime} \in \cA_i$ and $a_j^{\prime \prime} \in \cA_j$,
    \begin{enumerate}[(1)]
        \item  $\underset{r, \vep \in[0,1]}{\sup}\left|\frac{\delta^2 V_i}{\delta a_i \delta a_j}\left(\bz+r\left(\ba^{\vep}-{ \bz }\right) ; a_i^{\prime}, a_j^{\prime \prime}\right)\right|<\infty$, where $\ba^{\vep}:=\left(a_i+\vep\left(\tilde{a}_i^{\prime}-a_i\right), a_{-i}\right)$; \label{item: boundedness}
        \item 
        \label{item:continuity}
        $[0,1]^N \ni \varepsilon \mapsto \frac{\delta^2 V_i}{\delta a_i a_j}\left(\bz+\varepsilon \cdot(\ba-\bz) ; a_i^{\prime}, a_j^{\prime \prime}\right)$ is continuous at 0 , where $\bz+\varepsilon \cdot (\ba- \bz):=\left(z_i+\varepsilon_i\left(a_i-z_i\right)\right)_{i \in { [N]}}$.
    \end{enumerate}
    Fix $\bz \in \cA^{(N)}$ and define
    $\Phi:\cA^{(N)} \to \sR$ by 
    \begin{equation}\label{eq:Phi}
        \Phi(\ba)=\int_0^1 \sum_{j=1}^N \frac{\delta V_j}{\delta a_j}\left(\bz+r(\ba-\bz) ; a_j-z_j\right) \mathrm{d} r.
    \end{equation}
    Then $\Phi$ is an $\alpha$-potential function of $\cG$ with 
    \begin{equation}
    \label{eq:alpha_general}
         \alpha \le  { 2} \sup_{i \in { [N]}, a'_i\in \cA_i, \ba ,\ba''   \in \cA^{(N)}}
  \sum_{j=1}^N\left| \frac{\delta^2 V_i}{\delta a_i \delta a_j} \left(\ba ; a_i^{\prime} , a''_j  \right) 
-\frac{\delta^2 V_j}{\delta a_j \delta a_i}\left(\ba ; a''_j , a_i^{\prime} \right)  
 \right|.
    \end{equation}
\end{theorem}


 Theorem \ref{thm:potential_hessian} constructs an $\alpha$-potential function using the objective functions' linear derivatives, which exist for general strategy classes without requiring a topological structure.
 The corresponding 
    $\alpha$  is quantified  explicitly   in terms of the magnitude of the asymmetry of the second-order linear derivatives, and  $\alpha=0$ recovers  the symmetric case in these earlier works \cite{monderer1996potential, leonardos2021global, hosseinirad2023general}. 
{  The base-point action $\bz$ ensures that $\bz + r(\ba - \bz)$ remains in $\cA^{(N)}$, so that the linear derivatives are well-defined. The specific choice of $\bz$ will not change the upper bound of $\alpha$, as \eqref{eq:alpha_general} takes the supremum over all strategies, but it may lead to different minima of $\alpha$-potential functions $\Phi$.}
    The proof of Theorem \ref{thm:potential_hessian} is given in Section \ref{sec:proof_theorem_hessian}.

{

 The $ \alpha$-potential function \eqref{eq:Phi} involves aggregating all players' strategies and the derivatives of their objective functions linearly through the parameter $r$.
When the objective functions are sufficiently regular, analogue $\alpha$-potential functions can be constructed through nonlinear aggregation of all players' strategies.    
Indeed, we have
    \begin{prop}
 Suppose  that 
 for all $i\in [N]$,
 $\cA_i$ is an open subset of a  normed vector space,
 and 
    the objective function $V_i$ is  continuously Fr\'echet differentiable in $\cA_i$. Fix $\bz \in \cA^{(N)}$, and for all $i\in [N]$, let $p_i:[0,1]\times \cA_i\to \cA_i$ be a continuously differentiable reparameterization of  $\cA_i$ such that 
for all $a_i\in \cA_i$, $p_i(0,a_i)=z_i$ and 
$p_i(1,a_i)=a_i$. 
Then one can define 
\begin{equation}\label{eq:Phi_p}
    \Phi(\ba) =\int_0^1 \sum_{i=1}^N ({\partial_{a_i} V_i}) (p(r, \ba)) \partial_t p_i(r ,a_i) \d r,
\end{equation}
where $p(r,\ba) \coloneqq  (p_i(r,a_i))_{i\in [N]}$,
and   ${\partial_{a_i} V_i}$ is the   
 Fr\'echet derivative of $V_i$.
 \end{prop}
Consequently, if we assume further regularity of $(V_i)_{i\in [N]}$, the corresponding $\alpha$ 
 for \eqref{eq:Phi_p}
 can be quantified in terms of the asymmetry in second-order derivatives of   
 objective functions
 and the derivatives of the parameterization $p$
 as in Theorem \ref{thm:potential_hessian}.

 It is worth noting that this $\alpha$-potential function \eqref{eq:Phi_p} extends the characterization of potential functions for static games with finite-dimensional strategy spaces as established in \cite[Theorem 4.5]{monderer1996potential},
and coincides the expression  \eqref{eq:Phi} 
by setting   
  $\frac{\delta V_i}{\delta a_i}(\ba; a'_i)=({\partial_{a_i} V_i})  (\ba) a'_i $
  and $p_i(r,a_i)=z_i+r (a_i-z_i)$.
When the game $\cG$ is a potential game (i.e., $\alpha=0$), any potential function is given by \eqref{eq:Phi_p} (or \eqref{eq:Phi}) up to an additive constant, as all potential functions share the same gradient and are therefore equivalent up to a constant. 
 

For ease of exposition and clarity,  we focus on the $\alpha$-potential function given in \eqref{eq:Phi} in the subsequent analysis.
As we will see, for stochastic differential games, 
the adoption of linear derivatives in  \eqref{eq:Phi} simplifies the analysis and avoids the tedious verification of the Fr\'echet differentiability of
 $(V_i)_{i\in [N]}$. 
Moreover, 
minimizing \eqref{eq:Phi} will be shown as a class of  conditional McKean-Vlasov control problem.

 }

 \section{Stochastic differential game and its  $\a$-potential function}  
\label{sec:general_stochastic_game}

This section characterizes  $\alpha$-potential function \eqref{eq:Phi} given in  Theorem \ref{thm:potential_hessian} 
for 
 stochastic
differential
games whose state
dynamics is a controlled diffusion. 
Under  suitable regularity conditions,
the linear derivative of { objective function}s are   characterized through
  the sensitivity processes of the state dynamics with respect to controls. 

 Let $T\in (0,\infty)$, 
 let 
   $(\Omega,\cF,\sP)$ 
be a complete probability space on which an $m$-dimensional Brownian motion $W=(W_t)_{t\ge 0}$ is defined,
and 
let  $\sF $  be
 the  $\sP$-completion of the  filtration generated by  $W$.
For each $p\ge 1$ and Euclidean space $(E,|\cdot|)$, let $\cS^p( E)$  be the space of $E$-valued $\sF$-progressively measurable processes $X:\Omega\times [0,T]\to E$  satisfying $\|X\|_{\cS^p(E)}=\sE[\sup_{s\in [0,T]}|X_s|^p]^{1/p}<\infty$,
and let $\cH^p(E)$  be the space of $E$-valued $\sF$-progressively measurable processes $X:\Omega\times [0,T]\to E$  satisfying $\|X\|_{\cH^p(E)}=\sE[\int_0^T|X_s|^p\d s]^{1/p}<\infty$.
With a slight abuse of notation, for any $m,n\in \sN$,
we identify the product spaces $ \cS^p(\sR^n)^m$ and  $\cH^p(\sR^n)^m$ with 
$\cS^p(\sR^{mn})$ and 
$\cH^p(\sR^{mn})$, respectively.

Consider the differential game 
$\cG_{\rm diff} $ defined as follows: 
let ${ [N]}=\{1,\ldots,N\} $, 
and for each  $i\in { [N]}$, let 
$A_i\subset \sR^n $ be a convex  set, 
and 
let 
$\cA^i$ be the set of processes $u_i\in   \cH^2(\sR^n)$
taking values in   $A_i$, 
 representing the set of admissible (open-loop) controls of player $i$.
For each $ \bu=(u_i)_{i\in { [N]}} \in \cH^2(\sR^{Nn})$, let  ${\bX}^{ \bu} = (X^{\bu}_{i})_{i=1}^N $ be the associated  state process governed by the following dynamics: 
for all $i\in { [N]}$ and $t\in [0,T]$,
\begin{equation} 
\label{eq:X_i_general} 
     \d X_{t,i}=  b_i(t,  
     {\bX}_{t }, \bu_{t})    \d t+\sum_{k=1}^m \sigma_{ik}(t,  {\bX}_{t }, \bu_{t}) \d {W}^k_t, 
    \quad X_{0,i}=x_i,
\end{equation}
where $ x_i\in \sR^d$ is a given initial state,
$  b_i :[0,T]\times  \sR^{Nd }\times \sR^{Nn }\to \sR^d $ 
and 
$\sigma_i=(\sigma_{i1},\ldots, \sigma_{im}) :[0,T]\times  \sR^{Nd }\times \sR^{Nn }\to \sR^{d\times m}$
are    given measurable  functions, 
and $W=({W}^k)_{k=1}^m $ is an $m$-dimensional $\sF$-Brownian motion on the space $(\Omega, \cF,\sP)$.
The { objective function} $V_i:\cA^{(N)}\subset \cH^2(\sR^{Nn })\to \sR$ of player $i$ is given by
\begin{align}
\label{eq:value_i_general}
V_i(\bu)=\sE\left[\int_0^T f_i(t, \bX^{\bu}_t, \bu_t)\,\d t+g_i( \bX^{\bu}_T)\right],
\end{align}
where $f_i:[0,T]\times \sR^{N d} \times \sR^{Nn } \to \sR$ and $g_i:\sR^{Nd } \to \sR$ are given measurable functions. 
Player $i$ aims to minimize \eqref{eq:value_i_general} over all admissible controls in $  \cA_i$.

We impose the following    regularity condition  on the coefficients of \eqref{eq:X_i_general}-\eqref{eq:value_i_general}. It guarantees   for each $\bu\in \cH^2(\sR^{Nn})$, \eqref{eq:X_i_general} admits a unique 
strong solution ${\bX}^{ \bu} \in \cS^2(\sR^{Nd})$, 
and \eqref{eq:value_i_general}  is well-defined. 
\begin{assumption}
\label{assum:regularity_general}
    For all $i\in { [N]}$, $A_i$ is a nonempty convex subset of $\sR^n$.
\begin{enumerate}[(1)]
\item  For all  $t\in [0,T]$,
    $(x,u)\mapsto (b_i(t,x,u), \sigma_i(t,x,u), f_i(t,x,u), g_i(x))$ is twice continuously differentiable.
\item    
For all $\varphi\in \{b_i,\sigma_i\}$,
$\sup_{t\in [0,T]} |\varphi(t,0,0)| <\infty$,
and $(x,u)\mapsto \varphi(t,x,u)$   has bounded first and second derivatives   (uniformly in $t$).
\item 
\label{item:f_g_general}
$\sup_{t\in [0,T]}(|f_i(t,0,0)|+|(\partial_{(x,u)}f_i)(t,0,0)|)<\infty$,
and 
$(x,u)\mapsto (f_i(t,x,u), g_i(x))$     has bounded   second derivatives (uniformly in $t$).
\end{enumerate}
 
\end{assumption}


We proceed to characterize the $\alpha$-potential  function \eqref{eq:Phi} for the   game $\cG_{\rm diff}$.
This is achieved by 
expressing  the linear derivatives of the { objective function}   \eqref{eq:value_i_general}  using the sensitivity processes of the state dynamics \eqref{eq:X_i_general}. 
In this following, we present only the first-order linear derivatives, as these are sufficient to characterize the $\alpha$-potential function. The second-order linear derivatives are given in Section \ref{sec:open_loop}, which will be used to quantify the constant $\alpha$ defined in  \eqref{eq:alpha_general}.

We start by introducing the sensitivity of the controlled state with respect to a single player's  control.
For each 
   $ \bu \in \cH^2(\sR^{Nn})$,  
   let $\bX^\bu $
   be the state process satisfying  \eqref{eq:X_i_general}. For each 
   $h\in { [N]}$ and   $u'_h\in \cH^{2}(\sR^n)$, 
define  $\bY^{\bu,u'_h} \in \cS^2(\sR^{Nd})$ 
as the solution of the following dynamics: 
for all  $t\in [0,T]$ and $i\in { [N]}$,
\begin{align} 
\label{eq:Y_h_i_general}
\begin{split}
    \mathrm{d} Y^h_{t,i} &=((\partial_x{b}_i) \left(t,  \bX^{\bu}_t, \bu_t \right) \bY^{h}_{t}
     + (\partial_{u_h}{b}_i)\left(t,  \bX^{\bu}_t,  \bu_t \right) u'_{t,h})\mathrm{d} t
     \\
     &\quad +\sum_{k=1}^m\left((\partial_x{\sigma}_{ik}) \left(t,  \bX^{\bu}_t, \bu_t \right)  \bY^{h}_{t}
     + (\partial_{u_h}{\sigma}_{ik})\left(t,  \bX^{ \bu}_t,  \bu_t \right) u_{t,h}'\right)\mathrm{d} W^k_t
        , \quad Y^h_{0,i}=0.
 \end{split}
\end{align}
By  \cite[Lemma 4.7]{carmona2016lectures},
for all  $u'_h\in \cH^2(\sR^n)$, 
$   
\lim_{\varepsilon\searrow 0}
\sE\Big[\sup_{t\in [0,T]}
\left|\frac{1  }{\varepsilon}(\bX^{\bu^\varepsilon}_t -\bX^{\bu}_t)
-
\bY^{\bu,u'_h}_t
\right|^2
\Big]
=0,
$   
where $\bu^\varepsilon =(u_h+\varepsilon u'_h ,u_{-h})$ for all $\varepsilon\in (0,1)$.
That is,
in the $L^2$ sense,
$\bY^{\bu,u'_h}$ is the derivative  of the controlled state $\bX^{\bu}$ when player $h$   varies  her control  in the direction $u'_h$.

Now, the linear derivatives of $V_i$ in 
\eqref{eq:value_i_general}
can be represented using the sensitivity processes given by  
\eqref{eq:Y_h_i_general}.
Indeed,  
for all $i, h\in { [N]}$, define the map
$\frac{\delta V_i}{\delta u_h}: \cA^{(N)} \times  \cH^2(\sR^n)  \rightarrow \mathbb{R}$ such that for all $\bu\in \cA^{(N)}$ and $u'_h\in \cH^2(\sR^n)$,
 \begin{equation}
 \label{eq:linear_derivative_general}
    \begin{aligned}
    & \frac{\delta V_i}{\delta u_h}\left(\bu ; u_h^{\prime}\right)
    \coloneqq \mathbb{E}\Bigg[
    \int_0^T
    \begin{pmatrix}
        \bY^{\bu, u'_h}_t \\
    u_{t,h}' 
    \end{pmatrix}^\top 
    \begin{pmatrix}
        \partial_x f_i \\
    \partial_{u_h} f_i
    \end{pmatrix}
     (t, \bX_t^{ \bu}, \bu_t  )\,
     \mathrm{d} t
     + (\partial_x g_i)^\top (\bX_T^{ \bu}) \bY^{\bu, u'_h}_T\Bigg]. 
    \end{aligned}  
\end{equation}
By 
the convexity of $\cA_h$ and 
\cite[Lemma 4.8]{carmona2016lectures},   
for all $u'_h\in \cA_h$,
$\lim _{\varepsilon \searrow 0} \frac{V_i\left(\bu ^\vep\right) -V_i(\bu)}{\varepsilon}=\frac{\delta V_i}{\delta u_i}\left(\bu ; u_h^{\prime}-u_h\right) 
$, 
where $\bu^\varepsilon =(u_h+\varepsilon (u'_h-u_h) ,u_{-h})$ for all $\varepsilon\in (0,1)$.
That is, 
$\frac{\delta V_i}{\delta u_h}$ is the linear derivative of $V_i$ with respect to $\cA_h$.  

Using
the expression \eqref{eq:linear_derivative_general} of $(\frac{\delta V_i}{\delta u_i })_{i\in { [N]}}$,
the following theorem 
characterizes the 
$\alpha$-potential function 
for the differential game $\cG_{\rm diff}$.


\begin{theorem}
\label{thm:potential_diff}
Consider the  game $\cG_{\rm diff}$  defined by \eqref{eq:X_i_general}-\eqref{eq:value_i_general}.
Suppose Assumption \ref{assum:regularity_general} holds.
For   any fixed  $\bz=(z_i)_{i\in { [N]}}\in \cA^{(N)}$,
the  function $\Phi:\cA^{(N)} \to \sR$  in  \eqref{eq:Phi}
can be expressed as
\begin{align}
\label{eq:potential_differential}
\begin{split}
\Phi(\bu)&  =\int_0^1      \sum_{i=1}^N  \mathbb{E}\Bigg[
    \int_0^T
    \begin{pmatrix}
        \bY^{ \bu^r,     u_{i}-z_{i}}_t \\
    u_{t,i}-z_{t,i} 
    \end{pmatrix}^\top 
    \begin{pmatrix}
        \partial_x f_i \\
    \partial_{u_i} f_i
    \end{pmatrix}
      (t, \bX_t^{\bu^r}, \bu^r_t   )\,
     \mathrm{d} t
     + (\partial_x g_i)^\top  (\bX_T^{ \bu^r} ) \bY^{\bu^r, u_{i}-z_{i}}_T\Bigg] \d r
 \end{split}
\end{align}
{with $\bu^r \coloneqq \bz+r(\bu-\bz)$}.

\end{theorem}
{
The expression 
\eqref{eq:potential_differential}
follows directly from 
\eqref{eq:Phi} for $\Phi(\bu)$
and \eqref{eq:linear_derivative_general}
for $\frac{\delta V_i}{\delta u_i }$,
by substituting 
$h$ with $i$,
$\bu $
with $\bz+r(\bu-\bz)$,
and $u'_h$ with $u_i-z_i$. 
}

The $\alpha$-potential function in \eqref{eq:potential_differential}
 can be alternatively expressed using backward stochastic differential
equations (BSDEs). The proof follows directly from {\cite[Corollary 4.11]{carmona2016lectures}.}   

\begin{prop}
Under the setting of Theorem \ref{thm:potential_diff}, $\Phi$ defined in \eqref{eq:potential_differential} can be 
equivalently written as 
\begin{align}
\label{eq:Phi_bsde}
\begin{split}
\Phi(\bu)&  =\int_0^1 \sum_{i=1}^N  \mathbb{E}\left[\int_0^T
( \partial_{u_i} \sH^i)^\top (t, \bX^{\bu^r}_t, \bu^r_t, \bG^{i,\bu^r}_t,  \bH^{i,\bu^r}_t) (u_{t,i}-z_{t,i})
  \mathrm{d} t \right]\d r 
  \quad \textnormal{with $\bu^r \coloneqq \bz+r(\bu-\bz)$},
 \end{split}
\end{align}
where   for each $i\in { [N]}$,
  $ 
\sH^i(t,x,u,\mathfrak{g})\coloneqq    b^\top (t,x , u )   \mathfrak{g} +\operatorname{tr}\big((\sigma\sigma^\top)(t,x,u)\mathfrak{h} 
\big)
+f_i(t,x,u) 
 $
 with 
 $b = \vcat(b_1, \ldots , b_N) $ 
 and 
 $\sigma =
 \vcat(\sigma_1,\ldots , \sigma_N)
$,\footnotemark 
 and for each 
  $\bu \in \cH^2(\sR^{Nn})$,
 $(\bG^{i,\bu}, \bH^{i,\bu})\in \cS^2(\sR^{Nd})\times \cH^2(\sR^{Nd \times m})$ satisfies  
\begin{align*}
\d \bG^{i}_t =-(\partial_x \sH^i)(t,\bX^{\bu}_t,\bu_t, \bG^{i}_t, \bH^{i}_t )\d t+ \bH^{i}_t  \d W_t,
\quad \forall t\in [0,T];
\quad \bG^i_T =(\partial_x g_i)(\bX^{\bu}_T).
\end{align*}

\footnotetext{
We denote by  
$\vcat(A_1, \ldots, A_N)\coloneqq (A_1^\top,\ldots, A_N^\top)^\top  $   the vertical concentration
of    matrices 
  $A_i\in \sR^{m_i\times n}$, $1\le i\le N$.
}
\end{prop} 
 
   In the sequel, we adopt the representation \eqref{eq:potential_differential} 
   of $\alpha$-potential function in terms of    the sensitivity processes.
    A  detailed exploration of 
   the BSDE approach for $\alpha$-potential games is left for future work.

{

\section{Quantifying  $\alpha$ for  stochastic differential game} \label{sec:open_loop}

In this section, we   
 quantify  $\alpha$ in \eqref{eq:alpha_general} 
 for stochastic differential games based on the structure of the game. The analysis relies on characterizing the second-order derivatives of { objective function}s by utilizing the second-order sensitivity of the state dynamics with respect to the controls.
 
 Let $\cG_{\rm diff}  $ be the differential game defined in Section \ref{sec:general_stochastic_game}.
  For ease of exposition, in this section, we assume that 
  each player has one-dimensional state and control processes, 
with the drift of \eqref{eq:X_i_general} depending linearly on the control and the diffusion of \eqref{eq:X_i_general}  being independent of both the state and control. Similar analysis can be extended to sufficiently regular nonlinear drift and diffusion coefficients in a multidimensional setting.
  More precisely, for each
  $ \bu=(u_i)_{i\in { [N]}} \in \cH^2(\sR^N)$, 
let  ${\bX}^{ \bu} = (X^{\bu}_{i})_{i=1}^N \in \cS^2(\sR^N)$ be the associated  state process governed by the following dynamics: 
for all $i\in { [N]}$ and $t\in [0,T]$,
\begin{equation} 
\label{eq:X_i_u_i_open} 
     \d X_{t,i}=
     \left( b_i(t,  X_{t,i}, 
     {\bX}_{t })+  u_{t,i}  \right)   \d t+\sigma_i(t) \d {W}_t^i, 
    \quad X_{0,i}=x_i,
\end{equation}
where  $x_i\in \sR$,
$  b_i :[0,T]\times \sR\times   \sR^N \to \sR$ is a given sufficiently regular function, 
$\sigma_i :[0,T] \to  \sR$ is a given   measurable function, 
and $W=({W}^i)_{i \in { [N]}} $ is an $N$-dimensional $\sF$-Brownian motion.
Let $\cA^{(N)}\subset \cH^2(\sR^N)$ be a nonempty convex set, representing the joint control profiles of all players.  
Player $i$'s { objective function} $V_i:\cA^{(N)} \to \sR$  is given as in \eqref{eq:value_i_general}:
\begin{align}
\label{eq:value_i_open}
V_i(\bu)=\sE\left[\int_0^T f_i(t, \bX^{\bu}_t, \bu_t)\,\d t+g_i( \bX^{\bu}_T)\right],
\end{align}
where $f_i:[0,T]\times \sR^N \times \sR^N \to \sR$ and $g_i:\sR^N \to \sR$ are given measurable functions.

 Note that in \eqref{eq:X_i_u_i_open}, we have   expressed the dependence of $b_i$ on the private state $X^{\bu}_i$ and the population state $\bX^{\bu}$ separately. This separation allows for specifying the structure of the drift coefficient.
To this end, let 
$\mathscr{F}^{0,2}([0,T]\times \sR\times \sR^{N} ; \sR )$ be the vector space of  measurable functions
$\psi: [0,T]\times \sR\times \sR^{N}\to \sR$ such that  
\begin{enumerate}[(1)]
\item for all $t\in [0,T]$, $(x,y)\mapsto \psi(t,x,y)$ is twice continuously differentiable,
\item there exists  $L^\psi, L^\psi_y \ge 0$ such that 
for all $  (t, x,y) \in [0,T]\times  \sR\times \sR^{N}$ and $i,j\in { [N]}$,
  $|   \psi  (t,0,0)|\le L^\psi$,
  $| (\partial_x \psi ) (t,x,y)|\le L^\psi$, 
  $| (\partial^2_{xx} \psi ) (t,x,y)|\le L^\psi$, 
  $| (\partial_{y_i} \psi ) (t,x,y)|\le L^\psi_y/N$,
  $| (\partial^2_{xy_i} \psi ) (t,x,y)|\le L^\psi_y/N$, and
  $| (\partial^2_{y_iy_j} \psi ) (t,x,y)|\le \frac{1}{N} L^\psi_{y} \mathds{1}_{i=j} +\frac{1}{N^2} L^\psi_{y}  \mathds{1}_{i\not =j} $.
\end{enumerate}
For any $\psi=(\psi_i)_{i\in { [N]}}  \in \mathscr{F}^{0,2}([0,T]\times \sR\times \sR^{N};\sR)^N$, 
we write $L^\psi=\max_{i\in { [N]}} L^{\psi_i}$ and $L^\psi_y= \max_{i\in { [N]}} L^{\psi_i}_y $.

In the sequel, we impose the following regularity conditions on the coefficients of \eqref{eq:X_i_u_i_open}-\eqref{eq:value_i_open}.

\begin{assumption}
\label{assum:regularity}
    For all $i\in { [N]}$,
    $b_i\in \mathscr{F}^{0,2}([0,T]\times \sR\times \sR^{N};\sR)$,
      $\sigma_i\in L^\infty([0,T];\sR)$,
    and $f_i$ and $g_i$ satisfy the conditions in Assumption \ref{assum:regularity_general}\ref{item:f_g_general}.
    
\end{assumption}

 \begin{remark}  
\label{rmk:mean_field} 
For each $i\in { [N]}$, the condition    $b_i\in \mathscr{F}^{0,2}([0,T]\times \mathbb{R}\times \mathbb{R}^{N};\mathbb{R})$ implies 
the partial derivatives of $y\mapsto b_i(t,x,y)$ 
admit explicit decay rates 
in terms of $N$. 
This assumption naturally holds if each player's state depends on the empirical measure of the joint state process, i.e., the mean field interaction.
To see it, suppose that 
$b_i(t,x,y)=h\big(t,x,\frac{1}{N}\sum_{j=1}^N \delta_{y_j}\big)$ with $ (t,x,y)\in [0,T]\times \sR\times \sR^{N},$
for a measurable function $h:[0,T]\times \sR\times \cP_2(\sR)\to \sR$,
where  $\cP_2(\sR)$ is the space of probability measures on $\sR$ with second moments.
If $(x,\mu)\mapsto h(t,x,\mu)$ is  sufficiently regular,  
then by \cite[Propositions 5.35 and 5.91]{carmona2018probabilistic},
$  (\partial_{y_i} b_i ) (t,x,y)=\frac{1}{N}(\partial_\mu h)(t,x,\frac{1}{N}\sum_{j=1}^N \delta_{y_j})( y_i)$,
$  (\partial^2_{x y_i} b_i ) (t,x,y)=\frac{1}{N}(\partial_\mu\partial_{x} h)(t,x,\frac{1}{N}\sum_{j=1}^N \delta_{y_j})( y_i)$,
and 
  \begin{align*}
 (\partial^2_{y_iy_j}b_i ) (t,x,y)=\frac{1}{N}(\partial_v  \partial_\mu h)\left(t,x,\frac{1}{N}\sum_{\ell=1}^N \delta_{y_\ell}\right)( y_i)\delta_{i,j}+\frac{1}{N^2}
 (\partial^2_\mu h)\left(t,x,\frac{1}{N}\sum_{\ell=1}^N \delta_{y_\ell}\right)( y_i, y_j),
  \end{align*}
  where $( \partial_\mu h)(t,x,\mu)(\cdot)  $ (resp.~$   ( \partial_\mu \partial_{x} h)(t,x,\mu)(\cdot)  $) is Lions  derivative of $\mu\mapsto h(t,x,\mu)$
  (resp.~$ \mu\mapsto(\partial_{x} h)(t,x,\mu)$), 
  $(\partial_v \partial_\mu h)(t,x,\mu)(\cdot)$ is the derivative  of $ v\mapsto ( \partial_\mu h)(t,x,\mu)(v) $,
  and  $(\partial^2_\mu h)(t,x,\mu)(v, \cdot) $ is the Lions derivative of $ \mu\mapsto ( \partial_\mu h)(t,x,\mu)(v)$.
  Hence if $\partial_\mu h$, $ \partial_\mu \partial_{x} h$, $\partial_v \partial_\mu h$ and $\partial^2_\mu h$ are continuous and uniformly bounded, then $b_i \in \mathscr{F}^{0,2}([0,T]\times \sR\times \sR^{N};\sR)$ with a constant $L^{b_i}_y$ depending on the upper bounds of the Lions  derivatives but independent of $N$.

The dependence of the constant $L^b_y = \max_{i\in { [N]}}L^{b_i}_y$ on $N$ reflects the degree of coupling among all players'  state dynamics. 
For instance, 
if  $L^b_y$
remains bounded as $N\to \infty$,
then 
the state dynamics can have mean field type interactions.
Alternatively, 
if  $L^b_y =0$, then all players' states are decoupled.

 \end{remark}

To 
 quantify  the magnitude of the asymmetry of the second-order linear derivatives of { objective function}s \eqref{eq:value_i_open},
 hence  $\alpha$   in \eqref{eq:alpha_general},
 we characterize the linear derivatives  
using the sensitivity processes of \eqref{eq:X_i_u_i_open}. 
Observe that for  the state dynamics \eqref{eq:X_i_u_i_open}, 
the dynamics   \eqref{eq:Y_h_i_general} for 
the first-order sensitivity process 
$\bY^{\bu,u'_h} \in \cS^2(\sR^N)$ simplifies into  for all  $t\in [0,T]$, 
\begin{align} 
\label{eq:Y_h_i_open}
\begin{split}
    \d Y^h_{t,i}
    &=
    \bigg[ (\partial_x b_i)(t, X^\bu_{t,i}, \bX^\bu_{t})Y^h_{t,i}
    +
\sum_{j=1}^N (\partial_{y_j} b_i)(t, X^\bu_{t,i}, \bX^\bu_{t})Y^h_{t,j}  
    +\delta_{h,i} u'_{t,h}
    \bigg]
    \d t, 
    \quad Y^h_{0,i}=0; \quad \forall i\in { [N]},
 \end{split}
\end{align} 
where $\delta_{i,j}$  denotes the Kronecker delta such that 
$\delta_{i,j}=0$ if $i=j$ and $0$ otherwise.
We now 
 characterize the second-order sensitivity of the state process with respect to the changes in two players' controls.
For each $h, \ell \in { [N]}$ with $h\not = \ell$,
and each  $u'_h, u''_\ell \in  \cH^{4}( \sR)$,
  define 
${\bZ}^{\bu,u'_h,u''_\ell} \in \cS^2(\sR^N) $
as the solution of 
    the following dynamics: 
for all $i\in { [N]}$ and $t\in [0,T]$,
\begin{align} 
\label{eq:Z_h_ell_i_open}
\begin{split}
    \d Z^{h,\ell} _{t,i}
    &=
    \bigg[ 
  (\partial_x b_i)(t, X^\bu_{t,i}, \bX^\bu_{t}) Z^{h,\ell}_{t,i}
    +
\sum_{j=1}^N (\partial_{y_j} b_i)(t, X^\bu_{t,i}, \bX^\bu_{t})Z^{h,\ell}_{t,j}
    +\mathfrak{f}^{\bu,u'_h,u''_\ell}_{t,i}
    \bigg]
    \d t, 
    \quad Z^{h,\ell}_{0,i}=0,
 \end{split}   
\end{align}
where  
$\mathfrak{f}^{\bu,u'_h,u''_\ell}_{i}:\Omega\times [0,T]\to \sR$ is defined by 
\begin{align}
\begin{split}
\label{eq:f_phi_h_ell_open_loop}
\mathfrak{f}^{\bu,u'_h,u''_\ell}_{t,i}
& \coloneqq 
 \begin{pmatrix}
  Y^{\bu,u'_h}_{t, i} 
  \\
  \bY^{\bu,u'_h}_t
\end{pmatrix}^\top
\begin{pmatrix}
 \partial^2_{xx} b_i    & \partial^2_{xy} b_i
\\
\partial^2_{yx} b_i  & \partial^2_{yy} b_i
\end{pmatrix}(t, X^\bu_{t,i},\bX^\bu_t)
 \begin{pmatrix}
  Y^{\bu,u''_\ell}_{t,i}
  \\
  \bY^{\bu,u''_\ell}_t
\end{pmatrix},
\end{split}
\end{align}
and   $\bY^{\bu,u'_h}$  and  $\bY^{\bu,u''_\ell} $
  are defined as in   \eqref{eq:Y_h_i_open}.
Similar arguments as that for  
  \cite[Lemma 4.7]{carmona2016lectures} show that 
for all     $u'_h, u''_\ell \in  \mathcal{H}^{4}( \sR)$, 
$ 
\lim_{\varepsilon\searrow 0}
\sE\left[\sup_{t\in [0,T]}
\left|
\frac{1 }{\varepsilon}(\bY^{\bu^\varepsilon,u'_h}_t -\bY^{\bu,u'_h}_t)
-{\bZ}^{\bu, u'_h, u''_\ell}_t 
\right|^2
\right]
=0,
$  
where $\bu^\varepsilon =(u_\ell+\varepsilon u''_\ell ,u_{-\ell})$ for all $\varepsilon\in (0,1)$.
That is,
${\bZ}^{\bu, u'_h, u''_\ell }$ is the second-order derivative  of the   state $\bX^{\bu}$ when player $h$ first  varies  her control  in the direction  $u'_h$, and then player $\ell$ varies her control in the direction $u''_\ell$.

Now, the linear derivatives of $V_i$ in 
\eqref{eq:value_i_open}
can be represented using the sensitivity processes satisfying  
\eqref{eq:Y_h_i_open} and \eqref{eq:Z_h_ell_i_open}.
The first order linear derivative 
$ \frac{\delta V_i}{\delta u_h}$
of $V_i$
is given  as in 
\eqref{eq:linear_derivative_general}.
For the second-order linear derivatives, 
  for all $h,\ell\in { [N]}$,
define the map $\frac{\delta^2 V_i}{\delta u_h \delta u_{\ell}}: \cA^{(N)} \times \cH^4(\sR)\times \cH^4(\sR) \rightarrow \mathbb{R}$
such that for all $\bu \in \cA^{(N)}$ and $u'_h,u''_\ell\in \cH^4(\sR^4)$,
    \begin{equation}
\label{eq:2d_linear_derivative_open}
        \begin{aligned}
& \frac{\delta^2 V_i }{\delta u_h \delta u_{\ell}}\left(\bu ; u_h^{\prime}, u_{\ell}^{\prime \prime}\right)
 \\
& =\mathbb{E}\left[\int_0^T
\left(
\begin{pmatrix}
\bY^{\bu, u'_h}_t \\
 u'_{t,h} 
\end{pmatrix}^{\top}
\begin{pmatrix}
    \partial^2_{x x} f_i & \partial^2_{x u_\ell } f_i \\
\partial^2_{u_h x} f_i & \partial^2_{u_h u_\ell} f_i
\end{pmatrix}\left(t, \bX_t^{  \bu}, \bu_t \right)
\begin{pmatrix}
\bY^{\bu, u''_\ell}_t \\
u''_{t,\ell}    
\end{pmatrix}
+ (\partial_x f_i )^\top 
\left(t, \bX_t^{ \bu}, \bu_t \right)
\bZ^{\bu, u'_h, u_\ell''}_t
\right)
\mathrm{d} t\right] \\
&\quad  +\mathbb{E}\left[\left(
\bY^{\bu, u'_h}_T
\right)^{\top}\left(\partial^2_{x x} g_i\right)(\bX_T^{  \bu}) \bY^{\bu, u''_\ell}_T + \left(\partial_x g_i\right)^\top \left(\bX_T^{\bu}\right) \bZ^{\bu, u_h',u_\ell''}_T\right] .
\end{aligned}
    \end{equation}
Consequently, 
if $\cA_h$ and $\cA_\ell $ are convex subsets of $\cH^4(\sR)$, then by \cite[Lemma 4.8]{carmona2016lectures}, 
$\lim_{\varepsilon \searrow 0}
\frac{1}{\vep}
\big(\frac{\delta V_i}{\delta u_h}(\bu^\vep ; u_h^{\prime})
 - \frac{\delta V_i}{\delta u_h}(\bu ; u_h^{\prime})
\big)\\
=
\frac{\delta^2 V_i }{\delta u_h \delta u_{\ell}}(\bu ; u_h^{\prime}, u''_{\ell}-u_\ell)
$ for all $u'_h\in \cA_h$
and $u''_\ell \in \cA_\ell$,
where $\bu^\varepsilon =(u_\ell +\varepsilon (u''_\ell -u_\ell) ,u_{-\ell})$ for all $\varepsilon\in (0,1)$.
That is, 
$\frac{\delta^2 V_i}{\delta u_h \delta u_\ell }(\bu; u'_h, \cdot) $ is the linear derivative of $\bu  \mapsto \frac{\delta  V_i}{\delta u_h  }(\bu; u'_h)$ with respect to $\cA_\ell$, and hence the second-order linear derivative of $V_i$.

{ 
Before stating the theorem, we introduce a few constants that will be used in the analysis. For any \( i, j \in [N] \) with \( i \neq j \), we define $
\Delta^f_{i,j} = f_i - f_j$, $\Delta^g_{i,j} = g_i - g_j,$  
as well as the following three constants $C_{V,1}^{i,j}$ $C_{V,2}^{i,j}$ $C_{V,3}^{i,j}$,  depending  on the upper bounds of the first- and second-order derivatives of $\Delta^f_{i,j}$ and $\Delta^g_{i,j}$ in $(x,u)$:
\begin{align}
\label{eq:C_V_1_open} 
    C_{V,1}^{i,j} & \coloneqq \|   \partial^2_{x_ix_j } \Delta^f_{i,j }  \|_{L^\infty}  + \|   \partial^2_{x_i u_j } \Delta^f_{i,j }  \|_{L^\infty}  + \|   \partial^2_{u_i x_j } \Delta^f_{i,j }  \|_{L^\infty} + \|   \partial^2_{u_i u_j }\Delta_{i,j}^f \|_{L^\infty}+  \|   \partial^2_{x_ix_j } \Delta^g_{i,j }  \|_{L^\infty},
\\
\label{eq:C_V_2_open} 
  C_{V,2}^{i,j} 
   &   \coloneqq \sum_{ \ell\in { [N]}\setminus\{j\}}
  \| \partial^2_{ u_i x_\ell} \Delta_{i,j}^f \|_{L^\infty} 
+  \sum_{ h\in { [N]}\setminus\{i\}} \| \partial_{x_h u_j} \Delta_{i,j}^f \|_{L^\infty} 
+  \sum_{h\in \{i,j\} }  \left( \|( \partial_{x_h} \Delta^f_{i,j } ) (\cdot, 0, 0) \|_{L^2} + |( \partial_{x_h} \Delta^g_{i,j } ) ( 0)| \right) \notag\\
&\quad 
+  \sum_{h\in \{i,j\}, \ell\in { [N]}}  
 \left(\|  \partial^2_{x_hx_\ell } \Delta^f_{i,j } \|_{L^\infty} 
+\|  \partial^2_{x_h u_\ell } \Delta^f_{i,j } \|_{L^\infty} 
+\|  \partial^2_{x_hx_\ell } \Delta^g_{i,j } \|_{L^\infty} 
\right),
\\
     C_{V,3}^{i,j} & \coloneqq 
     \sum_{h\in { [N]}\setminus \{i,j\} }   
\left(\|( \partial_{x_h} \Delta^f_{i,j } ) (\cdot, 0, 0) \|_{L^2}
+ | (\partial_{x_h} \Delta^g_{i,j } ) ( 0)| \right)
  \nonumber  \\
    &\quad 
    + \sum_{\substack{h\in { [N]}\setminus \{i,j\}\\\ell \in { [N]}\setminus\{i,j \}} } 
 \left(\|  \partial^2_{x_h x_\ell } \Delta^f_{i,j } \|_{L^\infty} 
+\|  \partial^2_{x_h u_\ell } \Delta^f_{i,j } \|_{L^\infty} 
+\|  \partial^2_{x_h x_\ell } \Delta^g_{i,j } \|_{L^\infty} \right),
\label{eq:C_V_3_open} 
\end{align}
where $\|\cdot\|_\infty$ denotes the sup-norm norm.}
  
We are ready to present the upper bound of the  $\alpha$ defined in \eqref{eq:alpha_general} for general cost functions $(f_i,g_i)_{i\in { [N]}}$, \emph{without imposing any structural assumptions}.
\begin{theorem}
\label{thm:value_jacobian_open}
Suppose Assumption \ref{assum:regularity} holds.
Then for all   $ \bu\in \cH^2(\sR^N)   $ and 
$ u'_i, u''_j\in \cH^4(\sR)$,
    \begin{equation}
        \label{eq:upper_bound_difference}
        \left|\frac{\delta^2 V_i}{\delta u_i \delta u_j}\left(\bu ; u_i^{\prime}, u_j^{\prime \prime}\right)-\frac{\delta^2 V_j}{\delta u_j \delta u_i}\left(\bu ; u_j^{\prime \prime}, u_i^{\prime}\right)\right|
        \leq C \| u_i'\|_{\cH^4(\sR)} \| u_j''\|_{\cH^4(\sR)}  \left(C_{V,1}^{i,j} + L_y^b 
        \left(
        \frac{ 1}{N}C_{V,2}^{i,j} + \frac{ 1}{N^2} C_{V,3}^{i,j}\right)\right),
    \end{equation}
where 
the constant $L^b_y$
represents the coupling in the state dynamics (see Remark \ref{rmk:mean_field}),
the constants 
$C_{V,1}^{i,j}$, $C_{V,2}^{i,j}$
and $C_{V,3}^{i,j}$,
defined in \eqref{eq:C_V_1_open}, \eqref{eq:C_V_2_open} and \eqref{eq:C_V_3_open}, respectively, 
and 
the constant  $C\ge 0$  
 depends  only on   the upper bounds of      $T$, 
 $\max_{i\in { [N]}} |x_i|$,
 $\max_{i\in { [N]}} \|\sigma_i\|_{L^2}$, 
 $L^{b }$ and  $L^{b }_y$.

Consequently, if
$ \sup_{i\in { [N]}, u_i\in \cA_i}\|u_i\|_{\cH^4(\sR)}<\infty$ and $0\in \cA_i$,
then $\cG_{\rm diff}$
is an $\alpha$-potential game 
with 
  an $\alpha$-potential function $\Phi$
  given by \eqref{eq:potential_differential} with $\bz =0$,
  and a constant $\alpha$ satisfying 
\begin{equation}
\label{eq:alpha_bound}
\alpha \le  C 
 \underset{i \in { [N]}}{\max} 
  \underset{j\in { [N]}\setminus \{i\}}{\sum}  \Big(C_{V,1}^{i,j} + L_y^b 
        \big(
        \frac{ 1}{N}C_{V,2}^{i,j} + \frac{ 1}{N^2} C_{V,3}^{i,j}\big)\Big)
\end{equation}
for a constant  $C\ge 0$     independent of the cost functions. 

\end{theorem}

\begin{remark}
\label{rmk:alpha_bound}
Since  the minimizer of the function $\Phi$ given in Theorem \ref{thm: verification} is an $\epsilon$-Nash equilibrium of the game 
\eqref{eq:X_i_u_i_open}-\eqref{eq:value_i_open},
with $\epsilon\le \alpha$ in  \eqref{eq:alpha_bound}. 
One can construct approximate Nash equilibria for 
$N$-player games without  the symmetry and homogeneity conditions among players imposed for mean field approximations  \cite{ 
carmona2018probabilistic}.
Moreover,  
the   upper bound   \eqref{eq:alpha_bound} is expressed in terms of the number of players, the strength of interactions, and the degree of heterogeneity among the players, proving rich insights for assessing the approximate Nash equilibria in relation to the game structure, compared to the classical mean field approximation, which typically bounds the approximation error solely based on the number of players $N$.

{  

The   condition 
$ \sup_{i\in { [N]}, u_i\in \cA_i}\|u_i\|_{\cH^4(\sR)}<\infty$ 
for the estimate  
\eqref{eq:alpha_bound}
can be relaxed to 
$ \sup_{i\in { [N]}, u_i\in \cA_i}\|u_i\|_{\cH^2(\sR)}<\infty$ 
if the state dynamics 
\eqref{eq:X_i_u_i_open} is decoupled, i.e., if the drift $b_i$ is independent of  $(X_j)_{j\not =i}$;
see Example \ref{ex:dist_game_open_distributed}
and Section \ref{sec:example}.
Indeed, 
the appearance of  $\|u'_i\|_{\cH^4}$ and    $\|u''_j\|_{\cH^4}$ 
 in the estimate   \eqref{eq:upper_bound_difference} 
is due to the $L^2$-estimate of the process $\bZ^{\bu, u'_i,u''_j}$;
see Proposition \ref{prop:Z_h_ell_i_moment_open} and \eqref{eq:cost_f_x_Z_open_loop}. If the state dynamics is decoupled, then  $\bZ^{\bu, u'_i,u''_j} =0$ 
for $i\not = j$, 
and the additional condition on the $\|\cdot\|_{\cH^4(\sR)}$ is unnecessary. 
The  uniform integrability condition 
$ \sup_{i\in { [N]}, u_i\in \cA_i}\|u_i\|_{\cH^2(\sR)}<\infty$ 
comes from estimating the $\|\cdot\|_{\cH^2(\sR)}$-norm 
of $\bY^{\bu, u'_i}$
uniformly over $u'_i\in \cA_i$ and $i\in [N]$.
 This highlights the dependence of $\alpha$ on the choice of admissible control classes $(\cA_i)_{i\in [N]}$. This dependence may be useful for analyzing the sensitivity of $\alpha$-NE with respect to design of game strategies.  

A similar estimate of $\alpha$ can be established if the drift coefficient in \eqref{eq:X_i_u_i_open} depends nonlinearly on $u_i$. In such cases, the sensitivity equations \eqref{eq:Y_h_i_open} and \eqref{eq:Z_h_ell_i_open} will incorporate the derivatives of the drift coefficient with respect to $u_i$, and the constant $C$ in Theorem \ref{thm:value_jacobian_open} will depend on the upper bounds of these derivatives.
}
    
\end{remark}
 
The proof of Theorem \ref{thm:value_jacobian_open} 
 is given in Section \ref{sec:open_loop_jacobian_proof}. 
The essential step is to establish precise   estimates of the sensitivity processes $\bY^{\bu,u'_h}$ and $\bZ^{\bu,u'_h, u''_\ell}$
in terms of the number of players $N$ and the indices $h,\ell$. 
These estimates quantify the   dependence of each player's state process on the changes in other players' controls, with a constant depending explicitly on the coupling strength $L^b_y$ in the drift coefficients (see Remark \ref{rmk:mean_field}) and the number of players. 

Theorem \ref{thm:value_jacobian_open} simplifies the task of quantifying the constant $\alpha$ in \eqref{eq:alpha_general} to bounding the difference of derivatives of the cost functions. 
For instance, 
 the following example presents a special case where the state dynamics \eqref{eq:X_i_u_i_open}  are decoupled, and 
$\alpha=0$. As a result,
the minimizer of the potential function $\Phi$ given in Theorem \ref{thm: verification} is a Nash equilibrium of the  $N$-player game \eqref{eq:X_i_u_i_open}-\eqref{eq:value_i_open}.

\begin{example}[games with $\alpha=0$]
\label{ex:dist_game_open_distributed}
 Consider the  game $\cG_{\rm diff} $  defined as in \eqref{eq:X_i_u_i_open}-\eqref{eq:value_i_open}.
 Suppose Assumption \ref{assum:regularity} holds, and 
for all $i\in { [N]}$,
$(t,x,y)\mapsto b_i(t,x,y)$ is independent of $y$, and 
    $f_i$ and $g_i$ are of the form
    \begin{align}
    \label{eq:distributed_open_potential}
    f_i(t, x,  u) &= c_i(t, x_i,u_i)  +\tilde{f}_i\left(t,  x_i, u_i, \bar \mu_{(x,u)_{-i}}\right),
    \quad
    g_i (x) = \bar{g_i}(x_i)+\tilde g_i\left(x_i,   \bar \mu_{x_{-i}} \right).
\end{align}
where $\bar{\mu}_{(x,u)_{-i}} =\frac{1}{N-1}\sum_{j\in I_{N}\setminus\{i\}}\delta_{(x_j,u_j)} $,  
$\bar{\mu}_{x_{-i}} =\frac{1}{N-1}\sum_{j\in I_{N}\setminus\{i\}}\delta_{x_j } $,
and $c_i:[0,T]\times \sR\times \sR\to \sR$, $\tilde{f}_i:[0,T]\times   \sR\times \sR\times \cP_2(\sR\times \sR)\to \sR$,
$\bar g_i: \sR\to \sR$,
and $\tilde{g}_i:\sR\times \cP_2(\sR)\to \sR$ are twice continuously differentiable.
Assume further that 
there exist twice continuously differentiable   functions $F: [0,T]\times  \cP_2(\R\times \sR) \to  \R$ and $G: \cP_2(\R)\to  \R$
such that   for all $i\in { [N]}$, $(t, x,u)\in [0,T]\times \sR^N\times \sR^N$ and $(x'_i,u'_i)\in \sR\times \sR$,
\begin{equation}\label{eqn:f_g_in_carmona}
   \begin{aligned}
 \tilde{f}_i(t, x_i, u_i, \bar \mu_{(x,u)_{-i}} )-\tilde{f}_i(t, x^{\prime}_i, u'_i, \bar \mu_{(x,u)_{-i}}  ) 
 &=F\Big(t, \frac{1}{N} \delta_{(x_i,u_i)}+\frac{N-1}{N}\bar \mu_{(x,u)_{-i}} \Big)-F\Big( t, \frac{1}{N} \delta_{(x^{\prime}_i,u'_i)}+\frac{N-1}{N}\bar \mu_{(x,u)_{-i}} \Big),
\\
\tilde{g}_i\left(x_i,\bar \mu_{x_{-i}} \right)-
\tilde{g}_i\left(x^{\prime}_i, \bar \mu_{x_{-i}}  \right) 
&=G\Big(\frac{1}{N} \delta_{x_i}+\frac{N-1}{N}\bar \mu_{x_{-i}}   \Big)
-G\Big(\frac{1}{N} \delta_{x^{\prime}_i}+\frac{N-1}{N}\bar \mu_{x_{-i}} \Big).
\end{aligned} \end{equation}
Then $\alpha =0$ and  $\cG_{\rm diff}$  is a potential game.

\end{example}

Example \ref{ex:dist_game_open_distributed} extends \cite[Proposition 2.24]{carmona2018probabilistic} from cost functions dependent solely on state variables to those dependent on both state and control variables.
It 
 follows from the fact that  
 by    
 \eqref{eqn:f_g_in_carmona}, 
      $h_i^f \coloneqq  \tilde{f}_i -F$
    and 
 $h_i^g \coloneqq  \tilde{g}_i -G$   are   independent of $(x_i,u_i)$.
  This  implies $C^{i,j}_{V,1}=0$ as 
 defined in \eqref{eq:C_V_1_open}.
  As the states are decoupled,   $L^b_y=0$ and hence $\alpha=0$ by Theorem \ref{thm:value_jacobian_open}.
  

When all players' state dynamics \eqref{eq:X_i_u_i_open} are coupled, 
a stronger  condition on the cost functions is needed   to ensure 
the constant $\alpha$ in \eqref{eq:alpha_bound} decays to zero as the number of players $N\to \infty$. 
The following example shows that if the cost functions in \eqref{eq:value_i_open} depend  on
the joint states and controls   
only through their 
 empirical measures, then  the $N$-player game \eqref{eq:X_i_u_i_open}-\eqref{eq:value_i_open} is an $\alpha$-potential game with $\alpha =\cO(1/N)$ as $N\to \infty$.

\begin{example}[games with mean field interactions]
\label{ex:mean_field_open}

Consider the  game $\cG_{\rm diff}  $  defined by \eqref{eq:X_i_u_i_open}-\eqref{eq:value_i_open}.
Suppose Assumption \ref{assum:regularity} holds 
and 
there exists $L  \ge 0$ such that  
$\sup_{i\in { [N]}, u\in \cA_i}\|u\|_{\cH^4(\sR)}\le L$,
$\max_{i\in { [N]}}|x_i|\le L$,
and 
$\max_{i\in { [N]}} \|\sigma_i\|_{L^\infty} \le L$.
Assume further that 
there exists $f_0:[0,T]\times \sR^N\times \sR^N\to \sR$ and $g_0:\sR^N\to \sR$
such that
for all $i\in { [N]}$,
$b_i$,
$f_i$ and $g_i$ are of the following form:
\begin{align}
b_i(t,x_i,x) &= \bar{b}_i\left(t,x_i,\frac{1}{N}\sum_{\ell=1}^N \delta_{x_\ell}\right),
\label{eq:mean_field_open_b}
\\
f_i(t, x,u)&= f_0(t,x,u)+c_i( u_i)+\bar{f}_i\left(t,\frac{1}{N}\sum_{\ell=1}^N \delta_{(x_\ell,u_\ell)} \right),
\quad 
g_i(x) = g_0(x) +  \bar{g}_i\left( \frac{1}{N}\sum_{\ell=1}^N \delta_{ x_\ell } \right),
\label{eq:mean_field_open_cost}
\end{align}
where 
$\bar{b}_i: [0,T]\times \sR\times  \cP_2(\sR)\to \sR$,
$c_i:\sR\to \sR$, 
$\bar{f}_i:[0,T]\times \cP_2(\sR\times \sR)\to \sR$,   and $\bar{g}_i:\cP_2(\sR)\to \sR$
are   twice continuously differentiable   
with  bounded  second-order derivatives (uniformly in $N$). 
Then $\cG_{\rm diff}  $ is an $\alpha$-potential game with $\alpha \le  {C}/{N}$, for a  constant   $C\ge 0$  independent of $N$. 

\end{example}

Example \ref{ex:mean_field_open} follows from the fact that by \eqref{eq:mean_field_open_b} and \eqref{eq:mean_field_open_cost}, 
$|(\partial_{x_h} \Delta^f_{i,j})(t,0,0)| + |(\partial_{x_h} \Delta^g_{i,j} )(0)| \le C/N$
and 
\begin{align*}
    &    |(\partial^2_{x_h x_\ell} \Delta^f_{i,j})(t,x,u)|
      +|(\partial^2_{x_h u_\ell} \Delta^f_{i,j})(t,x,u)|
      +|(\partial^2_{x_h x_\ell} \Delta^g_{i,j})(x)|
      \le C \left( \frac{1}{N}
 \mathds{1}_{h =\ell}
      +\frac{1}{N^{2} }  
 \mathds{1}_{h \not =\ell}\right)
\end{align*}
for some constant $C\ge 0$ independent of $N$ (see Remark \ref{rmk:mean_field}),
which yields the 
   bound of $\alpha$ due to   Theorem \ref{thm:value_jacobian_open}.

\begin{remark}
\label{rmk:mean_field_interaction}
Example \ref{ex:mean_field_open}    allows all players to have 
different admissible control sets  $\cA_i$,    and heterogeneous dependencies on the empirical measures of the joint state and control profiles. 
This is in contrast to the classical $N$-player mean field games with symmetric and homogeneous players (see  \cite{carmona2018probabilistic}).

Note  that even if all players have homogeneous coefficients, the conditions  in Example \ref{ex:mean_field_open}  differ from those   for potential mean field  games (MFGs) introduced in \cite{lasry2006jeux, lasry2007mean, cardaliaguet2017learning}.
An MFG is considered potential if there exists an optimal control problem whose optimal trajectories coincide with the equilibria of the MFG. This is a weaker condition than the notion of   
$N$-player potential game described in \cite{monderer1996potential}, as it is a local property that concerns only the minimizer of the potential function.

In contrast,  Example \ref{ex:mean_field_open} allows  the $\alpha$-potential function to     control  the derivatives of each player's { objective function} globally, with an error of order 
$\cO(1/N)$ as $N\to\infty$. 
This property is crucial for ensuring the convergence of gradient-based learning algorithms (see \cite{guo2023markov} and references therein).
Consequently, when 
$b_i$ depends on the empirical measure of the states, we require the cost functions 
$\bar f_i$ and $\bar g_i$ in \eqref{eq:mean_field_open_cost} to depend on the state and controls only through their empirical measures.
 Assuming the uniqueness of Nash equilibria in MFGs, the minimum of the 
$\alpha$-potential function (with appropriate scaling) converges to the minimum of the mean field potential function as $N\to \infty$, provided that sufficient conditions are met to allow the interchangeability of minimization and the limit.

\end{remark}
  \section{Conditional McKean-Vlasov control problem for $\alpha$-NE}  
\label{sec:MKV_DPP}
 Given an $\alpha$-potential function  $\cG_{\rm diff}$, 
this section establishes a dynamic programming  approach to minimize 
  the $\alpha$-potential function $\Phi$  in \eqref{eq:potential_differential} 
  over $\mathcal A^{(N)}$, { where the state process is governed by \eqref{eq:X_i_general}}. 
  The main difficulty is that 
  the   objective   \eqref{eq:potential_differential} depends   on the  aggregated behaviour of the state processes
with respect to $r\in [0,1]$,
which   acts as    an additional noise independent of the Brownian motion  $W$. 
Meanwhile, the admissible controls in  $\cA^{(N)}$  are adapted to a smaller filtration $\sF$ that
depends only on $W$ but 
 is independent of $r$.
To apply the  dynamic programming approach,
we embed the optimization problem into a  suitable conditional McKean--Vlasov (MKV) control problem. 

\subsection{Conditional MKV control problem}\label{ssec: ConditionalMKV} 
We start with some necessary   notation:
For each 
$t, s \in [0,T]$, 
let  $W^t_s\coloneqq W_{s\vee t} -W_t$ be the Brownian increment
  after time $t$,
  and let 
 the filtration $\sF^{t}$ 
 be the $\sP$-complement of 
the filtration generated by  
$W^t=(W^t_s)_{s\ge 0}$. 
Note that $\sF^0$ coincides with $\sF$. 
For each Euclidean space $E$, 
we denote by $\cP_2(E)$ 
the set  of   probability measures   $\mu$ on $E$ with finite second moment,
i.e., $\|\mu\|_2^2\coloneqq \int_E |x|^2\mu(\d x)<\infty$. 
The space $\cP_2(E)$ is equipped with the $2$-Wasserstein distance.
We assume without loss of generality (see e.g., \cite{cosso2023optimal}) that   there exists a sub-$\sigma$-field $\cG\subset \cF$, 
 which is independent of $W$  and is  ``rich enough" in the sense that 
 $\cP_2(\cS)=\{\cL(\xi)\mid \xi\in L^2( \cG; \cS ) \} $,
 where 
 $\cL(\xi)$ denotes  the distribution of $\xi$ under $\sP$,
   $\cS \coloneqq \sR^{(N+1)Nd}\times [0,1]$, 
 and  $L^2(\cG; \cS)$ is the set of  $\cS$-valued  $\cG$-measurable  square integrable random
variables on  $(\Omega, \cF,\sP)$. 
We define $\sG\coloneqq  (\cG_t )_{t\in [0,T]}$ to be 
   the    filtration generated by   $W$,
  augmented with $\cG$ and $\sP$-null sets.

Now we introduce the MKV control problem associated with \eqref{eq:potential_differential} and \eqref{eq:X_i_general}.
The state process of the MKV control problem takes values in $\cS \coloneqq \sR^{(N+1)Nd}\times [0,1]$,
encompassing the original state process $\bX^{\bu}$, 
  the sensitivity processes $(\bY^{\bu, u_i})_{i\in { [N]}}$,
  and the additional 
 parameter $r$.
More precisely, 
let $A=\prod_{i=1}^N A_i$, and 
fix
$z =(z_i)_{i\in { [N]}}\in A$. 
For each $t\in [0,T]$,  let $\cA^t$ be  the set of 
$\sF^t$-progressively measurable square integrable processes taking values in $A$.
Let $\cP_2^{\textnormal{Unif}}(\cS)$ be the space     of  measures     $\nu\in \cP_2(\cS)$ 
 whose marginal $\nu\vert_{[0,1]}$ on $[0,1]$ is the uniform distribution: 
 \begin{equation*}
     \cP_2^{\textnormal{Unif}}(\cS)\coloneqq 
     \{\nu\in \cP_2(\cS)\mid \nu\vert_{[0,1]}=\operatorname{Unif}(0,1)\}.
 \end{equation*}
For each $\nu\in \cP_2^{\textnormal{Unif}}(\cS)$, $\bu \in \cA^t$, 
and  $(\xi,\mathfrak r)\in L^2(\cG; \cS)  $   with  
 $\cL(\xi, \mathfrak r)=\nu$, 
  consider 
the   process 
$\sX^{t,\xi,\mathfrak r,\bu}$ governed by  the following dynamics,
which concentrates   the state process  \eqref{eq:X_i_general} 
and the sensitivity processes 
\eqref{eq:Y_h_i_general}: \begin{align}
\label{eq:state_lifted}
\sX^{t,\xi,\mathfrak r,\bu}_s=\xi+\int_t^s
B(v,\sX^{t,\xi,\mathfrak r,\bu}_{v},  \mathfrak r, \bu_{v})\d v + 
\int_t^s \Sigma(v,\sX^{t,\xi,\mathfrak r,\bu}_{v},  \mathfrak r, \bu_{v})\d W^t_v,
\quad s\in [t,T],
\end{align}
where  $B= \vcat(B_1,\ldots , B_{N+1}) : [0,T]\times \cS \times A\to \sR^{(N+1)Nd}$
is defined by: 
for all $t\in [0,T]$,  $\sx=\vcat(x, y_1,\ldots, y_N)\in \sR^{(N+1)Nd}$, $r\in [0,1]$, and $u=(u_i)_{i\in { [N]}}\in A$,  
$B_1(t,\sx, r, u)\coloneqq 
\vcat(
b_1(t,x , u), \ldots,  
b_N(t,x , u)
)
$  and 
\begin{align*}
B_{i+1}(t,\sx, r, u)\coloneqq 
\begin{pmatrix}
(\partial_x{b}_1) \left(t,  x, z+r(u-z) \right) y_i
     + (\partial_{u_i}{b}_1)\left(t, x, z+r(u-z) \right) (u_i-z_i)
\\
\vdots
\\
(\partial_x{b}_N) \left(t,  x, z+r(u-z) \right) y_i
     + (\partial_{u_i}{b}_N)\left(t, x, z+r(u-z) \right) (u_i-z_i)
\end{pmatrix},
\quad  \forall i\in { [N]},
\end{align*}
and  
 $\Sigma =(\Sigma_1,\ldots, \Sigma_m):[0,T]\times \cS \times A\to  \sR^{(N+1)Nd\times m}$ 
is defined such that for all $k=1,\ldots, m$, $\Sigma_k$ is defined in the same way as   $B$, but with $b_i$ replaced by $\sigma_{ik}$ for all $i\in { [N]}$. 
 Under Assumption \ref{assum:regularity_general}, 
 $(\sX^{t,\xi,\mathfrak r,\bu}, \mathfrak r)$ is a uniquely defined $\cS$-valued $\sG$-adapted 
 square integrable process.
Moreover, 
as $\mathfrak r $ is independent of $\cF^{t}_s$ and is stationary in time, 
the conditional law 
$\mu^{t,\xi,\mathfrak r,\bu}_s  \coloneqq \cL(\sX^{t,\xi,\mathfrak r,\bu}_s,\mathfrak r | \cF^t_s)$, $s\in [t,T]$,
 is a  $\cP_2^{\textnormal{Unif}}(\cS)$-valued $\sG$-optional process  (see \cite[Lemma A.1]{djete2022mckean}).

Consider     the  following   cost functional,
which is a dynamic version of the $\alpha$-potential function \eqref{eq:alpha_general}:
 \begin{align}
 \label{eq:cost_lifted}
 J(t,\xi,\mathfrak r,\bu)
   &\coloneqq \mathbb{E} \bigg[\int_t^T
  \left\langle  
    F(s, \cdot,\cdot, \bu_s ),  \mu^{t,\xi,\mathfrak r,\bu}_s \right\rangle
\mathrm{d} s 
 +
   \left\langle  
 G ,  \mu^{t,\xi,\mathfrak r,\bu}_T \right\rangle\bigg],
 \end{align}
 where 
$\mu^{t,\xi,\mathfrak r,\bu}_s  = \cL(\sX^{t,\xi,\mathfrak r,\bu}_s,\mathfrak r | \cF^t_s)$,
  $F: [0,T]\times \cS\times A\to \sR$ 
 and $G:  \cS \to \sR$ 
 are defined by: 
for all $t\in [0,T]$, $\sx=\vcat(x,y_1,\ldots, y_N)  \in \sR^{(N+1)Nd}$,
$r\in [0,1]$
and $u= (u_i)_{i\in { [N]}} \in A$,
\begin{align}
\label{eq:F_G_lift}
F(t, \mathbbm{x}, r, u ) \coloneqq  \sum_{j=1}^N 
    \left(\begin{array}{c}
y_j \\
 u_j-z_j
\end{array}\right)^{\top}\left(\begin{array}{c}
\partial_x f_j \\
\partial_{u_j} f_j
\end{array}\right)\left(t, x, z+r(u-z)\right),
\quad 
G( \mathbbm{x},r) \coloneqq\sum_{j=1}^N    y_j^\top \left(\partial_x g_j\right)(x)\,,
\end{align}
where   $\langle h , \mu\rangle$  denotes   the integral of the function $h$ with respect to the measure $\mu$.

The following proposition identifies 
 minimizing the $\alpha$-potential function $\Phi$ in 
\eqref{eq:potential_differential}
as solving  an MKV control problem 
with a specific initial condition.
The result relies on the crucial observation that 
  the  cost functional $J$ in \eqref{eq:cost_lifted} satisfies   the
law invariance property \cite{cosso2019zero,djete2022mckean}, i.e., it 
  depends on the law of $(\xi,\mathfrak r)$ instead of  the specific choice of   the random variable $(\xi,\mathfrak r)$ itself.

 \begin{prop}
 \label{prop:law_invariance}
Suppose Assumption \ref{assum:regularity_general} holds.
Let  $J$ be  defined in  \eqref{eq:cost_lifted}. 
For all  $(t,\nu)\in [0,T]\times \cP^{\textnormal{Unif}}_2(\cS) $,    $\bu\in \cA^t$, 
 and   
 $(\xi,\mathfrak r), (\xi',\mathfrak r')\in L^2(\cG;\cS)$ with law $\nu$,  
 $J(t,\xi,\mathfrak r,\bu) =  J(t,\xi',\mathfrak r',\bu)$.
 
Consequently, 
the { optimal value function} for minimizing  
\eqref{eq:cost_lifted} 
can be identified as
   $V: [0,T]\times \cP^{\textnormal{Unif}}_2(\cS)\to \sR\cup\{-\infty, \infty\}$: 
   \begin{equation}
   \label{eq:MV_t_nu}
   V(t,\nu) \coloneqq \inf_{\bu \in \cA^t }  J(t,\xi,\mathfrak r,\pmb u),
   \end{equation}
   for any 
   $(\xi,\mathfrak r)\in L^2(\cG; \cS)  $   with  
 $\cL(\xi, \mathfrak r)=\nu$.
 Moreover, let 
 $\Phi$ be defined in 
\eqref{eq:potential_differential}, 
 we have 
 \begin{equation}
\label{eq:MV_potential}
V\left(0,\delta_{\vcat(x_1,\ldots, x_N, {0}_{N^2d})}\otimes \operatorname{Unif}(0,1)\right) = \inf_{\bu \in \cA^{(N)} }\Phi(\bu),
\end{equation}
where  $x_i$ is the initial state of \eqref{eq:X_i_general} and 
${0}_{N^2d}\in \sR^{N^2d}$ is the zero vector. 

 \end{prop}

The law invariance of $J$ follows from the fact that   
each   $\bu\in \cA^t$ 
is adapted to the filtration of $W$, and    is  independent of   $\cG$.
Then by the strong uniqueness of  \eqref{eq:state_lifted},
it holds that for all 
$(\xi,\mathfrak r), (\xi',\mathfrak r')\in L^2(\cG;\cS)$ with law $\nu$, 
$\cL(\sX^{t,\xi,\mathfrak r,\bu}, \bu)=\cL(\sX^{t,\xi',\mathfrak r',\bu}, \bu)$,
and hence  $J(t,\xi,\mathfrak r,\bu) =  J(t,\xi',\mathfrak r',\bu)$
(see \cite[Proposition 2.4]{djete2022mckean}). 
The identity \eqref{eq:MV_potential}
follows from 
$\cA^0=\cA^{(N)}$ and  by the     law of iterated conditional expectations,  
$\Phi(\bu) =J(t,\xi,\mathfrak r,\pmb u)$  with      $ \xi=\vcat(x_1,\ldots, x_N, {0}_{N^2d})$ and  
      a uniform random variable $\mathfrak r\in L^2(\cG; [0,1]) $.

We remark that \eqref{eq:MV_t_nu} is a specific   stochastic control problem with 
    conditional    MKV  dynamics, 
   where
    the  state  
   \eqref{eq:state_lifted} does not involve law dependence, and 
   the cost functions \eqref{eq:cost_lifted} depend linearly on the conditional distribution.
As a result, the dynamic programming approach, developed for general MKV control problems in  \cite{pham2017dynamic,   djete2022mckean}, 
can be applied to minimize the $\alpha$-potential function $\Phi$.  

\subsection{HJB equation for the $\alpha$-potential function}\label{ssec: HJB} 
In the section, we identify the { optimal value function} \eqref{eq:MV_t_nu} as a solution of  an 
HJB equation. 
We will adopt the notion of linear 
 derivative   with respect to probability measures 
 as in \cite{guo2023ito, de2024mean, guo2024ito}, as it allows for a clear distinction between the derivatives with respect to the marginal laws of $\sX^{t,\xi,\mathfrak r,\bu}$ and  $\mathfrak r $; see Remark \ref{rmk:linear_derivative}. 
 
Specially, we say 
a function  $\phi:[0,T]\times \cP_2(E)\to \sR$ 
is in 
$C^{1,2}( [0,T]\times \cP_2(E))$
if there exist continuous functions 
$ \frac{\delta \phi}{\delta \mu}: [0,T]\times \cP_2(E)\times E \to  \sR$
and  
 $  \frac{\delta^2 \phi}{\delta \mu^2}: [0,T]\times \cP_2(E)\times E\times E \to  \sR$ 
 such that $\frac{\delta^2 \phi}{\delta \mu^2}$ is symmetric in its last two  arguments and the following properties hold:
\begin{itemize}
    \item continuously differentiable:
$\partial_t \phi(t,\mu)$, 
$\partial_v \frac{\delta \phi}{\delta \mu}(t,\mu, v)$,
$\partial^2_{vv} \frac{\delta \phi}{\delta \mu}(t,\mu, v)$
and 
$ \partial^2_{vv'} \frac{\delta^2 \phi}{\delta^2 \mu}(t,\mu, v,v')$ exist and are continuous in $(t,\mu, v,v')$.
\item 
 locally uniform bound: for any compact $K\subset \cP_2(E)$, there exists    $c_K\ge 0$ such that for all $(t,\mu)\in [0,T]\times K$ and $v,v'\in E$, 
 $
 |\partial_v \frac{\delta \phi}{\delta \mu}(t,\mu, v)| \le c_K(1+|v|)$,
 $
 |\partial^2_{vv} \frac{\delta \phi}{\delta \mu}(t,\mu, v)|+ |\partial^2_{vv'} \frac{\delta^2 \phi}{\delta^2 \mu}(t,\mu, v,v')|\le c_K$.
 \item 
fundamental theorem of calculus: for all $\mu,\nu\in \cP_2(E)$ and $t\in [0,T]$, 
\begin{align*}
    \phi(t,\mu)-\phi(t,\nu)&=\int_0^1 \int_E \frac{\delta \phi}{\delta \mu}(t, \lambda \mu+(1-\lambda)\nu, v)(\mu-\nu)(\d v)\d \lambda, 
    \\
     \phi(t,\mu)-\phi(t,\nu)
     &- \int_E \frac{\delta \phi}{\delta \mu}(t,   \nu, v)(\mu-\nu)(\d v)
     \\
     &=\int_0^1\int_0^r \int_{E\times E} \frac{\delta^2 \phi}{\delta^2 \mu}(t, s \mu+(1-s)\nu, v,v')(\mu-\nu) (\d v )(\mu-\nu) ( \d v') \d s\d r. 
\end{align*}
 \end{itemize}
For each $u\in A$,
 $\phi\in C^{1,2}( [0,T]\times \cP_2(\cS))$, $t\in [0,T]$, and $\mu\in  \cP_2(\cS)$, define the function $  \mathbb L^u \phi (t, \mu):\cS\to \sR$ by 
 \begin{align}
 \label{eq:L_u}
 \mathbb L^u \phi (t,\mu) (\sx ,r)
 \coloneqq 
B(t,\sx,  r, u)^\top  \partial_{\sx}\frac{\delta \phi}{\delta \mu} (t, \mu,  \mathbbm{x},  r )
+\frac{1}{2}\operatorname{tr}\left(   
   (\Sigma \Sigma^\top)(t, \mathbbm{x} ,r , u)
\partial^2_{\sx\sx}\frac{\delta \phi}{\delta \mu}  (t, \mu, \sx ,r  )
\right)
 \end{align}
with $B$ and $\Sigma$   in \eqref{eq:state_lifted}, 
 and define 
 the function $  \mathbb M^u \phi (t,\mu):\cS\times \cS\to \sR$ by 
 \begin{align}
 \label{eq:M_u}
 \mathbb M^u \phi (t,\mu) (\sx ,r, \sx', r')
 \coloneqq 
 \frac{1}{2}\operatorname{tr}\left(  
 \Sigma (t, \mathbbm{x} ,r , u) 
  \Sigma(t, \mathbbm{x}' ,r' , u)^\top 
\partial^2_{\sx\sx'}\frac{\delta^2 \phi}{\delta^2 \mu} 
 (t,\mu, \sx ,r ,\sx',r' )
\right).
 \end{align} 
 Note that under Assumption \ref{assum:regularity_general},
 $ \mathbb L^u \phi (t, \mu)\in L^1(\cS,\mu)$ 
 and 
  $ \mathbb M^u \phi (t, \mu)\in L^1(\cS\times \cS, \mu\otimes \mu)$. 
Define the Hamiltonian 
\begin{equation}\label{eq: hamiltonian}
     \hat H (t,\mu, \phi, u)\coloneqq  
  \langle   \mathbb L^u \phi (t,\mu) ,\mu\rangle
  +
    \langle   \  \mathbb M^u \phi (t,\mu)
 ,\mu\otimes \mu \rangle
  + \langle F(t,\cdot, \cdot, u)  ,\mu\rangle,
\end{equation}
 with   $F$   defined in \eqref{eq:F_G_lift}.

\begin{remark}
\label{rmk:linear_derivative}
    As $\mathfrak r$ is stationary in \eqref{eq:state_lifted}, the operators $\mathbb L^u $
    and $\mathbb M^u $
    only involve the partial derivative with respect to the $\sx$-component, and not  the  derivative with respect to the $r$-component. 
    One can equivalently express these operators using the Lions   derivatives as in \cite{pham2017dynamic}. Indeed, let $\partial_\mu\phi$ be the Lions derivative of $\phi$, 
    \begin{align*}
& \mathbb L^u \phi (t,\mu) (\sx ,r)
 =
 \begin{pmatrix} B(t,\sx,  r, u)\\ 0 \end{pmatrix}^\top  \partial_{\mu}\phi (t, \mu)( \mathbbm{x},  r )
+\frac{1}{2}\operatorname{tr}\left(   \begin{pmatrix}  (\Sigma \Sigma^\top) (t, \mathbbm{x} ,r , u)&  0_{  N(N+1)d}
   \\
   0_{  N(N+1)d}^\top & 0
   \end{pmatrix}
\partial_{(\sx,r)}\partial_{\mu } \phi(t, \mu)(\sx ,r  )
\right),
 \end{align*}
due to the relation 
$\partial_\mu \phi =\partial_{(\sx ,r)}\frac{\delta \phi}{\delta \mu} $
(see     \cite[Proposition 5.48]{carmona2018probabilistic}). Similar expression holds for  $\mathbb M^u \phi$. 
We adopt the expressions \eqref{eq:L_u} and 
\eqref{eq:M_u}
to simplify the notation. 
\end{remark}

We now  present  a verification theorem, which   constructs  
an optimal control of \eqref{eq:MV_t_nu} (and \eqref{eq:potential_differential})
in an 
   analytic feedback
form  using a smooth solution to an HJB equation  
in the Wasserstein space.

\begin{theorem} \label{thm: verification}
Suppose   Assumption \ref{assum:regularity_general} holds.
Let    $v\in C^{1,2}( [0,T]\times \cP_2(\cS))$ be such that for a  constant $C\ge 0$,
$$
|v(t,\mu)|\le C(1+\|\mu\|_2^2),\quad 
 \left|\partial_{(\sx, r)}\frac{\delta v}{\delta \mu}(t,\mu ,\sx ,r )\right|\le C(1+|\sx|+\|\mu\|_2),
 \quad (t,\mu) \in [0,T]\times \cP_2^{\textnormal{Unif}}(\cS), (\sx, r)\in \cS.
 $$ 
Assume that 
$\inf_{u\in A}  \hat H (t,\mu, v, u)\in \sR$ for all  $(t,\mu)$, and 
$v$ satisfies the following HJB equation:
 \begin{equation}\label{eq:HJB}
  \left\{
   \begin{aligned}
\partial_t w (t, \mu)  + \min_{u\in A}  \hat H (t,\mu, w, u)
 &=0, \quad (t,\mu)\in [0,T)\times \cP_2^{\textnormal{Unif}}(\cS),
 \\ 
w(T,\mu) &  =  \langle G,  \mu\rangle, \quad \mu \in \cP_2^{\textnormal{Unif}}(\cS).
\end{aligned}\right.
 \end{equation}
 Assume further  that there exists a measurable map 
$\hat a :[0,T]\times \cP_2^{\textnormal{Unif}}(\cS)\to A$
such that 
for all $(t,\mu)\in [0,T]\times \cP_2^{\textnormal{Unif}}(\cS)$,
\begin{equation}\label{eqn: hat_a_in_DPP}
    \hat a (t,\mu) \in \arg\min_{u\in A}  \hat H (t,\mu, v, u),
\end{equation}
  for any  
     $(\xi,\mathfrak r)\in L^2(\cG;\cS)$
with law $\mu $,  the following  
equation 
\begin{align}
\hat \sX_s=\xi+\int_t^s
B\left(v, \hat \sX_{v},  \mathfrak r, \hat a \big(v, \cL(\hat\sX_{v}, \mathfrak r \mid \cF^t_v)\big)\right)\d v + 
\int_t^s \Sigma\left(v, \hat \sX_{v},  \mathfrak r, \hat a \big(v, \cL(\hat\sX_{v}, \mathfrak r \mid \cF^t_v)\big) \right) \d W^t_v,
\quad s\in [t,T]
\end{align}
admits a square integrable solution 
$\hat{\sX}^{t,\xi,\mathfrak r}$,     
 and   the feedback control 
 $\hat\bu^{t,\xi,\mathfrak r}_s\coloneqq  \hat a (s, \cL(\hat{\sX}^{t,\xi,\mathfrak r}_{s}, \mathfrak r \mid \cF^t_s)) $,
 $s\in [t,T]$,
 is in $\cA^t$. 
 Then 
 $v = V$ on $[0,T]\times \cP_2^{\textnormal{Unif}}(\cS)$,  
 and for all $(t,\mu)\in [0,T]\times \cP_2^{\textnormal{Unif}}(\cS)$,
$\hat\bu^{t,\xi,\mathfrak r}\in \cA^t$
with     $\cL(\xi,\mathfrak r)=\mu$
is an optimal control for $V(t,\mu)$.
   
   Consequently, 
   given $ \xi=\vcat(x_1,\ldots, x_N, {0}_{N^2d})$ and  
      a uniform random variable $\mathfrak r\in L^2(\cG; [0,1]) $,
the control 
   $\hat\bu^{0,\xi,\mathfrak r}\in \cA^{(N)}$ minimizes the $\alpha$-potential function 
   $\Phi$ given in 
\eqref{eq:potential_differential},
thus is an $\alpha$-Nash equilibrium of the game $\cG_{\rm diff} $ in Section \ref{sec:general_stochastic_game}.

  \end{theorem}

  Theorem \ref{thm: verification} only requires the function $v$ to satisfy the HJB equation \eqref{eq:HJB} on the subspace   $\cP_2^{\textnormal{Unif}}(\cS)$, rather than on the entire space $\cP_2(\cS)$ as is the case for general MKV control problems \cite{pham2017dynamic, djete2022mckean}. This is due to the fact that 
the flow  
$  (\mu^{t,\xi,\mathfrak r,\bu}_s )_{s\in [t,T]}= (\cL(\sX^{t,\xi,\mathfrak r,\bu}_s,\mathfrak r | \cF^t_s))_{s\in [t,T]}$  remains in  $\cP_2^{\textnormal{Unif}}(\cS)$ for any control $\bu\in\cA^t$. 
Restricting the domain of  \eqref{eq:HJB} to 
$\cP_2^{\textnormal{Unif}}(\cS)$ 
is essential for minimizing  \eqref{eq:potential_differential}  analytically  in linear-quadratic games; see Section \ref{sec:example}.

{
 Theorem \ref{thm: verification} 
   adapts \cite[Theorem 4.2]{pham2017dynamic}
to the present setting.
Compared with \cite{pham2017dynamic},
since we do not assume the compactness of the action space, we introduce the additional assumption 
on the finiteness of  
$\inf_{u\in A}  \hat H (t,\mu, v, u)$.
With this assumption in place, the  proof
follows     
directly along the same lines as the same  lines of  the  verification theorem \cite[Theorem 4.2]{pham2017dynamic}.
}
 Indeed, fix $(t,\mu)\in [0,T]\times \cP_2^{\textnormal{Unif}}(\cS)$ and 
     $(\xi,\mathfrak r)\in L^2(\cG;\cS)$
with law $\mu$.
 For any  
  $\bu \in \cA^t$, 
     applying  It\^{o}'s formula   in \cite{guo2024ito} (see also \cite[Theorem 4.17]{carmona2018probabilisticII})
to $s\mapsto v(s, \cL(\sX^{t,\xi,\mathfrak r,\bu}_{s}, \mathfrak r \mid \cF^t_s))$ and 
using the fact that $\cL(\sX^{t,\xi,\mathfrak r,\bu}_{s}, \mathfrak r \mid \cF^t_s)$  lies in $   \cP_2^{\textnormal{Unif}}(\cS)$ and the condition \eqref{eq:HJB} of $v$ yield that 
 $v(t,\mu)\le J(t,\xi,\mathfrak \tau ,\bu) $,
 which implies that $v(t,\mu)\le V(t,\mu)$. 
{ The condition \eqref{eqn: hat_a_in_DPP} of $\hat a$ 
and the assumption 
$\hat\bu^{t,\xi,\mathfrak r}\in \mathcal A^t$
imply the optimality of $\hat\bu^{t,\xi,\mathfrak r}$.}
A special case of Theorem \ref{thm: verification} for a class of linear-quadratic games  is presented   in   Theorem  \ref{thm:lq_verification},
 with a detailed proof provided.

  }

 \section{A toy example: linear-quadratic $\alpha$-potential games}
\label{sec:example}
 
 In this section, we illustrate our results through a simple  linear-quadratic (LQ) game $\cG_{\textrm{LQ}}$ on a   undirected graph $G = (V, E)$. The vertex of the graph is the set of players $V = { [N]}$, 
 and each edge
  between the vertices  represents   a dependency between the associated players.
   The { objective function} of player $i$ in this game is given by
\begin{align}
\label{eq: cost_ex}
V_i(\bu)=\sE\left[\int_0^T \left( 
 u_{i,t}^2 + \frac{1}{N} \sum_{j=1}^N q_{ij}  (X^{\bu}_{i,t}-X^{\bu}_{j,t})^2 \right)\,\d t+ {\gamma_i} (X^{\bu}_{i,T} - d_i)^2\right],
\end{align}
where 
$q_{ij},\gamma_i\ge 0$, $d_i\in \sR$,  and 
for any 
$\bu =(u_i)_{i\in { [N]}}\in \cH^2(\sR^{N})$, the state process  $\bX_t^\bu$ is governed by:  
\begin{equation}\label{eq: dynamics_ex}
    \begin{aligned}
        \d X_{i, t} = \left( a_i(t) X_{i,t} +   u_{i, t}  \right)\d t +  \sigma_i(t) \d W^i_{ t},  \quad  X_{i,0} =x_i, \quad t \in [0,T],
        \; i\in { [N]},
    \end{aligned}
\end{equation}
where $x_i\in \sR$,    $a_i,  \sigma_i: [0,T]\to \sR$ are  (possibly distinct) continuous  functions.  
Player $i$'s aims to minimize \eqref{eq: cost_ex} over the control set 
\begin{equation}
\label{eq:control_ex}
\cA_i=\{u_i\in \cH^2(\sR)|\|u\|_{\cH^2(\sR)}\le L \},
    \end{equation}
where $L>0$ is a given 
 sufficiently large constant.

 The above  game 
 can be viewed as  
 a crowd flocking game \cite{lachapelle2011mean, aurell2018mean,  carmona2023synchronization}. The goal is for all players to reach their respective destinations by a specified terminal time. During the game, players exhibit a tendency to group together, mimicking the collective behavior observed in natural flocks or herds. This phenomenon, known as flocking, is driven by factors such as safety, efficiency, and social interaction.

\subsection{Quantifying $\alpha$ for $\cG_{\textrm{LQ}}$}
\label{sec:alpha_crowd}
Since    the  dynamics \eqref{eq: dynamics_ex}
is decoupled, 
 Theorem \ref{thm:value_jacobian_open}
 and Remark \ref{rmk:alpha_bound} imply  that    $\cG_{\textrm{LQ}}$ is an $\alpha_N$-potential game with  
\begin{equation}
\label{eq:alpha_N}
        \alpha_N \leq C \frac{1}{N}  \max_{i\in { [N]}}\sum_{j\ne i} |q_{ji}- q_{ij}|.
\end{equation}
Suppose that  the constants 
 $(q_{ij}, \gamma_i, d_i, x_i)_{i,j\in { [N]}}, L$ and the sup-norms of  $(a_i,\sigma_i)_{i\in { [N]}}$ 
  are uniformly bounded in 
 $N$.
 Then, an explicit bound for 
$\alpha_N$ in terms of the number of players 
$N$ and the strength and symmetry of players' interactions  can be obtained, as illustrated below:

\begin{itemize}
    \item \textbf{Symmetric interaction.}
If the interaction weights \((q_{ij})_{i,j \in { [N]}} \) satisfy the \emph{pairwise} symmetry condition \(q_{ij} = q_{ji}\) for all $i,j\in { [N]}$, then \(\mathcal{G}_{\textrm{LQ}}\) is a potential game, i.e.,  $\alpha_N\equiv  0$ regardless of the number of players \(N\). This symmetry condition is common in many interaction kernels, where player \(i\)'s influence on player \(j\) depends only on the distance between them \cite{aurell2018mean, carmona2018probabilistic, aurell2022stochastic}.  
\item \textbf{Asymmetric interaction.} 
Suppose that 
$G$ be a   sparse graph 
with a bounded degree
$\max_{i\in V}\deg(i) = k$
for some 
$k \geq 2$,
i.e., 
each vertex is connected to at most $k$ vertices.
Assume further  that the interaction weights $(q_{ij})_{i,j \in [N]}$ exhibit an exponential decay of the form
\begin{equation}
\label{eq:exponential_decay}
q_{ij} = w_i \eta^{c(i,j)}, \quad \forall i,j\in[N],
\end{equation}
where $(w_i)_{i\in { [N]}}$ are  distinct
   positive constants that are uniformly bounded in 
$N$,   $\eta\in (0,1)$
is a given constant,
and 
$  c(i,j)$ is   the 
(shortest-path) distance between vertices $i$ and $j$.
Such a structure models localized interactions, where a player's impact is strongest on their immediate neighbors and diminishes further away \cite{gamarnik2013correlation,gamarnik2014correlation, gu2025mean}. 
For clarity of exposition, we assume a sufficiently fast decay rate $\eta$  satisfying $\eta < 1/k$.

In this setting,  
by   \eqref{eq:alpha_N},
there exists a constant $C\ge 0$, independent of $\eta$, $k$ and $N$, such that 
\begin{align}
\label{eq:exponential_decay_alpha}
    \alpha_N \leq \frac{C}{N} \max_{i \in [N]} \sum_{j \ne i} \eta^{c(i,j)}
   \leq \frac{C}{N} \sum_{\ell=1}^{\infty} \eta^{\ell} k^\ell = 
    C\frac{   \eta k}{(1-\eta k)N},
\end{align}
where 
the second inequality used the fact  that, for any vertex \( v \in V \), the number of vertices at distance \( \ell \) from \( v \) is at most \( k^\ell \).
The bound \eqref{eq:exponential_decay_alpha} demonstrates that \(\alpha_N\) decays to zero as the number of players increases. Additionally,  \(\alpha_N\) vanishes as \(\eta \to 0\), reflecting the weakening interactions among players.

 \end{itemize}

\subsection{Constructing $\alpha$-NE for $\cG_{\textrm{LQ}}$}
An  $\alpha_N$-NE of 
 $\cG_{\textrm{LQ}}$ can be constructed by minimizing the corresponding $\alpha_N$-potential function \eqref{eq:potential_differential}.
 The structure of $\cG_{\textrm{LQ}}$ significantly simplifies the $\alpha_N$-potential function 
 compared to  the general   case studied in Sections \ref{sec:general_stochastic_game} and \ref{sec:MKV_DPP}. 
 Indeed, 
 as $X^{\bu}_{i}$ depends only on $u_i$, the sensitivity processes 
 $Y^{\bu, u'_i}_{t,j}\equiv 0$ for $i\not =j$,  reducing the dimension of the state process in \eqref{eq:potential_differential}  from $\cO(N^2)$ to $\cO(N)$.  Moreover, due to the LQ structure \eqref{eq: cost_ex}-\eqref{eq: dynamics_ex}, the 
$\alpha$-potential function becomes a LQ control problem, whose minimizer can be solved analytically. 
    
We consider     an extended state 
 dynamics  including both   the original   state dynamics \eqref{eq: dynamics_ex}  for $\bX^{\bu}$, and the dynamics for 
  the sensitivity processes $(Y^{\bu, u_i}_i)_{i\in { [N]}}$. Specifically, 
fix 
 a uniform random variable $\mathfrak r\in L^2(\cG; [0,1]) $, and
for each $\bu \in \cH^2(\sR^N)$, 
consider   the
 $\sR^{2N}$-valued $\sG$-adapted 
 square integrable process  $\sX^{  \mathfrak{r},\bu} $ 
 governed by 
\begin{equation}\label{eq:LQ_lifted}
  \d \sX_t = \left( A(t) \sX_t + \cI_{\mathfrak{r}} \bu_t \right) \d t + \Sigma(t) \d W_t,\quad 
  \sX_0  = \vcat(x_1, \cdots, x_N, 0_N),
\end{equation}  
where  $\cI_\mathfrak{r} \coloneqq \vcat(\mathfrak{r}\sI_N, \sI_N) \in \sR^{2N \times N}$,  $A(t) \coloneqq \diag(\tilde{A}(t),\tilde{A}(t)) \in \sS^{2N}$ with 
$\tilde{A}(t) \coloneqq \diag(a_1(t), \cdots, a_N(t))  $,  
and $\Sigma(t) = \vcat({\sigma}(t), 0_{N \times N}) \in \sR^{2N \times N}$ with ${\sigma}(t) \coloneqq \diag(\sigma_1(t), \cdots, \sigma_N(t))$.
The $\alpha$-potential function $\Phi$ for $\cG_{\textrm{LQ}}$  is given by (see  \eqref{eq:alpha_general} and 
\eqref{eq:cost_lifted} with $z=0$):
\begin{equation}
    \begin{aligned}
\label{eq:LQPotential}
    \Phi(\bu) 
    = \E \left[
    \int_0^T \int_{\cS}   \left(  \mathbbm{x}^\top Q \mathbbm{x} + 2 r \bu^\top_t  \bu_t \right) \d \mu_{t}^{   \mathfrak{r},\bu}  \d t  + \int_{\cS} \left(    \mx^\top \bar Q \mx +2 \mathfrak{p}^\top \mx \right)  \d \mu_{T}^{  \mathfrak{r},\bu} 
    \right], 
\end{aligned}
\end{equation}
where 
$\cS\coloneqq \sR^{2N}\times [0,1]$, 
$\mu_t^{\mathfrak{r},\bu} \coloneqq \cL(\sX_t^{\mathfrak{r}, \bu}, \mathfrak{r} | \cF_t)$ for all $t$, 
\(Q  \coloneqq \begin{pmatrix}
    \mymathbb{0}_N & \tilde{Q}^\top \\
    \tilde{Q} & \mymathbb{0}_{N}
\end{pmatrix} \in\sS^{2N} \) with
$\tilde{Q}\in \sR^{N\times N}$ given by $\tilde{Q}_{i,i} = \frac{1}{N} \sum_{k\ne i, k\in { [N]}}q_{ik}$
and $\tilde{Q}_{i,j} = -\frac{q_{ij}}{N} $ for all $i \ne j$, 
$\bar Q \coloneqq \begin{pmatrix}
    \mymathbb{0}_N & \Gamma  \\ 
    \Gamma & \mymathbb{0}_N
\end{pmatrix} \in\sS^{2N}$ with $\Gamma \coloneqq \diag(\gamma_1, \cdots, \gamma_N) \in \sS^N$, and $\mathfrak{p} \coloneqq - \vcat(0_N,  \gamma_1 d_1,\cdots,  \gamma_N d_N) \in \sR^{2N}$.
Above and hereafter,
for each $n\in \sN$, 
we denote by       $\sS^n$   the space of $n\times n$ symmetric matrices,
by $\mymathbb{0}_n$ the $n\times n$ zero matrix, 
and by  $\diag(a_1,\ldots, a_n)$  the diagonal matrix with diagonal elements $(a_1,\ldots, a_n)$.

The minimizer of \eqref{eq:LQPotential} can be characterized   with suitable ordinary differential
equations (ODEs). 
These ODEs differ from the  Riccati equations for  usual LQ control problems studied in
 \cite{sun2016open},
due to the additional dependence on the parameter $r$ in \eqref{eq:LQ_lifted} and \eqref{eq:LQPotential}. To see this,
let    $M_0\in C^1([0,T] ;  \sS^{2N})$ satisfy  the  following linear ODE:    
\begin{equation}\label{eq:M0}
\dot{M_0} + A^\top M_0 + M_0  A + Q = 0;  \quad 
M_0(T) = \bar Q,
\end{equation}
where the  dot    denotes the time derivative. Consider the following Riccati equation
for $M_1\in C^1([0,T];\sS^{4N})$:
\begin{equation}\label{eq:M1}
\begin{aligned}
\dot{M}_1 &+ \begin{pmatrix}
       A & \mymathbb{0}_{2N} \\   \mymathbb{0}_{2N} &  A  \end{pmatrix}  M_1 + M_1 \begin{pmatrix}
       A & \mymathbb{0}_{2N} \\   \mymathbb{0}_{2N} &  A  \end{pmatrix}
        - K_{M_0,M_1}^\top K_{M_0,M_1} = 0; 
       \quad 
M_1(T) = \mymathbb{0}_{4N},
\end{aligned}
\end{equation}
with  $K_{M_0,M_1}: [0,T]\to \sR^{N\times 4N}$
defined by
\begin{align}
    K_{M_0,M_1} & \coloneqq 
\begin{pmatrix}
       \begin{pmatrix}
          \mymathbb{0}_{N}    & \sI_N
         \end{pmatrix}  M_0
         &
              \begin{pmatrix}
          \sI_N &\mymathbb{0}_{N}   
         \end{pmatrix}  M_0
\end{pmatrix}
           + \tilde I  M_1 , 
           \quad
\tilde I  \coloneqq \begin{pmatrix}
    \frac{1}{2} \sI_{N}  & \sI_{N} & \frac{1}{3}\sI_{N}& \frac{1}{2} \sI_{N}
    \end{pmatrix}\in \sR^{N\times 4N}.
\end{align}
The constants in $\tilde I$
correspond to $\sE[\mathfrak r] $
and $\sE[\mathfrak r^2] $ for 
  the uniform random variable  $\mathfrak r$
  in \eqref{eq:LQ_lifted}. 
Given a solution $M_1$ to \eqref{eq:M1}, consider the following linear ODE
for $M_2\in C^1([0,T]; \sR^{4N})$:
\begin{equation}\label{eq:M2}
\dot{M}_2 + \begin{pmatrix}
       A & \mymathbb{0}_{2N} \\   \mymathbb{0}_{2N} &  A  \end{pmatrix}  M_2-    K_{M_0,M_1}^\top  \tilde I  M_2    = 0;\quad  
M_2(T) = \begin{pmatrix} 
 \mathfrak{p} \\ 0_{2N}\end{pmatrix}.
\end{equation}

 The following theorem constructs a  minimizer of 
  $\cH^2(\sR^N)\ni \bu \mapsto \Phi(\bu)\in \sR$
based on  solutions of  \eqref{eq:M0}, \eqref{eq:M1} and \eqref{eq:M2},
which subsequently yields  an $\alpha_N$-NE of the game $\cG_{\textrm{LQ}}$.
The proof is given in Section \ref{sec:lq_verification}.

\begin{theorem}\label{thm:lq_verification}
 
Suppose that   $M_0 \in C^1([0,T] ;  \sS^{2N}) $,
$M_1  \in C^1([0,T];\sS^{4N})$,
and 
$M_2\in C^1([0,T];\sR^{4N})$
 satisfy \eqref{eq:M0}, \eqref{eq:M1}, and \eqref{eq:M2},
respectively. 
{
Define 
$ 
     \bu^*_t = 
      -   K_{M_0,M_1}(t)
  \begin{pmatrix}
 \sE [ {\sX}_t^{\mathfrak{r}, \bu^*}|\cF_t]    \\ 
 \sE [\mathfrak r  {\sX}_t^{\mathfrak{r}, \bu^*} |\cF_t]
  \end{pmatrix}  
  -
    \tilde I M_2(t) 
   $   for all   $t\in [0,T]$.
   Assume that 
   $\bu^*=(u^*_i)_{i\in [N]}$ satisfies
$\|u^*_i\|_{\cH^2(\sR)}\le L$    for all $i\in [N]$,
   with 
   $L>0$
   in \eqref{eq:control_ex}.
Then  $\bu^*$ is 
  an $\alpha_N$-NE of  $\cG_{\textrm{LQ}} $,
  with $\alpha_N$  satisfying \eqref{eq:alpha_N}.
  }
Moreover,   the process  
    $F_t \coloneqq  \begin{pmatrix}
 \sE [ {\sX}_t^{\mathfrak{r}, \bu^*}|\cF_t]    \\ 
 \sE [\mathfrak r  {\sX}_t^{\mathfrak{r}, \bu^*} |\cF_t]
  \end{pmatrix} $, $t\in [0,T]$,  
 satisfies the linear  SDE 
\begin{align}
\label{eq:F_dynamics}
 \d  F_t
  =  \left[ 
  \left(\begin{pmatrix}
       A(t) & \mymathbb{0}_{2N} \\   \mymathbb{0}_{2N} &  A(t)  \end{pmatrix}
  -\tilde I^\top  K_{M_0,M_1}(t)\right) 
   F_t
  -\tilde I^\top \tilde I M_2(t)  
  \right]\d t + \begin{pmatrix}
        \Sigma(t) \\ \frac{1}{2} \Sigma(t)
    \end{pmatrix}\d W_t,
    \quad F_0 = \begin{pmatrix}
        \sX_0\\
        \frac{1}{2}\sX_0
    \end{pmatrix}.
 \end{align}

\end{theorem}

\begin{remark}
Theorem \ref{thm:lq_verification}
leverages the  LQ structure of
$\cG_{\textrm{LQ}} $
to 
characterize  the $\alpha_N$-NE  $\bu^*$
as a feedback function of $F$,
which involves    finite conditional moments of   $({\sX}_t^{\mathfrak{r}, \bu^*}, \mathfrak r)$.
These moments serve as sufficient statistics for  the infinite-dimensional conditional law  $\cL({\sX}_t^{\mathfrak{r}, \bu^*}, \mathfrak r|\cF_t)$.
Notably, the process $F$ is Markovian and satisfies the linear SDE \eqref{eq:F_dynamics},
enabling the efficient computation  of the $\alpha_N$-NE.  
We remark  that  
the solvability of \eqref{eq:M0} and \eqref{eq:M2} follows from   linear ODE theory, 
and 
the solvability of  \eqref{eq:M1}
can be ensured at least for sufficiently small $T$.   
\end{remark}

\section{Proofs of main results}
\subsection{Proof of Theorem \ref{thm:potential_hessian}}
\label{sec:proof_theorem_hessian}

The following     lemmas  regarding the linear derivative
are given   in \cite[Lemmas 4.1 and 4.2]{guo2024towards},  and will be used in the proof of Theorem \ref{thm:potential_hessian}.

 \begin{lemma}
\label{lemma:derivative_line}
Suppose  
$\cA^{(N)}$ is    convex,
$i\in { [N]}$, 
and   $f: \cA^{(N)}\to \sR$ 
has  a linear  derivative $ \frac{\delta f}{\delta a_i} $ with respect to $\cA_i$.
Let $\ba=(a_i,a_{-i})\in  \cA^{(N)}$, $a'_i\in \cA_i$,
and for each $\varepsilon \in [0,1]$, let $\ba^\varepsilon =( a_i+\varepsilon(a'_i-a_i),a_{-i})$.
Then 
the function 
$[0,1]\ni \varepsilon\mapsto f(\ba^\varepsilon)\in \sR$
is differentiable and 
$\frac{\d }{\d \varepsilon }f(\ba^\varepsilon)=  \frac{\delta f}{\delta a_i} (\ba^\varepsilon ; a'_i-a_i)
 $
 for all $\varepsilon\in [0,1]$.  
\end{lemma}

 \begin{lemma}
 \label{lemma:multi-dimension_derivative}
 Suppose $\cA^{(N)} $ is convex  and for all $i\in { [N]}$, $f: \cA^{(N)}\to \sR$ has a linear derivative $ \frac{\delta f}{\delta a_i} $ with respect to $\cA_i$ such that for all $\bz,\ba\in \cA^{(N)}$ and $a'_i\in \cA_i $,  $[0,1]^N\ni {{\varepsilon}}\mapsto  \frac{\delta f}{\delta a_i}(\bz+\varepsilon\cdot (\ba-\bz) ;a'_i)$ is continuous at $0$, where  $\bz+\varepsilon\cdot (\ba-\bz)\coloneqq (z_i+\varepsilon_i(a_{i}-z_i))_{i\in { [N]}}$. Then for all  $\bz,\ba\in \cA^{(N)}$, the map $[0,1]\ni r\mapsto  f(\bz+r( \ba -\bz))\in \sR$ is differentiable and $\frac{\d}{\d r}   f(\bz+r( \ba -\bz))=\sum_{j=1}^N \frac{\delta  f}{ \delta a_j}(\bz+r(\ba-\bz); a_j-z_j)$.  
 \end{lemma}
\begin{proof}[Proof of Theorem \ref{thm:potential_hessian}]

By Condition \ref{item:continuity}  
and Lemma \ref{lemma:multi-dimension_derivative},
$[0,1]\ni r\mapsto \frac{\delta V_j}{\delta a_j}\left(\bz+r(\ba-\bz) ; a_j-z_j\right)\in \sR
$ is differentiable, and hence $\Phi$ in \eqref{eq:Phi} is well-defined.

We now prove that   $\Phi$ has a linear derivative with respect to $\cA_i$ for all $i\in { [N]}$.
To this end, 
let $i\in { [N]}$, $\ba\in \cA^{(N)} $ and $a'_i\in \cA_i$.
For all $\vep \in (0,1]$, let $\ba^{\varepsilon}:=\left(a_i+\varepsilon\left(a_i^{\prime}-a_i\right), a_{-i}\right)$. 
By the definition of $\Phi$ in \eqref{eq:Phi},
    $$
        \Phi\left(\ba^{\vep}\right)-\Phi(\ba)=  \int_0^1 \sum_{j=1}^N \frac{\delta V_j}{\delta a_j}\left(\bz+r\left(\ba^{\vep}-\bz\right) ; a_j+\vep \delta_{ji}\left(a_i^{\prime}-a_i\right)-z_j\right) \d r 
         -\int_0^1 \sum_{j=1}^N \frac{\delta V_j}{\delta a_j}\left(\bz+r(\ba-\bz) ; a_j-z_j\right) \d r.
    $$
Then by  $\bz+r\left(\ba^{\vep}-\bz\right) \in \cA^{(N)}$, for all $\vep \in(0,1]$,
\begin{equation}\label{eqn: Phi - Phi over ep}
      \begin{aligned}
          &\frac{\Phi\left(\ba^{\vep}\right)-\Phi(\ba)}{\vep}
         =\frac{1}{\vep} \int_0^1 \sum_{j=1}^N\left(\frac{\delta V_j}{\delta a_j}\left(\bz+r\left(\ba^{\vep}-\bz\right) ; a_j-z_j\right)-\frac{\delta V_j}{\delta a_j}\left(\bz+r(\ba-\bz) ; a_j-z_j\right)\right) \mathrm{d} r \\
        & \quad+\frac{1}{\vep} \int_0^1 \sum_{j=1}^N \vep \delta_{j i} \frac{\delta V_j}{\delta a_j}\left(\bz+r\left(\ba^{\vep}-\bz\right) ; a_j^{\prime}-a_j\right) \mathrm{d} r \\
        & =\int_0^1 \sum_{j=1}^N \frac{1}{\vep}\left(\frac{\delta V_j}{\delta a_j}\left(\bz+r\left(\ba^{\vep}-\bz\right) ; a_j-z_j\right)-\frac{\delta V_j}{\delta a_j}\left(\bz+r(\ba-\bz) ; a_j-z_j\right)\right) \mathrm{d} r 
        \\
        & \quad
        +\int_0^1 \frac{\delta V_i}{\delta a_i}\left(\bz+r\left(\ba^{\vep}-\bz\right) ; a_i^{\prime}-a_i\right) \d r.
    \end{aligned}
\end{equation}
To send $\vep \rightarrow 0$ in the above equation, note that for all $\vep \in[0,1], r\in[0,1], \left(\bz+r\left(\ba^{\vep}-\bz\right)\right)_{-i}=z_{-i}+r\left(a_{-i}-z_{-i}\right)$ and
$
    \left(\bz+r\left(\ba^{\vep}-\bz\right)\right)_i  =z_i+r\left(a_i+\vep\left(a_i^{\prime}-a_i\right)-z_i\right) 
     =z_i+r\left(a_i-z_i\right)+\vep\left(\left(z_i+r\left(a_i^{\prime}-z_i\right)\right)-\left(z_i+r\left(a_i-z_i\right)\right)\right)
$
with $z_i + r\left(a_i-z_i\right), z_i+r\left(a_i^{\prime}-z_i\right) \in \cA_i$. Thus for all $j \in { [N]}$, 
the twice differentiability of $V_j$ and Lemma \ref{lemma:derivative_line} 
imply that $\vep \mapsto \frac{\delta V_j}{\delta a_j}\left(\bz+r\left(\ba^{\vep}-\bz\right) ; a_j-z_j\right)$ is differentiable on $[0,1]$ and 
$$
\begin{aligned}
\frac{\mathrm{d}}{\mathrm{d} \varepsilon} \frac{\delta V_j}{\delta a_j}\left(\bz+r\left(\ba^{\varepsilon}-\bz\right) ; a_j-z_j\right) & =\frac{\delta^2 V_j}{\delta a_j \delta a_i}\left(\bz+r\left(\ba^{\varepsilon}-\bz\right) ; a_j-z_j, r\left(a_i^{\prime}-a_i\right)\right) \\
& =\frac{\delta^2 V_j}{\delta a_j \delta a_i}\left(\bz+r\left(\ba^{\varepsilon}-\bz\right) ; a_j-z_j, a_i^{\prime}-a_i\right) r,
\end{aligned}
$$
where the last identity used the linearity of $\frac{\delta^2 V_j}{\delta a_j \delta a_i}$ in its last component. Hence, by the mean value theorem and Condition \ref{item: boundedness}, for all $\vep \in (0,1]$,
$$ \begin{aligned}
& \left|\frac{1}{\vep}\left(\frac{\delta V_j}{\delta a_j}\left(\bz+r\left(\ba^{\varepsilon}-\bz\right) ; a_j-z_j\right)-\frac{\delta V_j}{\delta a_j}\left(\bz+r(\ba-\bz) ; a_j-z_j\right)\right)\right| \\
& \quad \leq \sup _{r,\, \vep \in[0,1]}\left|\frac{\delta^2 V_j}{\delta a_j \delta a_i}\left(\bz+r\left(\ba^{\vep}-\bz\right) ; a_j-z_j, a_i^{\prime}-a_i\right) r\right| < \infty.
\end{aligned} $$
Similarly, as $a_i^{\prime}-a_i \in \operatorname{span}\left(\cA_i\right)$, by the twice differentiability of $V_i$, for all $r \in(0,1)$,
$\lim _{\varepsilon \downarrow 0} \frac{\delta V_i}{\delta a_i}(\bz+r\left(\ba^{\varepsilon}-\bz\right) ; a_i^{\prime}-a_i)=\frac{\delta V_i}{\delta a_i}(\bz+r(\ba-\bz) ; a_i^{\prime}-a_i),$
and for all $r, \vep \in[0,1]$, by the mean value theorem, there exists $\tilde\vep \in [0,1]$ such that  
\begin{equation}
\label{eq:dV_da_r_eps}
\begin{aligned}
& \left|\frac{\delta V_i}{\delta a_i}\left(\bz+r\left(\ba^{\varepsilon}-\bz\right) ; a_i^{\prime}-a_i\right)\right|  \leq\left|\frac{\delta V_i}{\delta a_i}\left(\bz+r(\ba-\bz) ; a_i^{\prime}-a_i\right)\right|+\left|\frac{\delta^2 V_i}{\delta a_i \delta a_i}\left(z+r\left(\ba^{ \tilde\varepsilon}-\bz\right) ; a_i^{\prime}-a_i, a_i^{\prime}-a_i\right) r\right|.
\end{aligned}
\end{equation}
Using Lemma \ref{lemma:multi-dimension_derivative}, 
for all $a'_i\in \cA_i$,
\begin{equation}\label{eqn: dr dV}
    \frac{\d}{\d r} \frac{\delta V_i}{\delta a_i}\left(\bz+r(\ba-\bz) ; a_i^{\prime}\right)=\sum_{j=1}^N \frac{\delta^2 V_i}{\delta a_i \delta a_j}\left(\bz+r(\ba-\bz) ; a_i^{\prime}, a_j-z_j\right),
\end{equation}
which along with \eqref{eq:dV_da_r_eps} and Condition \ref{item: boundedness} implies that 
$\sup_{(r,\vep) \in [0,1]^2}\left|\frac{\delta V_i}{\delta a_i}\left(\bz+r\left(\ba^{\varepsilon}-\bz\right) ; a_i^{\prime}-a_i\right)\right|<\infty$. 
Hence, letting $\vep \rightarrow 0$ in \eqref{eqn: Phi - Phi over ep} and using Lebesgue's dominated convergence theorem give
\begin{align*}
    \left.\frac{\d}{\d \vep} \Phi\left(\ba^{\vep}\right)\right|_{\vep=0}     =\int_0^1 \sum_{j=1}^N \frac{\delta^2 V_j}{\delta a_j \delta a_i}\left(\bz+r(\ba-\bz) ; a_j-z_j, a_i^{\prime}-a_i\right) r \d r+\int_0^1 \frac{\delta V_i}{\delta a_i}\left(\bz+r(\ba-\bz) ; a_i^{\prime}-a_i\right) \d r.
\end{align*}
Let $\mathcal E:[0,1]\to \sR$ be given by
\begin{equation}
    \label{eq:residual}
    \cE_r \coloneqq 
     \sum_{j=1}^N
     \left(
     \frac{\delta^2 V_j}{\delta a_j \delta a_i}\left(\bz+r(\ba-\bz) ; a_j-z_j, a_i^{\prime}-a_i\right) -  \frac{\delta^2 V_i}{\delta a_i \delta a_j} \left(\bz+r(\ba-\bz) ; a_i^{\prime}-a_i, a_j-z_j \right) 
     \right).
\end{equation}
Then by \eqref{eqn: dr dV},  
    \begin{align} 
    \label{eqn: d_Phi_d_epsilon}
    \begin{split}
      \left.\frac{\d}{\d \vep} \Phi\left(\ba^{\vep}\right)\right|_{\vep=0}    
    & = \int_0^1 \left( \sum_{j=1}^N \frac{\delta^2 V_i}{\delta a_i \delta a_j}  \left(\bz+r(\ba-\bz) ;  a_i^{\prime}-a_i, a_j-z_j\right) 
    +\cE_r \right) r \d r
    +\int_0^1 \frac{\delta V_i}{\delta a_i}\left(\bz+r(\ba-\bz) ; a_i^{\prime}-a_i\right) \d r  \\
    & = \int_0^1 r \frac{\d}{\d r}\left(\frac{\delta V_i}{\delta a_i}\left(\bz+r(\ba-\bz) ; a_i^{\prime}-a_i\right)\right) \d r+\int_0^1 \frac{\delta V_i}{\delta a_i}\left(\bz+r(\ba-\bz) ; a_i^{\prime}-a_i\right) \d r     + \int_0^1 \cE_r   r\d r   \\
    & = \frac{\delta V_i}{\delta a_i}\left(\ba ; a_i^{\prime}-a_i\right) + \int_0^1 \cE_r r \d r,
\end{split}
\end{align}
where the last line uses the integration by part formula. 
This proves the linear differentiability of $\Phi$.

Now we prove $\Phi$ is an $\alpha$-potential function of $\cG$. Let $i \in { [N]}, a_i^{\prime} \in \cA_i$ and $\ba \in \cA^{(N)}$. For each $\varepsilon \in[0,1]$, let $\ba^{\varepsilon}=\left(a_i+\varepsilon\left(a_i^{\prime}-a_i\right), a_{-i}\right) \in \cA^{(N)}$. By the differentiability of $V_i$ and Lemma \ref{lemma:derivative_line}, $\frac{\mathrm{d}}{\mathrm{d} \varepsilon} V_i\left(\ba^{\varepsilon}\right)=\frac{\delta V_i}{\delta a_i}\left(\ba^{\varepsilon} ; a_i^{\prime}-a_i\right)$ for all $\varepsilon \in[0,1]$, and $\varepsilon \mapsto \frac{\delta V_i}{\delta a_i}\left(\ba^{\varepsilon} ; a_i^{\prime}-a_i\right)$ is differentiable on $[0,1]$. This implies that $\varepsilon \mapsto V_i\left(\ba^{\varepsilon}\right)$ is continuously differentiable on $[0,1]$. 
Similarly,   by Lemma \ref{lemma:derivative_line}  and \eqref{eqn: d_Phi_d_epsilon} and the continuity assumption, $[0,1] \ni \varepsilon \mapsto \Phi\left(\ba^{\varepsilon}\right) \in \mathbb{R}$ is also continuously differentiable with 
$
   \frac{\d}{\d \vep} \Phi\left(\ba^{\vep}\right) =  \frac{\delta V_i}{\delta a_i}\left(\ba^\vep; a_i^{\prime}-a_i\right) + \int_0^1 \cE_{r,\vep} r \d r,
$
where $\cE_{r,\vep}$ is given by
$
\cE_{r,\vep} = \sum_{j=1}^N
     \left(
     \frac{\delta^2 V_j}{\delta a_j \delta a_i}\left(\bz+r(\ba^\vep-\bz); a_j-z_j, a_i^{\prime}-a_i\right) -  \frac{\delta^2 V_i}{\delta a_i \delta a_j} \left(\bz+r(\ba^\vep-\bz) ; a_i^{\prime}-a_i, a_j-z_j \right) 
     \right).
$
Hence by the fundamental theorem of calculus,
$$\begin{aligned} 
      V_i\left(\left(a_i^{\prime}, a_{-i}\right)\right)-V_i\left(\left(a_i, a_{-i}\right)\right) &=\int_0^1 \frac{\delta V_i}{\delta a_i}\left(\ba^{\varepsilon} ; a_i^{\prime}-a_i\right) \mathrm{d} \varepsilon 
     =\int_0^1 \frac{\d}{\d \vep} \Phi\left(\ba^{\vep}\right) \mathrm{d} \varepsilon - \int_{0}^1 \int_0^1  \cE_{r,\vep} r \d r \d \vep \\ 
    & =\Phi\left(\left(a_i^{\prime}, a_{-i}\right)\right)-\Phi\left(\left(a_i, a_{-i}\right)\right)  - \int_{0}^1 \int_0^1 \cE_{r,\vep}  r \d r \d \vep.
\end{aligned}
$$
Finally, the desired upper bound of $\alpha$ follows from the fact that 
$$
\left|\int_{0}^1 \int_0^1 \cE_{r,\vep}  r \d r \d \vep\right|
\le { 2}\sup_{i \in { [N]}, a'_i\in \cA_i, \ba, \ba''     \in \cA^{(N)}}
   \sum_{j=1}^N\left| \frac{\delta^2 V_i}{\delta a_i \delta a_j} \left(\ba ; a_i^{\prime} , a''_j  \right) 
-\frac{\delta^2 V_j}{\delta a_j \delta a_i}\left(\ba ; a''_j , a_i^{\prime} \right)  
 \right|.
$$
due to the bilinearity of $\frac{\delta^2 V_j}{\delta a_j \delta a_i}$ and $\frac{\delta^2 V_i}{\delta a_i \delta a_j}${, and the fact that $\int_0^1 \int_0^1 r \d r \d \epsilon = \frac{1}{2}$}.
This finishes the proof.
\end{proof}

\subsection{Proof of Theorem \ref{thm:value_jacobian_open}}
\label{sec:open_loop_jacobian_proof}

The following propositions estimate the moments of the state process $\bX^{\bu}$  and the   sensitivity processes $\bY^{\bu, u'_h}$
and $\bZ^{\bu, u'_h, u''_\ell}$. 
The proofs of these propositions are included in Section \ref{sec:open_loop_moment_estimate}. 

\begin{prop}
\label{prop:X_i_moment_open}
Suppose Assumption \ref{assum:regularity} holds. 
For each $\bu \in \cH^p(\sR^N)$,
the solution  $\bX^{\bu} \in \cH^p(\sR^N)$
to  \eqref{eq:X_i_u_i_open} 
satisfies for all $i\in { [N]}$,
 $\sup_{t\in [0,T]} \sE[| X^{{\bu}}_{t,i}|^p ]  \leq C_X^{ i,p }, $ 
with the constant $C_X^{i,p}$  defined by
$C_X^{i,p}  \coloneqq
 \Big( 
 |x_i|^p + (p-1)\|\sigma_i\|^p_{L^p}+ L^{b} T
 + \| u_i\|_{\mathcal{H}^p (\mathbb{R})}^p
 + \frac{L^{b}_y }{N}       
 \sum_{k=1}^N \big( |x_k|^p + (p-1)\|\sigma_k\|^p_{L^p}+ L^{b} T +  \| u_k\|_{\mathcal{H}^p (\mathbb{R})}^p \big) 
 \Big)e^{  c_p (  L^{b} +     L^{b}_y+1    )  T },$
and  $c_p\ge 1$ is a   constant  depending only on $p$. 

\end{prop}

\begin{prop}
  \label{prop:Y_h_i_moment_open_loop}
Suppose Assumption \ref{assum:regularity} holds and let  $p \geq 2$.
For all $\bu \in \cH^p(\sR^N)$,
$h\in { [N]}$ and 
$ u'_h  \in    \mathcal{H}^{p}(\sR) $,
the solution $\bY^{\bu, u'_h} \in \cH^p(\sR^N)$  of  \eqref{eq:Y_h_i_open} satisfies
for all $i\in { [N]}$,
\begin{align*}
 \begin{split}
\sup_{t\in [0,T]}  \sE[   |Y^{\bu,u'_h}_{t,i}|^p]
 &\le   \left(\delta_{h,i} C^{h,p}_Y
 +\frac{ ( L^{b}_y)^p}{N^p}\bar{C}^{h,p}_Y \right)\| u_h'\|^p_{\mathcal{H}^p (\mathbb{R})},
  \end{split}
 \end{align*}
where
$C^{h,p}_Y  \coloneqq     (2 T)^{p-1}   e^{p L^{b} T}$ and $
 \bar{C}^{h,p}_Y  \coloneqq 
 (2 T)^{2p-1} 
  e^{  p(      L^{b} +      L^{b}_y) T}    e^{p L^{b} T}.$
\end{prop}

 \begin{prop}
 \label{prop:Z_h_ell_i_moment_open} 
 Suppose Assumption \ref{assum:regularity} holds.
For all $\bu \in \cH^p(\sR^N)$,
$h,\ell \in { [N]}$ with $h\not =\ell$, and 
all 
$u'_h , u''_\ell\in   \cH^{2}(\sR)$,
the solution $\bZ^{\bu, u'_h, u''_\ell} \in \cH^p(\sR^N)$  of  \eqref{eq:Z_h_ell_i_open} satisfies
for all $i\in { [N]}$,
\begin{align*}
&\sup_{t\in[0,T]} \sE \left[   |Z^{\bu,u'_h,u''_\ell}_{t,i}|^2\right]
\leq 
C ( L^{b}_y)^2 \left((\delta_{h,i}  +\delta_{\ell,i} )\frac{1}{N^2}+\frac{1}{N^4}\right)
         \| u_h'\|_{\cH^4(\sR)}^2 \| u_\ell''\|_{\cH^4(\sR)}^2 ,
 \end{align*}
where $C\ge 0$ is a constant depending only on the upper bounds of $T$, $L^b, L_y^b$.
\end{prop}

 We now prove Theorem \ref{thm:value_jacobian_open} based on Propositions \ref{prop:X_i_moment_open}, \ref{prop:Y_h_i_moment_open_loop}
and \ref{prop:Z_h_ell_i_moment_open},

\begin{proof}[Proof of Theorem \ref{thm:value_jacobian_open}]
To simplify the notation,
we omit the dependence on $\bu$ in the superscript of all processes, i.e.,
  $\bX=\bX^\bu, \bY^i =\bY^{\bu, u_i'}$. We denote by    $C\ge 0$ a generic constant 
 depending only on the upper bounds of   $T$, 
$\max_{i\in { [N]}}  |x_i|^2$, $\max_{i\in { [N]}} \|\sigma_i\|_{L^2}$, 
$L^b, L_y^b$, $\max_{k\in { [N]}} \| u_k \|_{\mathcal{H}^2(\sR)}$.

By the definition of $\frac{\delta^2 V_j}{\delta u_j\delta u_i}(\bu;u''_j,u'_i)$ in \eqref{eq:2d_linear_derivative_open}
and the fact that $ \bZ^{\bu,u'_i,u''_j}=\bZ^{\bu,u''_j,u'_i}$,
\begin{align}
\label{eq:V_jacobian_difference_open_loop}
\begin{split}
&\left|\frac{\delta^2 V_i}{\delta u_i\delta u_j}(\bu;u'_i,u''_j)
-\frac{\delta^2 V_j}{\delta u_j\delta u_i}(\bu;u''_j,u'_i)
\right|
\\
&
=
\sE\left[ \int_0^T 
\left\{  \begin{pmatrix}
  \bY^{i}_t
  \\
   u_{t,i}'
\end{pmatrix}^\top
\begin{pmatrix}
 \partial^2_{xx} \Delta^f_{i,j }      & \partial^2_{xu_j} \Delta^f_{i,j } 
\\
\partial^2_{u_i x} \Delta^f_{i,j }    & \partial^2_{u_i u_j} \Delta^f_{i,j } 
\end{pmatrix}(t,\cdot)
 \begin{pmatrix}
  \bY^{ j}_t
  \\
  u''_{t,j}
\end{pmatrix}
+  \begin{pmatrix}
  \bZ^{i,j}_t
\end{pmatrix}^\top
\begin{pmatrix}
 \partial_{x} \Delta^f_{i,j }  
\end{pmatrix}(t, \cdot)
\right\} \d t\right]
\\
&\quad +
\sE\left[ 
 ( \bY^{i}_T)^\top  (\partial^2_{xx} \Delta^g_{i,j }  ) (\bX_T)\bY^{j}_T
 + (\bZ^{i,j}_T)^\top ( \partial_{x} \Delta^g_{i,j } )(\bX_T)
 \right],
 \end{split}
\end{align}
where we write for simplicity   $ \partial^2_{xx} \Delta^f_{i,j }  (t, \cdot) =  \partial^2_{xx} (f_i-f_j) (t,\bX_t, \bu_t)
$ and similarly for other derivatives. 
In the sequel, we 
derive upper bounds for all terms on the right-hand side of   \eqref{eq:V_jacobian_difference_open_loop}.
To estimate the term involving the Hessian of $\Delta^f_{i,j}$ in \eqref{eq:V_jacobian_difference_open_loop}, observe that for all $t\in [0,T]$, 
\begin{align}
\label{eq:f_qudratic_Yv_open_loop}
\begin{split}
   &\begin{pmatrix}
  \bY^{i}_t
  \\
   u_{t,i}'
\end{pmatrix}^\top
\begin{pmatrix}
 \partial^2_{xx} \Delta^f_{i,j }      & \partial^2_{xu} \Delta^f_{i,j } 
\\
\partial^2_{u_i x} \Delta^f_{i,j }    & \partial^2_{u_i u_j} \Delta^f_{i,j } 
\end{pmatrix}(t,\cdot)
 \begin{pmatrix}
  \bY^{ j}_t
  \\
  u''_{t,j}
\end{pmatrix}
=\sum_{h,\ell=1}^N  (\partial^2_{x_hx_\ell } \Delta^f_{i,j })  (t, \cdot)   Y^{i}_{t, h} Y^{ j}_{t, \ell} + u_{t,j}''  \sum_{h}^N  (\partial^2_{x_h u_j } \Delta^f_{i,j }) (t, \cdot)     Y^{i}_{t, h} 
\\
& 
\quad+  u_{t,i}'  \sum_{\ell=1}^N (\partial^2_{u_i x_\ell  } \Delta^f_{i,j }) (t, \cdot) Y^{ j}_{t, \ell}   
+  (\partial^2_{u_i u_j } \Delta^f_{i,j })  (t, \cdot) u_{t,i}'  u_{t,j}''. 
\end{split}
\end{align}

The first term on the right-hand side of \eqref{eq:f_qudratic_Yv_open_loop} satisfies the identity:
\begin{align*}
\begin{split}
 \sum_{h,\ell=1}^N & (\partial^2_{x_hx_\ell } \Delta^f_{i,j })  (t, \cdot)   Y^{i}_{t, h} Y^{ j}_{t, \ell} 
=  ( \partial^2_{x_ix_j } \Delta^f_{i,j } ) (t, \cdot)   Y^{i}_{t, i } Y^{ j}_{t, j} 
+\sum_{ \ell\in { [N]}\setminus\{j\}}  (\partial^2_{x_i x_\ell } \Delta^f_{i,j })  (t, \cdot)   Y^{i}_{t, i} Y^{ j}_{t, \ell} 
\\
&\quad 
+\sum_{ h\in { [N]}\setminus\{i\}}
 (\partial^2_{x_ h x_j } \Delta^f_{i,j })  (t, \cdot)   Y^{i}_{t, h} Y^{ j}_{t, j} 
+\sum_{  h\in { [N]}\setminus\{i\}, \ell\in { [N]}\setminus\{j\}}  (\partial^2_{x_h x_\ell } \Delta^f_{i,j } ) (t, \cdot)   Y^{i}_{t, h} Y^{ j}_{t, \ell}  ,
\end{split}
\end{align*}
which 
yields the following estimate:  
\begin{align}
&\Bigg| \sE\Bigg[\int_0^T  \sum_{h,\ell=1}^N ( \partial^2_{x_hx_\ell } \Delta^f_{i,j } ) (t, \cdot)   Y^{i}_{t, h} Y^{ j}_{t, \ell} \d t\Bigg]\Bigg|
\le \|   \partial^2_{x_ix_j } \Delta^f_{i,j }  \|_{L^\infty} \|Y^{i}_{i} Y^{ j}_{j}\|_{\cH^1(\sR)} 
+\sum_{ \ell\in { [N]}\setminus\{j\}} \|  \partial^2_{x_i x_\ell }  \Delta^f_{i,j }  \|_{L^\infty}  
\| Y^{i}_{i} Y^{ j}_{\ell}\|_{\cH^1(\sR)} 
\notag \\
&\quad 
+\sum_{ h\in { [N]}\setminus\{i\}}
 \| \partial^2_{x_ h x_j }  \Delta^f_{i,j }  \|_{L^\infty}  \|  Y^{i}_{h} Y^{ j}_{j}  \|_{\cH^1(\sR)} 
+\sum_{  h\in { [N]}\setminus\{i\}, \ell\in { [N]}\setminus\{j\}}  \| \partial^2_{x_h x_\ell }  \Delta^f_{i,j }  \|_{L^\infty}      \| Y^{i}_{h} Y^{ j}_{\ell}\|_{\cH^1(\sR)}  \notag\\ 
& \le
  C\| u_i'\|_{\cH^2(\sR)} \| u_j''\|_{\cH^2(\sR)}  \Bigg\{ \|   \partial^2_{x_ix_j } \Delta^f_{i,j }  \|_{L^\infty}  
+\frac{L^{b}_y }{N}
 \Bigg(\sum_{ \ell\in { [N]}\setminus\{j\}} \|  \partial^2_{x_i x_\ell }  \Delta^f_{i,j }  \|_{L^\infty} 
 +\sum_{ h\in { [N]}\setminus\{i\}} \| \partial^2_{x_ h x_j }  \Delta^f_{i,j }  \|_{L^\infty}
 \Bigg)
 \notag\\
 &\quad  + \frac{(L^{b}_y)^2 }{N^2} \Bigg(
\sum_{ h\in { [N]}\setminus\{i\}} \sum_{ \ell\in { [N]}\setminus\{j\}}  \| \partial^2_{x_h x_\ell }  \Delta^f_{i,j }  \|_{L^\infty} \Bigg)
\Bigg\}.\label{eq:cost_f_xx_open_loop}
\end{align}
where the second inequality follows from the Cauchy-Schwarz inequality and  Proposition \ref{prop:Y_h_i_moment_open_loop}.
Similarly, using  Propositions \ref{prop:Y_h_i_moment_open_loop}, 
the second and third terms  in \eqref{eq:f_qudratic_Yv_open_loop} can be bounded by 
\begin{align}
\label{eq:cost_f_xu_open_loop}
\begin{split}
&\Bigg| \sE\Bigg[\int_0^T u_{t,j}'' \sum_{h=1}^N  (\partial^2_{x_h u_j } \Delta^f_{i,j }) (t, \cdot)     Y^{i}_{t, h}
 \d t\Bigg]\Bigg|
 +\Bigg| \sE\Bigg[\int_0^T  u_{t,i}'  \sum_{\ell=1}^N (\partial^2_{u_i x_\ell  } \Delta^f_{i,j }) (t, \cdot) Y^{ j}_{t, \ell}   
 \d t\Bigg]\Bigg|
\\
\le  &  C\| u_i'\|_{\cH^2(\sR)} \| u_j''\|_{\cH^2(\sR)} \bigg\{ 
    \| \partial_{x_i u_j}^2 \Delta_{i,j}^f\|_{L^\infty} + \| \partial_{u_i x_j}^2 \Delta_{i,j}^f\|_{L^\infty} 
+\frac{L_y^b}{N} \big( \sum_{h\in { [N]} \setminus\{i\}} \| \partial^2_{x_h u_j} \Delta_{i,j}^f \|_{L^\infty}  
 +\sum_{\ell \in { [N]} \setminus\{j\}} \| \partial_{ u_i x_\ell} \Delta_{i,j}^f \|_{L^\infty} 
 \big) 
 \bigg\}
\end{split}
\end{align}
and the fourth term in \eqref{eq:f_qudratic_Yv_open_loop} can be bounded by
\begin{align}\label{eq:cost_f_uu_open_loop}
    \left|\E \left[ \int_0^T (\partial^2_{u_i u_j } \Delta^f_{i,j })  (t, \cdot) u_{t,i}'  u_{t,j}''  \d t \right]\right| \leq \| u_i'\|_{\cH^2(\sR)} \| u_j''\|_{\cH^2(\sR)} \|\partial^2_{u_i u_j } \Delta^f_{i,j } \|_{L^\infty}.
\end{align}
Combining \eqref{eq:cost_f_xx_open_loop}, \eqref{eq:cost_f_xu_open_loop}, and \eqref{eq:cost_f_uu_open_loop} yield the following bound of \eqref{eq:f_qudratic_Yv_open_loop}:
\begin{align}
\label{eq:f_qudratic_Yv_bound_open_loop}
\begin{split}
& \Bigg| \sE\Bigg[\int_0^T  \begin{pmatrix}
  \bY^{i}_t
  \\
  u'_{t,i} 
\end{pmatrix}^\top
\begin{pmatrix}
 \partial^2_{xx} \Delta^f_{i,j }      & \partial^2_{xu_j} \Delta^f_{i,j } 
\\
\partial^2_{u_ix} \Delta^f_{i,j }    & \partial^2_{u_iu_j} \Delta^f_{i,j } 
\end{pmatrix}(t, \cdot)
 \begin{pmatrix}
  \bY^{ j}_t
  \\
  u''_{t,j} 
\end{pmatrix}\d t\Bigg]\Bigg|
\\
&\le C \| u_i'\|_{\cH^2(\sR)} \| u_j''\|_{\cH^2(\sR)} \Bigg\{ 
\|   \partial^2_{x_ix_j } \Delta^f_{i,j }  \|_{L^\infty}  + \|   \partial^2_{x_i u_j } \Delta^f_{i,j }  \|_{L^\infty}  + \|   \partial^2_{u_i x_j } \Delta^f_{i,j }  \|_{L^\infty} + \|   \partial^2_{u_i u_j } \Delta^f_{i,j }  \|_{L^\infty}
\\
&
\quad +
\frac{L_y^b}{N} \Bigg(\sum_{ \ell\in { [N]}\setminus\{j\}} \big(\|  \partial^2_{x_i x_\ell }  \Delta^f_{i,j }  \|_{L^\infty}  
+
  \| \partial^2_{ u_i x_\ell} \Delta_{i,j}^f \|_{L^\infty} \big)
 +
 \sum_{ h\in { [N]}\setminus\{i\}} \big(\| \partial^2_{x_ h x_j }  \Delta^f_{i,j }  \|_{L^\infty}
 + 
\| \partial^2_{x_h u_j} \Delta_{i,j}^f \|_{L^\infty}  \big)
 \Bigg)
 \\
 &
 \quad + \frac{(L_y^b)^2}{N^2} \Bigg( \sum_{ h\in { [N]}\setminus\{i\}, \ell\in { [N]}\setminus\{j\}}  \| \partial^2_{x_h x_\ell }  \Delta^f_{i,j }  \|_{L^\infty}
 \Bigg)
\Bigg\}.
\end{split}
\end{align}

To estimate the term involving the gradient of $\Delta^f_{i,j}$ in \eqref{eq:V_jacobian_difference_open_loop}, observe that for all $t\in[0,T],$
 $\begin{pmatrix}
  \bZ^{i,j}_t
\end{pmatrix}^\top
\begin{pmatrix}
 \partial_{x} \Delta^f_{i,j }  
\end{pmatrix}(t, \cdot)
=\sum_{h=1}^N ( \partial_{x_h} \Delta^f_{i,j } ) (t, \cdot)  Z^{i,j}_{t,h}.$
The fundamental theorem of calculus   implies that  for all $(t,x,u)=(t,(x_\ell)_{\ell = 1}^N, (u_\ell)_{\ell = 1}^N)\in [0,T]\in \sR^N\times\sR^N$ and $h\in { [N]}$,
$
|( \partial_{x_h} \Delta^f_{i,j } ) (t, x, u) |\le |( \partial_{x_h} \Delta^f_{i,j } ) (t, 0, 0) |
+\sum_{\ell=1}^N  (\|  \partial^2_{x_hx_\ell } \Delta^f_{i,j } \|_{L^\infty} |x_\ell|
+\|  \partial^2_{x_h u_\ell } \Delta^f_{i,j } \|_{L^\infty} |u_\ell|),
$
which implies that 
\begin{align*}
\begin{split}
\Bigg| \sE&\Bigg[\int_0^T \sum_{h=1}^N ( \partial_{x_h} \Delta^f_{i,j } ) (t, \cdot)  Z^{i,j}_{t,h}  \d t\Bigg]\Bigg|
\le \sum_{h\in \{i,j\} }  \Big(  
\|( \partial_{x_h} \Delta^f_{i,j } ) (\cdot, 0, 0) \|_{L^2}
\| Z^{i,j}_{h}\|_{\cH^2(\sR)} 
\\
+ &\sum_{\ell=1}^N  
 (\|  \partial^2_{x_hx_\ell } \Delta^f_{i,j } \|_{L^\infty} \|X_{\ell}   Z^{i,j}_{h}\|_{\cH^1(\sR)} 
+\|  \partial^2_{x_h u_\ell } \Delta^f_{i,j } \|_{L^\infty}\|u_{\ell}   Z^{i,j}_{h}\|_{\cH^1(\sR)} 
)
\Big) 
\\
+ &\sum_{h\in { [N]}\setminus \{i,j\} }  \Big(  
\|( \partial_{x_h} \Delta^f_{i,j } ) (\cdot, 0, 0) \|_{L^2}
\| Z^{i,j}_{h}\|_{\cH^2(\sR)} 
+ \sum_{\ell=1}^N  
 (\|  \partial^2_{x_hx_\ell } \Delta^f_{i,j } \|_{L^\infty} \|X_{\ell}   Z^{i,j}_{h}\|_{\cH^1(\sR)} 
+\|  \partial^2_{x_h u_\ell } \Delta^f_{i,j } \|_{L^\infty}\|u_{\ell}   Z^{i,j}_{h}\|_{\cH^1(\sR)} 
)
\Big). 
\end{split}
\end{align*}
Then by  the Cauchy-Schwarz inequality and 
 Propositions \ref{prop:X_i_moment_open} and \ref{prop:Z_h_ell_i_moment_open},
 \begin{align}
 \label{eq:cost_f_x_Z_open_loop}
\begin{split}
&\Bigg| \sE\Bigg[\int_0^T \sum_{h=1}^N ( \partial_{x_h} \Delta^f_{i,j } ) (t, \cdot)  Z^{i,j}_{t,h}  \d t\Bigg]\Bigg|
\\
&\le
C\| u_i'\|_{\cH^4(\sR)} \| u_j''\|_{\cH^4(\sR)}L_y^b
\Bigg\{ \frac{ 1}{N}  \sum_{h\in \{i,j\} }  \Bigg(  
\|( \partial_{x_h} \Delta^f_{i,j } ) (\cdot, 0, 0) \|_{L^2}
+
 \sum_{\ell=1}^N  
 \left(\|  \partial^2_{x_hx_\ell } \Delta^f_{i,j } \|_{L^\infty} 
+\|  \partial^2_{x_h u_\ell } \Delta^f_{i,j } \|_{L^\infty} 
\right)
\Bigg)
\\
&\quad + \frac{1}{N^2} \sum_{h\in { [N]}\setminus \{i,j\} }  \Bigg(  
\|( \partial_{x_h} \Delta^f_{i,j } ) (\cdot, 0, 0) \|_{L^2}
+ \sum_{\ell=1}^N  
 \left(\|  \partial^2_{x_hx_\ell } \Delta^f_{i,j } \|_{L^\infty} 
+\|  \partial^2_{x_h u_\ell } \Delta^f_{i,j } \|_{L^\infty} 
\right)\Bigg) 
\Bigg\}. 
\end{split}
\end{align}

Finally,  using similar arguments as those for  \eqref{eq:cost_f_xx_open_loop} and \eqref{eq:cost_f_x_Z_open_loop} allows for estimating 
  the terms involving  $\Delta^g_{i,j}$ in \eqref{eq:V_jacobian_difference_open_loop}: 
\begin{align}
\label{eq:g_bound_open_loop}
\begin{split}
& \left|\sE\left[ 
 ( \bY^{i}_T)^\top  (\partial^2_{xx} \Delta^g_{i,j }  ) (\bX_T)\bY^{j}_T
 + (\bZ^{i,j}_T)^\top ( \partial_{x} \Delta^g_{i,j } )(\bX_T)
 \right]\right|
 \\
 &\le 
 C \| u_i'\|_{\cH^4(\sR)} \| u_j''\|_{\cH^4(\sR)}\Bigg\{ \|   \partial^2_{x_ix_j } \Delta^g_{i,j }  \|_{L^\infty}  
+ \frac{L_y^b}{N}
 \Bigg(
\sum_{h\in \{i,j\} }  
|( \partial_{x_h} \Delta^g_{i,j } ) ( 0)| 
+
 \sum_{h\in \{i,j\}, \ell\in { [N]} } 
 \|  \partial^2_{x_hx_\ell } \Delta^g_{i,j } \|_{L^\infty} 
  \Bigg)
\\
&
\quad+\frac{L_y^b}{N^2}\Bigg(
\sum_{h\in { [N]}\setminus \{i,j\} } | (\partial_{x_h} \Delta^g_{i,j } ) ( 0)|  
+ \sum_{h\in { [N]}\setminus \{i,j\} ,\ell \in { [N]}}  
 \|  \partial^2_{x_hx_\ell } \Delta^g_{i,j } \|_{L^\infty} 
+\sum_{ h\in { [N]}\setminus\{i\},  \ell\in { [N]}\setminus\{j\}}  \| \partial^2_{x_h x_\ell }  \Delta^g_{i,j }  \|_{L^\infty}  \Bigg)
\Bigg\}.
\end{split}
\end{align}
Note that the last two terms in the last line can be replaced by  $\sum_{h\in { [N]}\setminus \{i,j\},\ell \in { [N]}\setminus\{i,j\}}  
 \|  \partial^2_{x_hx_\ell } \Delta^g_{i,j } \|_{L^\infty} $,
 as the remaining ones can be absorbed in the terms with $1/N$. 
Consequently, using \eqref{eq:V_jacobian_difference_open_loop} and 
grouping the terms  in  
  the estimates  \eqref{eq:f_qudratic_Yv_bound_open_loop}, 
\eqref{eq:cost_f_x_Z_open_loop} and \eqref{eq:g_bound_open_loop} according to the orders   
$1/N$ and $1/N^2$ yield 
  \begin{equation*}
        \left|\frac{\delta^2 V_i}{\delta u_i \delta u_j}\left(\bu ; u_i^{\prime}, u_j^{\prime \prime}\right)-\frac{\delta^2 V_j}{\delta u_j \delta u_i}\left(\bu ; u_j^{\prime \prime}, u_i^{\prime}\right)\right|
        \leq C \| u_i'\|_{\cH^4(\sR)} \| u_j''\|_{\cH^4(\sR)}  \left(C_{V,1}^{i,j} + L_y^b 
        \left(
        \frac{ 1}{N}C_{V,2}^{i,j} + \frac{ 1}{N^2} C_{V,3}^{i,j}\right)\right),
    \end{equation*}
{ where $C_{V,1}^{i,j}, C_{V,2}^{i,j}$ and $C_{V,3}^{i,j}$ are given \eqref{eq:C_V_1_open}, \eqref{eq:C_V_2_open}, and \eqref{eq:C_V_3_open} respectively.}
This finishes the proof.
\end{proof}

 \subsection{Proof of Theorem \ref{thm:lq_verification}}
\label{sec:lq_verification}

 \begin{proof}[Proof of Theorem \ref{thm:lq_verification}]
  It suffices to show $\bu^*$ is a minimizer    of \eqref{eq:LQPotential} over $\cH^2(\sR^N)$.
         Define  $\hat{V}:[0,T]\times \cP_2(\cS)\to \sR$   such that for all $(t,\mu)\in [0,T]\times \cP_2 (\cS)$,
\begin{align*}
    \hat{V}(t, \mu) = 
     \tr (M_0(t)  \overline{\mu}_2)
    +   \begin{pmatrix}
    \overline{\mu} \\ \overline{\mu}_1
    \end{pmatrix}^\top M_1(t) \begin{pmatrix}
    \overline{\mu} \\ \overline{\mu}_1
    \end{pmatrix} 
    + 2 M_2(t)^\top \begin{pmatrix}
    \overline{\mu} \\ \overline{\mu}_1 
    \end{pmatrix}
    + M_3(t),
\end{align*}
where 
$\overline{\mu} \coloneqq \int_{\cS}  \mathbbm{x} \mu (\d (\mathbbm{x} ,r) ), $  $\overline\mu_1 \coloneqq  \int_{\cS}  r \mathbbm{x} \mu (\d (\mathbbm{x} ,r) ),$  $ \overline{\mu}_2 \coloneqq  \int_{\cS}  \mathbbm{x} \mx^\top \mu (\d (\mathbbm{x} ,r) )$,
and $M_3\in C([0,T];  \sR)$
satisfies 
$$
\dot{M}_3+  \operatorname{tr}\left(   
 \Sigma  \Sigma^\top 
 \left(M_0
+  
 \begin{pmatrix}
 \sI_{2N} \\ \frac{1}{2}\sI_{2N}
 \end{pmatrix}^\top
 M_1  \begin{pmatrix}
 \sI_{2N} \\ \frac{1}{2}\sI_{2N}
 \end{pmatrix}\right)
\right)
-(\tilde IM_2 )^\top\tilde IM_2=0; 
\quad 
M_3(T)=0.
$$
We shall  prove $\hat V $  satisfies the optimality condition \eqref{eq:MV_potential}. 
In the sequel, the time variable of all coefficients will be dropped when there is no
risk of confusion.

Let $\bu \in \cH^2(\sR^N)$,
let  
 $\sX^{  \mathfrak{r},\bu}\in \cS^2(\sR^{2N})$ satisfy \eqref{eq:LQ_lifted},
 and let $\mu_t^{\mathfrak{r},\bu} \coloneqq \cL(\sX_t^{\mathfrak{r}, \bu}, \mathfrak{r} | \cF_t)$ for all $t$.
 By 
It\^{o}'s formula   in \cite{guo2024ito} (see also \cite[Theorem 4.17]{carmona2018probabilisticII}), 
{\begin{align}
\label{eq:lq_ito}
\begin{split}
   \hat{V}(T, \mu_T^{\mathfrak{r},\bu}) - \hat{V}(0, \mu_0^{\mathfrak{r},\bu}) 
  &=\bar\sE\Bigg[\int_0^T \Bigg\{(\partial_t \hat V)(t,\mu_t^{\mathfrak{r},\bu})  
    +    \left(A(t) \tilde \sX_t^{\tilde{\mathfrak{r}} , \bu}+\cI_{\tilde{\mathfrak{r}} } \bu_t  \right)^\top  \partial_\sx \frac{\delta \hat V}{\delta\mu}(t, \mu_t^{\mathfrak{r},\bu}, \tilde \sX_t^{\tilde{\mathfrak{r}}, \bu}, \tilde{\mathfrak{r}} )
    \\
    &\quad + \frac{1}{2} \tr\left(\Sigma(t)\Sigma(t)^\top \partial^2_{\sx\sx} \frac{\delta \hat V}{\delta\mu}(t, \mu_t^{\mathfrak{r},\bu}, \tilde \sX_t^{\tilde{\mathfrak{r}}, \bu}, \tilde{\mathfrak{r}}) \right)
      \\
      &
   \quad  +  \frac{1}{2}  \tr\left(\Sigma(t)\Sigma(t)^\top \partial^2_{\sx\sx'} \frac{\delta^2 \hat V}{\delta^2 \mu}(t, \mu_t^{\mathfrak{r},\bu}, \tilde \sX_t^{\tilde{\mathfrak{r}}, \bu},   \tilde{\mathfrak{r}},
  \hat \sX_t^{\hat{\mathfrak{r}}, \bu} , \hat{\mathfrak{r}} ) \right)
   \Bigg\} \d t \,\Bigg\vert\,\cF_T
    \Bigg],
 \end{split}    
\end{align}}
where $(\tilde \sX^{\tilde{\mathfrak{r}}, \bu},   \tilde{\mathfrak{r}})$
and $(  \hat \sX^{\hat{\mathfrak{r}}, \bu} , \hat{\mathfrak{r}})$
are   conditional independent copies of $(  \sX^{\mathfrak{r}, \bu} , {\mathfrak{r}})$ given $\cF_T$ defined on an enlarged probability space $(\bar{\Omega},\bar{\cF},\bar{\sP})$
with $\cF_T\subset \bar{\cF}$,
and $\bar{\sE}[\cdot|\cF_T]$
is  the conditional
expectation  in the enlarged probability space.

We now compute   the right-hand side of \eqref{eq:lq_ito}. 
Note that $\bu $ and 
$\mu^{\mathfrak{r},\bu}$ are measurable 
with respect to $\cF_T$,
and 
$\mu^{\mathfrak{r},\bu}_t = \cL(\tilde \sX^{\tilde{\mathfrak{r}}, \bu},   \tilde{\mathfrak{r}}_t | \cF_T)=\cL(\hat \sX^{\hat{\mathfrak{r}}, \bu}_t,   \hat{\mathfrak{r}}| \cF_T) $ for all $t\in [0,T]$. Hence 
for all $t\in [0,T]$,  by the symmetry of $M_0(t)$ and $M_1(t)$,
\begin{equation}
\label{eq:lq_term1}
\begin{aligned}
  &   \bar{\sE}\left[\left(A(t) \tilde \sX_t^{\tilde{\mathfrak{r}} , \bu}+\cI_{\tilde{\mathfrak{r}} } \bu_t  \right)^\top  \partial_\sx \frac{\delta \hat V}{\delta\mu}(t, \mu_t^{\mathfrak{r},\bu}, \tilde \sX_t^{\tilde{\mathfrak{r}}, \bu}, \tilde{\mathfrak{r}} )\,\Bigg\vert\,\cF_T  \right]
  \\
  &=2\int_{\cS} \left(A(t) \sx + \cI_{ r} \bu_t    \right)^\top  \left(M_0(t) \mx +\begin{pmatrix} \mathbb{I}_{2N} \\ r\mathbb{I}_{2N} \end{pmatrix} ^\top M_1(t) \begin{pmatrix}
\overline{\mu_t^{\mathfrak{r},\bu}}
      \\
      (\overline{\mu_t^{\mathfrak{r},\bu}})_1
  \end{pmatrix}  +  \begin{pmatrix} \mathbb{I}_{2N} \\ r\mathbb{I}_{2N} \end{pmatrix} ^\top M_2(t)\right) \d   \mu_t^{\mathfrak{r},\bu}(\sx,r)
  \\
  &=
 2\Bigg\{ \tr\left(A^\top M_0 (\overline{\mu_t^{\mathfrak{r},\bu}})_2\right) + \begin{pmatrix}
\overline{\mu_t^{\mathfrak{r},\bu}}
      \\
      (\overline{\mu_t^{\mathfrak{r},\bu}})_1
  \end{pmatrix}^\top M_1 
  \begin{pmatrix}
       A & \mymathbb{0}_{2N} \\   \mymathbb{0}_{2N} &  A  \end{pmatrix} \begin{pmatrix}
\overline{\mu_t^{\mathfrak{r},\bu}}
      \\
      (\overline{\mu_t^{\mathfrak{r},\bu}})_1
  \end{pmatrix} + M_2^\top \begin{pmatrix}
       A & \mymathbb{0}_{2N} \\   \mymathbb{0}_{2N} &  A  \end{pmatrix} \begin{pmatrix}
\overline{\mu_t^{\mathfrak{r},\bu}}
      \\
      (\overline{\mu_t^{\mathfrak{r},\bu}})_1
  \end{pmatrix}
  \\
  &\quad +\bu_t^\top \left[  K_{M_0,M_1}
  \begin{pmatrix}
\overline{\mu_t^{\mathfrak{r},\bu}}
      \\
      (\overline{\mu_t^{\mathfrak{r},\bu}})_1
  \end{pmatrix}
  +
    \tilde I M_2\right]
 \Bigg\},
  \end{aligned}
\end{equation}
 where  the last term  used the fact that the marginal distribution of $\mu_t^{\mathfrak{r},\bu}$ on $[0,1]$ is the uniform distribution.
Moreover, 
\begin{align}
\label{eq:lq_term2}
\begin{split}
     \frac{1}{2} \tr\left(\Sigma(t)\Sigma(t)^\top \partial^2_{\sx\sx} \frac{\delta \hat V}{\delta\mu}(t, \mu_t^{\mathfrak{r},\bu}, \tilde \sX_t^{\tilde{\mathfrak{r}}, \bu}, \tilde{\mathfrak{r}}) \right)
     &=  \operatorname{tr}\left(   
 \Sigma(t) \Sigma(t)^\top M_0(t)
\right),
\\
 \frac{1}{2}  \tr\left(\Sigma(t)\Sigma(t)^\top \partial^2_{\sx\sx'} \frac{\delta^2 \hat V}{\delta^2 \mu}(t, \mu_t^{\mathfrak{r},\bu}, \tilde \sX_t^{\tilde{\mathfrak{r}}, \bu},   \tilde{\mathfrak{r}},
  \hat \sX_t^{\hat{\mathfrak{r}}, \bu} , \hat{\mathfrak{r}} ) \right)
  &=   \operatorname{tr}\left(   
 \Sigma(t) \Sigma(t)^\top 
 \begin{pmatrix}
 \sI_{2N} \\ \frac{1}{2}\sI_{2N}
 \end{pmatrix}^\top
 M_1(t) \begin{pmatrix}
 \sI_{2N} \\ \frac{1}{2}\sI_{2N}
 \end{pmatrix}
\right).
\end{split}
\end{align}
 Observe further  that for all $t\in [0,T]$,
 by completing the squares,
 \begin{align}
 \label{eq:lq_optimal}
 \begin{split}
   &  \bu_t^\top 2\left[  K_{M_0,M_1}
  \begin{pmatrix}
\overline{\mu_t^{\mathfrak{r},\bu}}
      \\
      (\overline{\mu_t^{\mathfrak{r},\bu}})_1
  \end{pmatrix}
  +
    \tilde I M_2\right]
   \\
   &  
    \ge  
   -   \bu_t^\top \bu_t   -
    \left[  K_{M_0,M_1}
  \begin{pmatrix}
\overline{\mu_t^{\mathfrak{r},\bu}}
      \\
      (\overline{\mu_t^{\mathfrak{r},\bu}})_1
  \end{pmatrix}
  +
    \tilde I M_2\right]^\top \left[  K_{M_0,M_1}
  \begin{pmatrix}
\overline{\mu_t^{\mathfrak{r},\bu}}
      \\
      (\overline{\mu_t^{\mathfrak{r},\bu}})_1
  \end{pmatrix}
  +
    \tilde I M_2\right]
    \\
    &
    \ge  
   - \int_\cS 2r \bu_t^\top \bu_t \d \mu_t^{\mathfrak{r},\bu} -
    \begin{pmatrix}
\overline{\mu_t^{\mathfrak{r},\bu}}
      \\
      (\overline{\mu_t^{\mathfrak{r},\bu}})_1
  \end{pmatrix}^\top K_{M_0,M_1}^\top K_{M_0,M_1} \begin{pmatrix}
\overline{\mu_t^{\mathfrak{r},\bu}}
      \\
      (\overline{\mu_t^{\mathfrak{r},\bu}})_1
  \end{pmatrix}
  \\
  &\quad 
  -2 (\tilde IM_2 )^\top 
  K_{M_0,M_1}
  \begin{pmatrix}
\overline{\mu_t^{\mathfrak{r},\bu}}
      \\
      (\overline{\mu_t^{\mathfrak{r},\bu}})_1
  \end{pmatrix}
  -(\tilde IM_2 )^\top\tilde IM_2 .
\end{split}
 \end{align}
 Hence   combining 
\eqref{eq:lq_ito},
\eqref{eq:lq_term1},
\eqref{eq:lq_term2},
 and \eqref{eq:lq_optimal}, and using the ODEs for   $M_0$,   $M_1$ and $M_2$ yield
\begin{align}
\label{eq:V_hat_condition}
\begin{split}
   \hat{V}(T, \mu_T^{\mathfrak{r},\bu}) - \hat{V}(0, \mu_0^{\mathfrak{r},\bu}) 
  &\ge  \int_0^T  
  - \int_\cS (\sx^\top Q\sx+2r \bu_t^\top \bu_t) \d \mu_t^{\mathfrak{r},\bu}
  \d t,
 \end{split}    
\end{align}
from which, by
using   
$ \hat{V}(T, \mu_T^{\mathfrak{r},\bu})= \int_{\cS}  \left(\mx^\top \bar{Q} \mathbbm{x} +2 \mathfrak p^\top \sx\right)
\d \mu_T^{\mathfrak{r},\bu} $
and taking the expectation, we obtain that  
$$
\hat V\left(0,\delta_{\vcat(x_1,\ldots, x_N, {0}_{N^2d})}\otimes \operatorname{Unif}(0,1)\right) \le \Phi(\bu ), \quad \forall \bu \in \cH^2(\sR^N).
$$
 
Finally, consider the feedback map 
$
\hat a(t,\mu)\coloneqq -\left[  K_{M_0,M_1}(t)
  \begin{pmatrix}
\overline{\mu}
      \\
      \overline{\mu}_1
  \end{pmatrix}
  +
    \tilde I M_2(t) \right]
     $ for all $(t,\mu)\in [0,T]\times \cP_2(\cS)$.
Since $\mathfrak r$ is fixed, 
by \cite[Theorem A.3]{djete2022mckean} and the boundedness of $K_{M_0,M_1}$ and $M_2$, the dynamics
\begin{equation}
\label{eq:lq_feedback}
 \d \sX_t = \left( A(t) \sX_t + \cI_{\mathfrak{r}} \hat a \big(t, \cL(\sX_{t}, \mathfrak r \mid \cF_t)\big) \right) \d t + \Sigma(t) \d W_t,\quad 
  \sX_0  = \vcat(x_1, \cdots, x_N, 0_N),
  \end{equation}
 admits a unique $\sG$-adapted strong solution $(\hat{\sX}, \mathfrak r) $ satisfying 
$\sE[\sup_{t\in [0,T]}||\hat\sX_{t}|^p] <\infty$ for any $p\ge 2$.
Thus the control $\bu^*_t =\hat a \big(t, \cL(\hat\sX_{t}, \mathfrak r \mid \cF_t)\big) $, $t\in [0,T]$,
 is in $\cA^{(N)}=\prod_{i\in { [N]}} \cA_i$,
 and achieves the equality in \eqref{eq:lq_optimal}.
This implies  \eqref{eq:V_hat_condition} is an equality and hence $\bu^*$ is the minimizer of $\Phi$. Since 
$\Phi$ is the $\alpha_N$-potential function of  $\cG_{\textrm{LQ}} $, 
$\bu^*$ is an $\alpha_N$-NE of 
 $\cG_{\textrm{LQ}} $ by Proposition \ref{prop:optimizer_and_NE}.  

To derive that dynamics of  $F_t \coloneqq  \begin{pmatrix}
 \sE [ {\sX}_t^{\mathfrak{r}, \bu^*}|\cF_t]    \\ 
 \sE [\mathfrak r  {\sX}_t^{\mathfrak{r}, \bu^*} |\cF_t]
  \end{pmatrix} $, $t\in [0,T]$, using \eqref{eq:lq_feedback},
 \begin{align*}
 \d  \begin{pmatrix}
  {\sX}_t^{\mathfrak{r}, \bu^*}     \\ 
 \mathfrak r  {\sX}_t^{\mathfrak{r}, \bu^*}
  \end{pmatrix}
  =  \left( \begin{pmatrix}
       A(t) & \mymathbb{0}_{2N} \\   \mymathbb{0}_{2N} &  A(t)  \end{pmatrix}
       \begin{pmatrix}
  {\sX}^{\mathfrak{r}, \bu^*}_t   \\ 
 \mathfrak r  {\sX}^{\mathfrak{r}, \bu^*}_t  
  \end{pmatrix}
  -\begin{pmatrix}
        \cI_{\mathfrak r} \\ \mathfrak r  \cI_{\mathfrak r} 
     \end{pmatrix}\left[  K_{M_0,M_1}(t)
  \begin{pmatrix}
 \sE[  {\sX}^{\mathfrak{r}, \bu^*}_t|\cF_t]   \\ 
\sE[  \mathfrak r  {\sX}^{\mathfrak{r}, \bu^*}_t  |\cF_t]
  \end{pmatrix}
  +
    \tilde I M_2(t) \right]
  \right)\d t + \begin{pmatrix}
        \Sigma(t) \\ \mathfrak{r} \Sigma(t)
    \end{pmatrix}\d W_t.
 \end{align*} 
 For each $t\in [0,T]$,
taking the conditional expectation with respect to $\cF_t$, applying the conditional Fubini Theorem 
and using the independence between $\mathfrak r$ and $\cF_t$
yield the dynamics \eqref{eq:F_dynamics} of $F$. This completes the proof.
  \end{proof}

\section{Proofs of Propositions \ref{prop:X_i_moment_open}, \ref{prop:Y_h_i_moment_open_loop}
and \ref{prop:Z_h_ell_i_moment_open}}
\label{sec:open_loop_moment_estimate}

The following lemma quantifies the growth of  $f\in \mathscr{F}^{0,2}([0,T]\times \sR\times \sR^{N};\sR)$ in the space variables,
 which will be used to prove Proposition \ref{prop:X_i_moment_open}.
 The proof follows directly from the mean value theorem and hence is omitted. 
 
 \begin{lemma}
\label{lemma:linear_growth}
Let $f\in \mathscr{F}^{0,2}([0,T]\times \sR\times \sR^{N};\sR)$. 
Then for all  $  t\in [0,T], x\in \sR$ and $y \in \sR^{N}$,
 $|f(t,x,y)| \le  L^f (1+|x|)+\frac{L^f_y}{N}\sum_{i=1}^N |y_i|$.

\end{lemma}

\begin{proof}[Proof of Proposition \ref{prop:X_i_moment_open}]
    Throughout this proof, we 
write $\bX=\bX^\bu$ for notational simplicity.
    By \eqref{eq:X_i_u_i_open} and It\^{o}'s formula, for all $t\in [0,T]$, 
    $$
      \d |X _{t,i}|^p=  |X_{t,i}|^{p-2}\left(
      p X_{t,i}  \left( u_{t,i} + b_i(t,  X_{t,i},  \mathbf{X}_{t}) \right)
      +\frac{p(p-1)}{2}\sigma_i^2(t) 
      \right)  
    \d t+p|X_{t,i}|^{p-2} X_{t,i} \sigma_i(t) \d W^i_t, 
      \quad X^p_{0,i} = x^p_i.
    $$
   Taking the expectation of both sides and using the fact that 
    $\left(\int_0^t |X_{u,i}|^{p-2}X_{u,i} \sigma_i(u) \d   W^i_u\right)_{t\ge 0} $
    is a martingale (see \cite[Problem 2.10.7]{zhang2017backward}) yield  that 
    \begin{align*}
      &  \sE[|X_{t,i}|^p ]
 \leq |x_i|^p+   \sE\left[\int_0^t   
 \left(
 p|{X}_{s,i} |^{p-1} \left( 
        |u_i(s)| + 
         L^{b} (1+|{X}_{s,i} | )+\frac{L^{b}_y}{N}
        \sum_{k=1}^N |{X}_{s,k}|
        \right)
        +   |X_{s,i}|^{p-2} \frac{p(p-1)}{2}\sigma_i^2(s) 
        \right)
 \d s \right]. 
  \end{align*}
  By Young's inequality, for all $a, b\ge 0$, 
  $ab \le \frac{p-1}{p}a^{p/(p-1)}+\frac{1}{p}b^p$ and 
  $a b \le \frac{p-2}{p} a^{p/(p-2)}+\frac{2}{p}b^{p/2}$ if $p>2$. 
  Hence 
   \begin{align}
        \label{eq:X_i_bound_open}
   \begin{split}
        \sE[|X_{t,i}|^p ]
&  \leq |x_i|^p +   \sE\bigg[\int_0^t   
\bigg(
    (p-1)|X_{s,i}|^p + |u_i(s)|^p + 
  L^{b} (1+ (2p-1)  |{X}_{s,i} |^{p} )
        \\
        &\quad 
        +\frac{L^{b}_y}{N}
        \sum_{k=1}^N 
        \left((p-1) |{X}_{s,i} |^{p}+        |{X}_{s,k}|^p \right)
        +    \frac{p(p-1)}{2}\left( \frac{p-2}{p} |X_{s,i}|^{p} +\frac{2}{p}|\sigma_i(s)|^{p} \right)  
        \bigg)
 \d s \bigg]
     \\
&  \leq |x_i|^p + (p-1)\|\sigma_i\|^p_{L^p}+ L^{b} T
+ \| u_i\|_{\mathcal{H}^p (\mathbb{R})}^p
    \\
        &\quad 
+  \int_0^t   
\bigg(
 \left( L^{b}  (2p-1)
  + L^{b}_y(p-1)+ \frac{(p-1)p}{2}
  \right)  \sE[ |{X}_{s,i} |^{p}] +\frac{L^{b}_y}{N}
        \sum_{k=1}^N      \sE[|{X}_{s,k}|^p ]
        \bigg)
 \d s.
 \end{split} 
  \end{align}
    Summing up the above equation over the index  $i\in { [N]}$ yields for all $t\in [0,T]$,
  \begin{align*}
   \begin{split}
       \sum_{i=1}^N \sE[|X_{t,i}|^p ]
&  \leq  \sum_{i=1}^N \left( |x_i|^p + (p-1)\|\sigma_i\|^p_{L^p}+ L^{b} T + \| u_i\|_{\mathcal{H}^p (\mathbb{R})}^p\right)
    \\
        &\quad 
+  \int_0^t   
\bigg(
 \left( L^{b}  (2p-1)
  + L^{b}_y p+ \frac{(p-1)p}{2}
  \right)   
        \sum_{k=1}^N      \sE[|{X}_{s,k}|^p ]
        \bigg)
 \d s,
 \end{split} 
  \end{align*}
      which along with   Gronwall's inequality implies that 
     \begin{align*}
     \begin{split}
        \sum_{k=1}^N\sE[|X_{t,k}| ^p] 
        &\leq  
        \sum_{k=1}^N \left( |x_k|^p + (p-1)\|\sigma_k\|^p_{L^p}+ L^{b} T +  \| u_k\|_{\mathcal{H}^p (\mathbb{R})}^p \right) 
        e^{  c_p (  L^{b} +     L^{b}_y+1    )  T },
\end{split}
    \end{align*}
for a constant $c_p\ge 1$ depending only on $p$. 
    Substituting the above inequality into \eqref{eq:X_i_bound_open}
    and applying Gronwall's inequality
    yield  
 \begin{align*}
   \begin{split}
       \sE[|X_{t,i}|^p ]
&  \leq |x_i|^p + (p-1)\|\sigma_i\|^p_{L^p}+ L^{b} T
+  \int_0^t   
 c_p\left( L^{b}   
  + L^{b}_y + 1 
  \right)  \sE[ |{X}_{s,i} |^{p}]  
 \d s  
    \\
        &\quad 
+ \frac{L^{b}_y  }{N}           \sum_{k=1}^N \left( |x_k|^p + (p-1)\|\sigma_k\|^p_{L^p}+ L^{b} T +  \| u_k\|_{\mathcal{H}^p (\mathbb{R})}^p \right) 
        e^{c_p (L^{b} + L^{b}_y+1)T}
 \\
 &\le \bigg( |x_i|^p + (p-1)\|\sigma_i\|^p_{L^p}+ L^{b} T
 + \| u_i\|_{\mathcal{H}^p (\mathbb{R})}^p\\
 &\quad 
 + \frac{L^{b}_y }{N}        \sum_{k=1}^N \left( |x_k|^p + (p-1)\|\sigma_k\|^p_{L^p}+ L^{b} T +  \| u_k\|_{\mathcal{H}^p (\mathbb{R})}^p \right) 
        e^{  c_p (  L^{b} +     L^{b}_y+1    )  T }
 \bigg)e^{  c_p (  L^{b} +     L^{b}_y+1    )  T }.
 \end{split} 
  \end{align*}
 This finishes the proof. 
\end{proof}

The following lemma will be used to estimate the sensitivity processes.

 \begin{lemma}
  \label{lemma:S_i_general}
Let $p\ge 2$ and for each $i,j\in { [N]}$,    
let 
$B_i,\bar{B}_{i,j}: \Omega\times [0,T]\to \sR$ 
be bounded adapted processes, 
and 
$f_i\in \cH^p(\sR) $.  
Let 
$\textbf{S}  =
(S_{i})_{i=1}^N\in \cS^p(\sR^N)$ satisfy the following dynamics: 
for all $t\in [0,T]$,
\begin{equation} 
\label{eq:S_i_general}
    \d S_{t,i}=
    \left(
     B_i(t) S_{t,i}+ \sum_{j=1}^N\bar{B}_{ij}(t)  S_{t,j}   
    +f_{t,i} 
    \right)
    \d t, 
    \quad S_{0,i}=0; 
      \quad \forall i=1, \cdots, N.
\end{equation}
Then for all  $i\in { [N]}$,
\begin{align*} 
  \sup_{t\in [0,T]}  \sE[  |S_{t,i}|^p]
    &\le 
  (2 T)^{p-1}   \left( 
   \left\|   
    f_{ i}    \right\|^p_{\cH^p(\sR)}
     + 
  \left\|  \sum_{k=1}^N
 |  f_{k}
   |   
      \right\|^p_{\cH^p(\sR)}
    \|\bar B\|^p_{\infty}
 T^p e^{ p (\|B\|_{\infty}+N \|\bar B\|_{\infty}) T}
   \right) e^{p\|B\|_\infty T}.
\end{align*}
where $\|B\|_{\infty}=\max_{i\in { [N]}}\|B_i\|_{L^\infty}$
and 
$\|\bar B\|_{\infty}=\max_{i,j\in { [N]}}\|\bar B_{i,j}\|_{L^\infty}$.
\end{lemma}

\begin{proof}
By \eqref{eq:S_i_general}, 
for all $t\in [0,T]$ and $i\in { [N]}$, 
\begin{align}
\label{eq:S_i_bound}
\begin{split}
    |S_{t,i}|
    &\le 
    \int_0^t
    \left( 
  \|B\|_{\infty} |S_{u,i}|+
  \|\bar B\|_{ \infty} \sum_{k=1}^N 
  | S_{u,k}|
    +  | f_{u,i} 
     |
    \right)
    \d u.
\end{split}    
\end{align}
  Summarizing \eqref{eq:S_i_bound}  over the index  $i\in { [N]}$ yields for all $t\in [0,T]$,
\begin{align*} 
  \sum_{k=1}^N   |S_{t,k}|
    &\le 
    \int_0^t \left( (\|B\|_{\infty}+ N\|\bar B\|_{\infty}) \sum_{k=1}^N  \left| S_{u,k}\right| +   \sum_{k=1}^N \left|     f_{u,k}  \right|     \right)  \d u,
\end{align*}
      which along with Gronwall's inequality implies that
 \begin{align*} 
   \sum_{k=1}^N   | S_{t,k}|
    &\le 
     \left(\int_0^T  \sum_{k=1}^N
      \left|  f_{u,i}
    \right|  
    \d u\right)e^{  (\|B\|_{\infty}+N \|\bar B\|_{\infty}) t}.
\end{align*}
  Substituting the above inequality into  \eqref{eq:S_i_bound}
    yields for all $t>0$,
    \begin{align*}
\begin{split}
    |S_{t,i}|
    & \le 
    \int_0^t
  \|B\|_{\infty} |S_{u,i}|\d u 
  +
  \left( \left(\int_0^T   \sum_{k=1}^N
      \left|  f_{u,i}
    \right|  
    \d u\right) \int_0^T   \|\bar B\|_{\infty}
 e^{  (\|B\|_{\infty}+ N\|\bar B\|_{\infty}) u}\d u
    \right) + \int_0^T | f_{u,i} | \d u.
\end{split}    
\end{align*}
This with    Gronwall's inequality shows that for all $t>0$,
    \begin{align*} 
    |S_{t,i}|
    &\le 
  \left( \int_0^T   
  | f_{u,i}|\d u   
     + 
 \left( \int_0^T  \sum_{k=1}^N
      \left|  f_{u,k}
    \right|   
    \d u  \right)
    \|\bar B\|_{\infty}
 T e^{  (\|B\|_{\infty}+N \|\bar B\|_{\infty}) T}
  \right)e^{\|B\|_\infty T}.
\end{align*}
Taking the $p$-th moments of 
 both sides of the above inequality and using the   fact that 
 $(a+b)^p\le 2^{p-1}(a^p+b^p)$ for all $a,b\ge 0$
 yield 
\begin{align*} 
   \sE[  |S_{t,i}|^p]
    &\le 
  \sE\left[ 
   \left(    \int_0^T   
  | f_{u,i}|\d u   
     + 
 \left( \int_0^T  \sum_{k=1}^N
      \left|  f_{u,k}
    \right|   
    \d u  \right)
    \|\bar B\|_{\infty}
 T e^{  (\|B\|_{\infty}+N \|\bar B\|_{\infty}) T}
  \right)^p
  \right] e^{p\|B\|_\infty T}
  \\
   &\le 
2^{p-1}  \sE\left[ 
   \left(    \int_0^T   
  | f_{u,i}|\d u   \right)^p
     + 
  \left( \int_0^T  \sum_{k=1}^N
      \left|  f_{u,k}
    \right|   
    \d u  \right)^p
    \|\bar B\|^p_{\infty}
 T^p e^{ p (\|B\|_{\infty}+N \|\bar B\|_{\infty}) T}
   \right] e^{p\|B\|_\infty T}
   \\
  &\le 
(2 T)^{p-1}   \left( 
   \left\|   
    f_{ i}    \right\|^p_{\cH^p(\sR)}
     + 
  \left\|  \sum_{k=1}^N
 |  f_{k}
   |   
      \right\|^p_{\cH^p(\sR)}
    \|\bar B\|^p_{\infty}
 T^p e^{ p (\|B\|_{\infty}+N \|\bar B\|_{\infty}) T}
   \right) e^{p\|B\|_\infty T}.
\end{align*}
This proves the desired estimate. 
\end{proof}

\begin{proof}[Proof of Proposition \ref{prop:Y_h_i_moment_open_loop}]
    To simplify the notation,
we write $\bX=\bX^\bu$ and $\bY^h=\bY^{\bu,u'_h}$.  
  Applying Lemma \ref{lemma:S_i_general} with 
  $\textbf{S} = \bY^h$, 
  $B_{i}(t) =\partial_x b_i(t, X_{t,i}, \bX_{t})$,
  $\bar{B}_{i,j}(t)=\partial_{y_j} b_i(t, X_{t,i}, \bX_{t})$
  and 
  $f_{t,i} = \delta_{h,i} u'_{t,h}$ yields that 
    for all  $i\in { [N]}$,
 \begin{align}
 \label{eq:Y_moment_estimate1_open}
 \begin{split}
\sup_{t\in [0,T]}  \sE[   |Y^h_{t,i}|^p]
 &\le  (2 T)^{p-1}  \left( 
      \left\|    
  f_{ i}  \right\|^p_{\cH^p(\sR)}
     + 
 \left\|   \sum_{k=1}^N
   \left|  f_{k}
    \right|   
       \right\|^p_{\cH^p(\sR)}
    \frac{ ( L^{b}_y)^p}{N^p} 
 T^p e^{  p(      L^{b} +      L^{b}_y) T}
  \right) e^{p L^{b} T}\\
  &\le (2 T)^{p-1}  \left( 
     \delta_{h,i} 
     + 
    \frac{ ( L^{b}_y)^p}{N^p} 
 T^p e^{  p(      L^{b} +      L^{b}_y) T}
  \right)  \| u_h'\|^p_{\mathcal{H}^p (\mathbb{R})} e^{p L^{b} T}.
  \end{split}
 \end{align}
 where we used $\|\bar{B}_{i,j}\|_{L^\infty}\le L^{b}/N$.
\end{proof}

Finally, to prove Proposition \ref{prop:Z_h_ell_i_moment_open},
we estimate the moment of the process
$\mathfrak{f}^{\bu,u'_h,u''_\ell}_{i}$ defined in \eqref{eq:f_phi_h_ell_open_loop}.

\begin{lemma}\label{lemma: bound_f_open_loop}
Suppose Assumption \ref{assum:regularity} holds.
For all 
 $ \bu\in \cH^2(\sR^N)  $, $i, h,\ell \in { [N]}$ with $h\not =\ell$,
 and   all 
$ u'_h,  
 u''_\ell\in  \cH^{4}(\sR)$,
the process $\mathfrak{f}^{\bu,u'_h,u''_\ell}_{i}$ defined in \eqref{eq:f_phi_h_ell_open_loop}
satisfies 
 \begin{align*}
\| \mathfrak{f}^{\bu,u'_h,u''_\ell}_{i}\|_{\cH^2(\sR)} 
\le C \left((\delta_{h,i}  +\delta_{\ell,i} )\frac{1 }{N}+\frac{1}{N^2}\right)L^{b}_y
      \|u_h' \|_{\cH^4(\sR)} \| u_\ell''\|_{\cH^4(\sR)} , 
 \end{align*}
 where   $C\ge 0$ is a    constant   
 depending only on   the upper bounds of      $T$, 
$L^b$ and $L_y^b$.
\end{lemma}
\begin{proof}
Fix $h,\ell \in { [N]}$ with $h\ne \ell$.
    To simplify the notation,
we write $\bX=\bX^\bu$, $\bY^h=\bY^{\bu, u'_h}$ and
$\bY^\ell=\bY^{\bu,u''_\ell}$. 
Observe that  by Proposition \ref{prop:Y_h_i_moment_open_loop}, 
 for all  $i,j\in { [N]}$,
\begin{align}
\label{eq:Y_product_open_loop}
\begin{split}
&\left\|   
Y^{h}_{i}Y^{\ell}_{j}   \right\|^2_{\cH^2(\sR)}
\le T \sup_{t\in [0,T]}
\sE [   | Y^{h}_{t, i}Y^{\ell}_{t, j} |^2] 
\le T \sup_{t\in [0,T]}
\sE [   | Y^{h}_{t, i}|^4]^{\frac{1}{2}} \sE[|Y^{\ell}_{t, j} |^4]^{\frac{1}{2}}
\\
&\le 
T \| u_h'\|_{\cH^4(\sR)}^2 \| u_\ell''\|_{\cH^4(\sR)}^2
\left( \delta_{h,i} C^{h,4}_Y
 +\frac{ ( L^{b}_y)^4}{N^4}\bar{C}^{h,4}_Y \right)^{\frac{1}{2}} 
  \left( \delta_{\ell,j} C^{\ell,4}_Y
 +\frac{( L^{b}_y)^4}{N^4}\bar{C}^{\ell,4}_Y \right)^{\frac{1}{2}} 
\\
&\le 
T \| u_h'\|_{\cH^4(\sR)}^2 \| u_\ell''\|_{\cH^4(\sR)}^2 \\
&\quad\times\bigg( \delta_{h,i} \delta_{\ell,j} (C^{h,4}_Y C^{\ell,4}_Y)^{\frac{1}{2}} 
+\frac{( L^{b}_y)^2}{N^2} \left(\delta_{h,i} (C^{h,4}_Y \bar{C}^{\ell,4}_Y)^{\frac{1}{2}} 
+\delta_{\ell,j} (C^{\ell,4}_Y \bar{C}^{h,4}_Y)^{\frac{1}{2}} \right)
  +\frac{( L^{b}_y)^4}{N^4}(\bar{C}^{h,4}_Y \bar{C}^{\ell,4}_Y)^{\frac{1}{2}} 
   \bigg)\\
&\leq
C  \| u_h'\|_{\cH^4(\sR)}^2 \| u_\ell''\|_{\cH^4(\sR)}^2 \bigg( \delta_{h,i} \delta_{\ell,j} 
+\frac{( L^{b}_y)^2}{N^2} \left(\delta_{h,i} 
+\delta_{\ell,j}  \right)
  +\frac{( L^{b}_y)^4}{N^4} 
   \bigg),
  \end{split}
\end{align}
where the third line follows by noting 
$
\sqrt{a_1 + \cdots + a_N} \leq \sqrt{a_1} + \cdots +\sqrt{a_N}
$
for any $a_1, \cdots, a_N \ge 0.$

 We now bound each term in \eqref{eq:f_phi_h_ell_open_loop}.
Observe that by \eqref{eq:f_phi_h_ell_open_loop}, for all $t\in [0,T]$, 
\begin{align}
\label{eq:f_phi_h_ell_open_loop_vector}
\begin{split}
\mathfrak{f}^{\bu,u'_h,u''_\ell}_{t,i}&
=  (\partial^2_{xx} b_i)(t, X_{t,i}, \bX_{t}) Y^{h}_{t,i}Y^{\ell}_{t,i}
 +  \sum_{j=1}^N (\partial^2_{x y_j} b_i)(t, X_{t,i}, \bX_{t}) 
(
Y^{ h}_{t,i}Y^{\ell}_{t,j}+Y^{\ell}_{t,i}Y^{h}_{t,j})
+\sum_{j,k=1}^N (\partial^2_{y_j y_k} b_i)(t, X_{t,i}, \bX_{t}) 
Y^{h}_{t,j}Y^{\ell}_{t,k}.
\end{split}
\end{align}
Apply \eqref{eq:Y_product_open_loop} with the fact that $\delta_{h,i}\delta_{\ell,i}=0$ as $h\not =\ell$ to get
 \begin{align}
  \label{eq:b_xx_open} 
 \begin{split}
 &\|  (\partial^2_{xx} b_i)(\cdot, X_{\cdot, i}, \bX_\cdot ) Y^{h}_{i}Y^{\ell}_{i}\|_{\cH^2(\sR)}
\le L^{b} \|   Y^{h}_{i}Y^{\ell}_{i}\|^2_{\cH^2(\sR)}
\le  C  \| u_h'\|_{\cH^4(\sR)} \| u_\ell''\|_{\cH^4(\sR)} \left((\delta_{h,i}  +\delta_{\ell,i} )\frac{L_y^b}{N}+\frac{( L^{b}_y)^2}{N^2}\right),
 \end{split}
 \end{align}
 where $C$ is a constant depending on $T$, $C_Y^{h,4}$ and $\bar{C}_Y^{h,4}$ for any $h \in { [N]}$.

We then estimate 
 $\sum_{j=1}^N (\partial^2_{x y_j} b_i)(\cdot, X_{\cdot, i}, \bX_{\cdot}) 
(
Y^{h}_{ i}Y^{ \ell}_{j}+Y^{ \ell}_{ i}Y^{ h}_{j}) $ in \eqref{eq:f_phi_h_ell_open_loop_vector}.
  The fact that   $ \partial^2_{x y_j} b_i $ is bounded by $L_{y} ^{b}/N$ and the inequality that 
 $(\sum_{k=1}^N a_k)^2\le N \sum_{k=1}^N a^2_k$ for all $a_1, a_2, \cdots, a_N \in  [0,\infty)$ show that 
  \begin{align*}
&  \left\|\sum_{j=1}^N (\partial^2_{x y_j} b_i)(\cdot, X_{\cdot, i}, \bX_{\cdot}) 
(
Y^{h}_{ i}Y^{ \ell}_{j}+Y^{ \ell}_{ i}Y^{ h}_{j}) \right\|^2_{\cH^2(\sR)}
\le \frac{(L^{b}_y)^2 }{N^2 }\left\|\sum_{j=1}^N (| 
Y^{h}_{ i}Y^{ \ell}_{j}|+|Y^{ \ell}_{ i}Y^{ h}_{j}|) \right\|^2_{\cH^2(\sR)}
\notag\\
&
=
  \frac{(L^{b}_y)^2 }{N^2 }\left\|
  | Y^{h}_{ i}Y^{ \ell}_{\ell}|
  +|Y^{ \ell}_{ i}Y^{ h}_{h}|
+  \sum_{j\not =\ell} | 
Y^{h}_{ i}Y^{ \ell}_{j}|
+ \sum_{j\not =h} |Y^{ \ell}_{ i}Y^{ h}_{j}| \right\|^2_{\cH^2(\sR)}
\notag\\
&\le 4 \frac{(L^{b}_y)^2 }{N^2 }\left(
\| Y^{h}_{ i}Y^{ \ell}_{\ell}\|^2_{\cH^2(\sR)}
  +\|Y^{ \ell}_{ i}Y^{ h}_{h}\|^2_{\cH^2(\sR)}
+ \bigg \|\sum_{j\not =\ell} | 
Y^{h}_{ i}Y^{ \ell}_{j}|\bigg\|^2_{\cH^2(\sR)}
+ \bigg\|\sum_{j\not =h} |Y^{ \ell}_{ i}Y^{ h}_{j}|  \bigg\|^2_{\cH^2(\sR)}\right)
\notag\\
&\le 4 \frac{(L^{b}_y)^2 }{N^2 }\left(
\| Y^{h}_{ i}Y^{ \ell}_{\ell}\|^2_{\cH^2(\sR)}
  +\|Y^{ \ell}_{ i}Y^{ h}_{h}\|^2_{\cH^2(\sR)}
+(N-1) \left( \sum_{j\not =\ell} \| 
Y^{h}_{ i}Y^{ \ell}_{j}\|^2_{\cH^2(\sR)}
+ \sum_{j\not =h} \|Y^{ \ell}_{ i}Y^{ h}_{j}\|^2_{\cH^2(\sR)} \right) \right),
 \end{align*}
  which along with \eqref{eq:Y_product_open_loop} yields
\begin{align}
\label{eq:b_xy_open}
 &\left\|\sum_{j=1}^N (\partial^2_{x y_j} b_i)(\cdot, X_{\cdot, i}, \bX_{\cdot}) 
(
Y^{h}_{ i}Y^{ \ell}_{j}+Y^{ \ell}_{ i}Y^{ h}_{j}) \right\|^2_{\cH^2(\sR)}
\le   \frac{4 (L^{b}_y)^2 }{N^2 } C  \| u_h'\|_{\cH^4(\sR)}^2 \| u_\ell''\|_{\cH^4(\sR)}^2 \bigg[
 \bigg( \delta_{h,i} 
+\frac{(L^{b}_y)^2}{N^2} \left(\delta_{h,i}  
+ 1\right)
  +\frac{(L^{b}_y)^4}{N^4}
   \bigg)
   \notag\\
   &\quad 
    +
    \bigg(  \delta_{\ell,i} 
+\frac{(L^{b}_y)^2}{N^2} \left(  1
+\delta_{\ell,i}  \right)
  +\frac{(L^{b}_y)^4}{N^4}
   \bigg)
 +\frac{ N-1}{N^2}      \bigg( 
   \delta_{h,i} 
  +\frac{(L^{b}_y)^2}{N^2}
+
\delta_{\ell,i} 
  +\frac{(L^{b}_y)^2}{N^2}
   \bigg)
   \bigg],
   \notag\\
    &\le C   \| u_h'\|_{\cH^4(\sR)}^2 \| u_\ell''\|_{\cH^4(\sR)}^2  \left(
(  \delta_{h,i} + \delta_{\ell,i})  \frac{(L_y^b)^2 }{N^2 } +\frac{(L_y^b)^4}{N^4}
\right).
\end{align}

Finally, we  estimate $\sum_{j,k=1}^N(\partial^2_{y_j y_k} b_i)(t, X_{i}, \bX) 
Y^{h}_{j}Y^{\ell}_{k}
$ in \eqref{eq:f_phi_h_ell_open_loop_vector}. 
We write for simplicity $(\partial^2_{y_j y_k} b_i)(\cdot) =(\partial^2_{y_j y_k} b_i)(\cdot, X_{\cdot, i}, \bX_\cdot) $  for all $j,k\in { [N]}$. 
Since $h\not =\ell$, 
${ [N]}\times { [N]} =\{(h,\ell)\} \cup   \{(h,h)\}\cup
\{(\ell,\ell)\}\cup  
\big\{(h,k)\mid k \in { [N]}\setminus\{h,\ell\}\big\}
\cup \big\{(j,\ell)\mid j \in { [N]}\setminus\{h,\ell \}\big\}
\cup
\big\{(j,k)\mid 
j\in   { [N]}\setminus\{h \}, k\in { [N]}\setminus \{\ell\}\big\}$, and hence
\begin{align}
\label{eq:sum_k_Y_ell_h_decompose_open_loop}
\begin{split}
 &  \sum_{j,k=1}^N (\partial^2_{y_j y_k} b_i)(\cdot) 
Y^{h}_{j}Y^{\ell}_{k}= 
(\partial^2_{y_h y_\ell } b_i)(\cdot) Y^{h}_{h}Y^{\ell}_{\ell}
+
 (\partial^2_{y_h y_h } b_i)(\cdot) Y^{h}_{h}Y^{\ell}_{h} + (\partial^2_{y_\ell y_\ell} b_i)(\cdot) 
Y^{h}_{\ell}Y^{\ell}_{\ell}
\\
&\quad +
\sum_{k\in { [N]}\setminus\{ \ell,h\} }(\partial^2_{y_h y_k }b_i)(\cdot) Y^{h}_{h}Y^{\ell}_{k}
+
\sum_{j\in { [N]}\setminus\{ h,\ell\}}  (\partial^2_{y_jy_\ell} b_i)(\cdot) 
Y^{h}_{j}Y^{\ell}_{\ell}+
\sum_{j\in { [N]}\setminus\{ h\}} \sum_{k\in { [N]}\setminus\{ \ell\}}  (\partial^2_{y_jy_k} b_i)(\cdot) 
Y^{h}_{j}Y^{\ell}_{k}.
 \end{split}
\end{align}
To analyze the first line in \eqref{eq:sum_k_Y_ell_h_decompose_open_loop},
note that by \eqref{eq:Y_product_open_loop} and $h\not =\ell$, 
\begin{align}
\label{eq:sum_k_Y_ell_h_decompose_term13_open_loop}
\begin{split}
\| (\partial^2_{y_h y_\ell } b_i)(\cdot) Y^{h}_{h}Y^{\ell}_{\ell}\|_{\cH^2(\sR)}
&\le 
\frac{L^{b}_y}{N^2} \|  Y^{h}_{h}Y^{\ell}_{\ell}\|_{\cH^2(\sR)}
\le C  \|u_h' \|_{\cH^4(\sR)} \| u_\ell''\|_{\cH^4(\sR)} \frac{L^{b}_y}{N^2},
\\
\| (\partial^2_{y_h y_h } b_i)(\cdot) Y^{h}_{h}Y^{\ell}_{h}\|_{\cH^2(\sR)}
& 
\le \frac{L^{b}_y}{N} \|   Y^{h}_{h}Y^{\ell}_{h}\|_{\cH^2(\sR)}
\le C  \|u_h' \|_{\cH^4(\sR)} \| u_\ell''\|_{\cH^4(\sR)} \frac{(L^{b}_y)^2}{N^2},
\\
\|(\partial^2_{y_\ell y_\ell} b_i)(\cdot) 
Y^{h}_{\ell}Y^{\ell}_{\ell}\|_{\cH^2(\sR)}
&\le
  \frac{L^{b}_y}{N} \|  Y^{h}_{\ell}Y^{\ell}_{\ell} \|_{\cH^2(\sR)}
\le C  \|u_h' \|_{\cH^4(\sR)} \| u_\ell''\|_{\cH^4(\sR)} \frac{(L^{b}_y)^2}{N^2}.
\end{split}
\end{align}
Moreover, to analyze the first two terms in the second line of \eqref{eq:sum_k_Y_ell_h_decompose_open_loop}, by \eqref{eq:Y_product_open_loop},
\begin{align}
\label{eq:sum_k_Y_ell_h_decompose_term4_open}
\begin{split}
&\left\|  \sum_{k\in { [N]}\setminus\{ \ell,h\} }(\partial^2_{y_h y_k } b_i)(\cdot) Y^{h}_{h}Y^{\ell}_{k}\right\|_{\cH^2(\sR)}
+\left\| \sum_{j\in { [N]}\setminus\{ h,\ell\}}  (\partial^2_{y_j y_\ell} b_i)(\cdot) 
Y^{h}_{j}Y^{\ell}_{\ell} \right\|_{\cH^2(\sR)}
\\
& \le   \frac{L^{b}_y}{N^2} \sum_{k\in { [N]}\setminus\{ \ell,h\} }  \left\|  Y^{h}_{h}Y^{\ell}_{k}\right\|_{\cH^2(\sR)}
+ \frac{L^{b}_y}{N^2} \sum_{j\in { [N]}\setminus\{ h,\ell\}}   \left\|  Y^{h}_{j}Y^{\ell}_{\ell} \right\|_{\cH^2(\sR)}
 \le C  \| u_h'\|_{\cH^4(\sR)} \| u_\ell''\|_{\cH^4(\sR)}  \frac{L^{b}_y}{N^2} .
 \end{split}
\end{align}
Furthermore, to analyze the last term in \eqref{eq:sum_k_Y_ell_h_decompose_open_loop}, as $h\not =\ell$, 
\begin{align}
\label{eq:sum_k_Y_ell_h_decompose_term5_open}
\begin{split}
 & \sum_{j\in { [N]}\setminus\{ h\}} \sum_{k\in { [N]}\setminus\{ \ell\}}  (\partial^2_{y_j y_k} b_i)(\cdot) 
Y^{h}_{j}Y^{\ell}_{k}  
\\
&  
= 
  \sum_{k\in { [N]}\setminus\{ \ell\}}  (\partial^2_{y_\ell y_k} b_i)(\cdot) 
Y^{h}_{\ell }Y^{\ell}_{k} 
   + \sum_{j\in { [N]}\setminus\{ h,\ell\}} 
\bigg( (\partial^2_{y_j y_j} b_i)(\cdot) 
Y^{h}_{j}Y^{\ell}_{j}  
 + \sum_{k\in { [N]}\setminus\{ \ell,j\}}  (\partial^2_{y_j y_k} b_i)(\cdot) 
Y^{h}_{j}Y^{\ell}_{k}  
 \bigg),
\end{split}
\end{align}
where the first and second  terms can be estimated by 
\begin{align}
\begin{split}
&\left\| \sum_{k\in { [N]}\setminus\{ \ell\}}  (\partial^2_{y_\ell y_k} b_i)(\cdot) 
Y^{h}_{\ell }Y^{\ell}_{k} \right\|_{\cH^2(\sR)}
+ \left\| \sum_{j\in { [N]}\setminus\{ h,\ell\}} 
  (\partial^2_{y_jy_j} b_i)(\cdot) 
Y^{h}_{j}Y^{\ell}_{j}  \right\|_{\cH^2(\sR)}
\\
&
\le \sum_{k\in { [N]}\setminus\{ \ell\}}   \frac{L^{b}_y}{N^{ 2 }}\left\|  
Y^{h}_{\ell }Y^{\ell}_{k} \right\|_{\cH^2(\sR)}
+ \sum_{j\in { [N]}\setminus\{ h,\ell\}}  \frac{L^{b}_y}{N}  \left\| 
Y^{h}_{j}Y^{\ell}_{j}  \right\|_{\cH^2(\sR)}
\\
&\le
C\left( N  \frac{(L^{b}_y)^3}{N^{ 4 }} + N \frac{(L^{b}_y)^3}{N^{ 3}} \right)\| u_h'\|_{\cH^4(\sR)} \| u_\ell''\|_{\cH^4(\sR)}
\le
C \frac{(L^{b}_y)^2}{N^2} \| u_h'\|_{\cH^4(\sR)} \| u_\ell''\|_{\cH^4(\sR)},
\end{split}
\end{align}
and the last term can be estimated by
\begin{align}
\begin{split}
&\left\|  \sum_{j\in { [N]}\setminus\{ h,\ell\}}  \sum_{k\in { [N]}\setminus\{ \ell,j\}}  (\partial^2_{y_j y_k} b_i)(\cdot) 
Y^{h}_{j}Y^{\ell}_{k}  
  \right\|_{\cH^2(\sR)}
  \le  \sum_{j\in { [N]}\setminus\{ h,\ell\}}  \sum_{k\in { [N]}\setminus\{ \ell,j\}}  \frac{L^{b}_y}{N^2} \left\|   
Y^{h}_{j}Y^{\ell}_{k}  
  \right\|_{\cH^2(\sR)}
\\
  &\le N^2  \frac{L^{b}_y}{N^2}  \frac{C{(L^{b}_y)^2}\| u_h'\|_{\cH^4(\sR)} \| u_\ell''\|_{\cH^4(\sR)}  }{N^2} 
  = C\frac{(L^{b}_y)^2}{N^2}\| u_h'\|_{\cH^4(\sR)} \| u_\ell''\|_{\cH^4(\sR)}  .
\end{split}
\end{align}
Hence combining \eqref{eq:sum_k_Y_ell_h_decompose_open_loop}, \eqref{eq:sum_k_Y_ell_h_decompose_term13_open_loop}, 
\eqref{eq:sum_k_Y_ell_h_decompose_term4_open} and \eqref{eq:sum_k_Y_ell_h_decompose_term5_open} yields
$\Big\|\sum_{j,k=1}^N (\partial^2_{y_j y_k} b_i)(\cdot,  X_{\cdot, i}, \bX_\cdot) 
Y^{h}_{j}Y^{\ell}_{k}
\Big\|_{\cH^2(\sR)}
\le C \| u_h'\|_{\cH^4(\sR)}\cdot \| u_\ell''\|_{\cH^4(\sR)}  \frac{L^{b}_{y} }{N^2}.$
 This along with
\eqref{eq:b_xx_open}
and \eqref{eq:b_xy_open} yield the desired estimate. 
\end{proof}

\begin{proof}[Proof of Proposition \ref{prop:Z_h_ell_i_moment_open}]
    To simplify the notation,
we write $\bX=\bX^\bu$, $\bY^h=\bY^{\bu,u'_h}$,
$\bY^\ell=\bY^{\bu,u''_\ell}$,   
$\bZ^{h, \ell}=\bZ^{\bu,u'_h,u''_\ell}$
and $\mathfrak{f}^{ h,\ell} = \mathfrak{f}^{\bu,u'_h,u''_\ell}$.
Applying Lemma \ref{lemma:S_i_general} with 
    $\textbf{S} = \bZ^{\ell, h} $, 
  $B_{i}(t) =\partial_x b_i(t, X_{t,i}, \bX_{t})$,
  $\bar{B}_{i,j}(t)=\partial_{y_j} b_i(t, X_{t,i}, \bX_{t})$
  and 
  $f_{t,i} =\mathfrak{f}^{ h,\ell}_{t,i}$ 
   yields that 
    for all  $i\in { [N]}$,
    \begin{align}
 \label{eq:Z_moment_estimate1_open}
 \begin{split}
        \sup_{t\in[0,T]} \E [|Z^{h,\ell}_{t,i}|^2]  
 &\le  2T   \left( 
  \left\|   
  \mathfrak{f}^{ h,\ell}_{ i}    \right\|^2_{\cH^2(\sR)}
     + 
       \left\|  \sum_{k=1}^N
 |   \mathfrak{f}^{ h,\ell}_{k}
   |   
      \right\|^2_{\cH^2(\sR)}
          \frac{ ( L^{b}_y)^2}{N^2}
 T^2 e^{2(L^{b} + L^{b}_y) T}
  \right) e^{2 L^{b} T}.
  \end{split}
    \end{align}
By Lemma \ref{lemma: bound_f_open_loop}, one can get
$ |\mathfrak{f}^{h,\ell}_{i} \|_{\mathcal{H}^2(\mathbb{R})}
    \leq  C \| u_h'\|_{\cH^4(\sR)} \| u_\ell''\|_{\cH^4(\sR)}    L_y^b\left( (\delta_{h,i}+\delta_{\ell,i})\frac{1}{N} + \frac{1}{N^2}\right), $
where $C \geq 0$ is a constant,  
 which depends  on the upper bounds of   $T$, $L^b, L_y^b$.
 Moreover,
 \begin{align*}
    \left\|  \sum_{k=1}^N
 |   \mathfrak{f}^{ h,\ell}_{k}|  \right\|_{\cH^2(\sR)}
  &    \le    \sum_{k=1}^N \left\|   
  \mathfrak{f}^{ h,\ell}_{k}  \right\|_{\cH^2(\sR)} 
  =  \sum_{k\in  \{h,\ell\}}    \left\|   
  \mathfrak{f}^{ h,\ell}_{k}  \right\|_{\cH^2(\sR)} 
  + \sum_{k\in { [N]}\setminus\{h,\ell\}}  \left\|   
  \mathfrak{f}^{ h,\ell}_{k}  \right\|_{\cH^2(\sR)} 
  \\
  &\le C 
  \left(\frac{1}{N}+ (N-2)\frac{1}{N^2}\right)  \| u_h'\|_{\cH^4(\sR)} \| u_\ell''\|_{\cH^4(\sR)} L^{b}_y
  \le \frac{C \| u_h'\|_{\cH^4(\sR)} \| u_\ell''\|_{\cH^4(\sR)} L^{b}_y }{N}.
\end{align*}
Summarizing the above estimates yields the desired conclusion. 
\end{proof}

\section{Conclusion}

\appendix

\bibliographystyle{siam}
\bibliography{references}

\end{document}